\documentclass[10pt,a4paper]{article}
\usepackage[utf8]{inputenc}
\usepackage[T1]{fontenc}
\usepackage{amsmath}
\usepackage{amsthm}
\usepackage{amsfonts}
\usepackage{amssymb}
\usepackage{graphicx}
\usepackage[inner=25mm,outer=25mm,top=30mm,bottom=30mm]{geometry}
\usepackage{times}

\usepackage{enumitem}
\usepackage{frenchineq}

\usepackage{authblk}

\usepackage{caption}
\usepackage{subcaption}
\usepackage{float}

\usepackage{xcolor}
\definecolor{color1}{rgb}{0.85,0.15,0.5}
\definecolor{color2}{rgb}{1,0.7,0}
\definecolor{color3}{rgb}{0.4,0.6,1}
\definecolor{color4}{rgb}{0,0,0}
\usepackage[scr=boondox,scrscaled=1.05]{mathalfa}
\usepackage[outline]{contour}

\usepackage{url}
\usepackage{graphicx}
\usepackage{pgf,tikz,tikz-network}
\usetikzlibrary{arrows,calc,patterns,patterns.meta,arrows.meta}
\usepackage{cancel}
\usepackage[colorinlistoftodos,bordercolor=orange,backgroundcolor=orange!20,linecolor=orange,textsize=scriptsize]{todonotes}

\usepackage[linesnumbered,ruled,vlined,nofillcomment]{algorithm2e}
\SetKwInput{KwInput}{input}
\SetKwInput{KwOutput}{output}

\SetCommentSty{comments}

\usepackage{hyperref}
\usepackage{cleveref}

\usepackage{pifont}
\usepackage{blkarray}

\usepackage{stackengine,scalerel} 
\usepackage{multicol}

\newcommand{\N}{\mathbb{N}}
\newcommand{\Npos}{\mathbb{N}_{> 0}}
\newcommand{\Z}{\mathbb{Z}}
\newcommand{\Q}{\mathbb{Q}}
\newcommand{\R}{\mathbb{R}}
\newcommand{\Rnonneg}{\mathbb{R}_{\geq 0}}

\newcommand{\C}{\mathbb{C}}

\newcommand{\T}{\mathbb{T}}

\newcommand{\tplus}{\oplus}  
\newcommand{\tsum}{\bigoplus}
\newcommand{\tdot}{\odot}

\newcommand{\bal}{\mathrel{\nabla}}
\newcommand{\nbal}{\mathrel{\cancel{\nabla}}}
\DeclareMathAlphabet{\mathbbold}{U}{BOONDOX-ds}{m}{n}
\newcommand{\zero}{{\mathbbold{0}}}
\newcommand{\unit}{\mathbbold{1}}
\newcommand{\NP}{\mathrm{NP}}

\def\<#1,#2>{\langle #1,#2\rangle}

\newcommand{\aff}{\operatorname{aff}}
\DeclareMathOperator*{\argmax}{arg\,max}

\newcommand{\cl}{\operatorname{cl}}

\newcommand{\conv}{\operatorname{conv}}

\newcommand{\hypo}{\operatorname{hypo}}

\newcommand{\relint}{\operatorname{relint}}

\newcommand{\supp}{\operatorname{supp}}
\newcommand{\tdet}{\operatorname{tdet}}
\newcommand{\trop}{\operatorname{trop}}
\newcommand{\val}{\operatorname{val}}

\newcommand{\ver}{\operatorname{ver}}

\newcommand{\hahnseries}[2]{#1[[t^{#2}]]}
\newcommand{\vfield}{\mathbf{K}}
\newcommand{\vgroup}{\Gamma}

\newcommand{\calA}{\mathcal{A}}
\newcommand{\calE}{\mathcal{E}}

\newcommand{\scrF}{\mathscr{F}}
\newcommand{\scrN}{\mathscr{N}}

\newcommand{\mac}{\mathcal{M}}
\newcommand{\var}{\mathcal{V}}

\newcommand{\tvar}{\mathcal{V}_\mathsf{trop}}

\newcommand{\tropi}{\mathcal{T}}

\newcommand{\cell}{\mathcal{C}}

\newcommand{\face}{\mathcal{F}}

\newcommand{\longto}{\longrightarrow}
\newcommand{\nrhd}{\ntriangleright}
\newcommand{\puis}{\boldsymbol}
\newcommand{\supco}{\mathbin{\scriptstyle \square}}

\newtheorem{thm}{Theorem}[section]
\crefname{thm}{Theorem}{Theorems}
\newtheorem{prp}[thm]{Proposition}
\crefname{prp}{proposition}{Propositions}
\newtheorem{lem}[thm]{Lemma}
\crefname{lem}{Lemma}{Lemmas}
\newtheorem{cor}[thm]{Corollary}
\crefname{cor}{Corollary}{Corollaries}

\theoremstyle{definition}
\newtheorem{dfn}[thm]{Definition}
\crefname{dfn}{Definition}{Definitions}
\newtheorem{obs}[thm]{Observation}
\crefname{obs}{Observation}{Observations}

\crefname{asm}{Assumption}{Assumptions}

\theoremstyle{remark}
\newtheorem{rmk}[thm]{Remark}
\crefname{rmk}{Remark}{Remarks}
\newtheorem{expl}[thm]{Example}
\crefname{expl}{Example}{Examples}
\newtheorem{prt}[thm]{Property}
\crefname{prt}{Property}{Properties}

\newenvironment{info}{%
 \small \centering \textbf{Article info} \par \flushleft}{}

\author[1,2]{Marianne Akian}
\author[2,1]{Antoine Béreau}
\author[1,2]{Stéphane Gaubert}
\affil[1]{\normalsize Inria, 91120 Palaiseau, France}
\affil[2]{\normalsize CMAP, CNRS, \'Ecole polytechnique, Institut Polytechnique de Paris, 91120 Palaiseau, France}

\title{The Nullstellensatz and Positivstellensatz for Sparse Tropical Polynomial Systems}
\date{}

\begin{document}
\maketitle

\begin{abstract}
Grigoriev and Podolskii (2018) have established a tropical analogue of the effective Nullstellensatz, showing that a system of tropical polynomial equations is solvable if and only if a linearized system obtained from a truncated Macaulay matrix is solvable. They provided an upper bound of the minimal admissible truncation degree,
as a function of the degrees of the tropical polynomials.
We establish a tropical Nullstellensatz adapted to {\em sparse} tropical polynomial systems. Our approach is inspired by a construction of Canny-Emiris (1993), refined by Sturmfels (1994). This leads to an improved bound of the truncation degree, which coincides with the classical Macaulay degree in the case of $n+1$ equations in $n$ unknowns. We also establish a tropical Positivstellensatz, allowing one to decide the inclusion of tropical basic semialgebraic sets. This allows one to reduce decision problems for tropical semi-algebraic sets to the solution of systems of tropical linear equalities and inequalities.
\end{abstract}

\vspace*{\baselineskip}

\begin{info}
\begin{center}
\begin{tabular}{ll}
\textbf{Keywords:} & Tropical geometry, Polynomial systems, Newton polytope, Tropical hypersurfaces,\\
& Combinatorics of convex polytopes\\[3pt]
\textbf{MSC classes:} & 14T15 (Primary) 05E14, 14N10, 52B20 (Secondary)
\end{tabular}
\end{center}
\end{info}


\section{Introduction}

\subsubsection*{\bf Motivation}
A basic problem in computational tropical geometry consists in solving systems of equations or of inequations involving multivariate tropical polynomials.

Such systems arise in the study of {\em non-archimedean amoebas}, which are images by a non-archimedean valuation of an algebraic set~\cite{EKL05}. A super-approximation of a non-archimedean amoeba  can be obtained by translating the defining equations of the algebraic set in the tropical semifield. This leads to a system of tropical polynomials equations, the solution set of which is called 
a {\em tropical prevariety}. Moreover, when working over an algebraically closed non-archimedean field, this approximation is exact under appropriate genericity conditions, or if the defining set is sufficiently rich, see in particular the `fundamental theorem of tropical geometry' in~\cite{MS15}, in which case the non-archimedean amoeba is called a  {\em tropical variety}.

Similarly, the solution sets of systems of tropical polynomial inequalities provide upper approximations of images by a non-archimedean valuation of semi-algebraic sets over real closed non-archimedean fields, and these approximations are exact under genericity conditions~\cite{itenberg_viro,AGS20,Jell2020}. Moreover, owing to their combinatorial nature, tropical polynomial systems are often easier to grasp than their classical analogues. These ideas, which can be traced back to works of Viro (see~\cite{viro}), Gelfand, Kapranov, Zelevinsky~\cite{GKZ94}, Sturmfels (see~\cite{Huber1995}), or Mikhalkin~\cite{mikhalkin_enumerative}, are at the heart of tropical geometry, see~\cite{itenberg_book,MS15} for background.

Systems of tropical polynomial equations and inequations are also interesting in their own right, since they arise in specific applications, independently of the nonarchimedean interpretation. In particular, they arise in
the computation of stationary behaviors of discrete event systems~\cite{CGQ95a}, see e.g.~\cite{emergency15} for an application to performance evaluation of emergency call centers. Other motivations arise from auction theory~\cite{Baldwin2019}, or chemical reaction networks, see e.g.~\cite{radulescu}.

Grigoriev and Podolskii established in~\cite{GP18} a {\em tropical Nullstellensatz}, which states that a system of tropical polynomial equations is solvable if and only if its linearization, obtained by truncating the Macaulay matrix up to an appropriate degree, is solvable. Their results
also apply to polynomial inequations. Since systems of tropical linear
equations and inequations reduce to mean payoff games~\cite{AGG12}, this provides both theoretical tools (strong duality theorems) and algorithms. The proof of~\cite{GP18}, confirming a conjecture made by Grigoriev~\cite{Grigoriev2012}, is based on very ingenious geometric arguments. However, Grigoriev and Podolskii observed that their proof leads to an estimate of the truncation degree which may not be optimal. 

\subsubsection*{\bf Contribution}

We provide here a {\em tropical Nullstellensatz} (\textbf{\Cref{thm:nullstellensatz}}) for systems of tropical equations and inequations, taking into account the {\em sparse} structure. When specialized to full polynomials with given degrees, we recover the statement of Grigoriev and Podolskii, with an improved degree bound, matching the classical Macaulay degree bound for systems of $n+1$ equations in $n$ unknowns, see \cite{Laz81, Laz83, Giu84}, and also~\cite{Bronstein2005,Cox2015} for background. Our approach also leads to a {\em tropical Positivstellensatz} (\textbf{\Cref{thm:positivstellensatz}}), allowing one to check the inclusion between two tropical basic semialgebraic subsets of $\R^n$, by reduction to a system with both strict and weak tropical linear inequalities. This two results are ultimately combined in the most general form in an \textit{hybrid tropical Positivstellensatz} (\textbf{\Cref{thm:positivstellensatz-hybrid}}).

To do so, we connect the tropical Nullstellensatz with a fundamental notion in classical polynomial system solving, the {\em Canny-Emiris sets}.
In \cite{CE93} and \cite{Emi05}, Canny and Emiris developed an algorithm to compute sparse resultants, for systems of $n+1$ polynomial equations in $n$ variables. This involves the construction of a submatrix of the Macaulay matrix,
in which rows and columns are determined by considering the integer points of a generic perturbation of the Minkowski sum of the Newton polytopes of the polynomials.
Their construction was generalized by Sturmfels in \cite[\S 3]{Stu94}, by considering a generic collection of polyhedral concave functions. Here, we apply this tool to a different problem, showing that any Canny-Emiris set leads to a tropically valid linearization. In a nutshell, we show that if a tropical system is unfeasible, then, the `row contents' arising in the Canny-Emiris construction provide a submatrix of the Macaulay matrix which is tropically nonsingular, and which therefore serves as an unfeasibility certificate.  The case of systems combining polynomial equalities as well as weak and strict inequalities is more delicate: it relies on additional ingredients, including the Shapley-Folkman theorem, and it leads to a degree bound which may not be optimal.
In this case, the reduction involves tropical linear systems combining equalities with weak and strict inequalities. The latter still reduce to mean-payoff games~\cite{uli2013}.

An advantage of the linearization approach lies in the {\em scalability} of mean-payoff games algorithms. Whereas the existence of a polytime algorithm to solve a mean-payoff game is a long standing open problem, in practice, typical large sparse instances of mean-payoff games can be solved efficiently by value-iteration type algorithms~\cite{zwick_paterson,AGG12} or by policy iteration algorithms, see in particular~\cite{dhingra,chaloupka}. The application of game algorithms to solve instances arising from tropical polynomial systems is further discussed in the final section of~\cite{ABG23a}.

For simplicity, we state and prove the present tropical Nullstellensatz and Positivstellensatz by working over the ordinary tropical semifield $\R_\infty$. More generally, a tropical semifield can be constructed over any divisible (totally) ordered group, and this is helpful to study `higher rank' tropicalizations~\cite{Aro10a, Aro10b, AGT16, AI22, JS23}.
Then, using the completeness of the first order theory of divisible ordered groups~\cite[\S 4.3]{Rob56}, it is immediate to see that the present results carry over to this case.

\subsubsection*{\bf Related work}

As discussed above, the present work is inspired by~\cite{GP18}. The main contribution here is the handling of sparsity, with a new proof, leading to an improved truncation bound in the dense case. We also handle strict inequalities, which allows us to interpret the results in terms of `Positivstellens\"atze'.
However, contrarily to what is done in~\cite{GP18}, we limit our attention to the `toric case', looking only for solutions in $\R^n$, see \Cref{rmk:non-toric} for more information.

The standard approach to the computation of tropical prevarieties is to exploit the duality between arrangements of tropical hypersurfaces and mixed-polyhedral subdivisions. In that way, decision problems concerning tropical prevarieties can be reduced to the enumeration of mixed cells, see~\cite{jensenhomotopy,Malajovich2016}. A number of current works deal with the efficient computation of tropical varieties and prevariarieties, see~\cite{Markwig2019, Ren2022} and the references therein.

In these approaches, {\em all} solutions are typically constructed. In contrast, when applying the tropical Nullstellensatz or Positivstellensatz, we do not need to enumerate all the cells of a polyhedral complex (which amounts to testing all candidate solutions). Instead, we directly decide the feasibility or unfeasibility by reduction to a mean payoff game. In the special case of linear inequalities, the advantage of the second approach can be quantified: there are exponentially many tropically extreme solutions~\cite{AlGK09}, hence exponentially many cells, making the enumeration unfeasible unless the dimension is only of a few dozens, whereas checking the feasibility reduces to solving a mean payoff game with a size linearly bounded in input size, which can be done for large systems with pseudo-polynomial complexity bounds. This advantage subsists for a significant class of higher degree instances, although the size of the game now becomes exponential in the input size. We expect the present approach to be especially useful in the `Positivstellensatz' case, in which the exhaustive cell enumeration is especially heavy, whereas only a single feasible cell, or an unfeasibility certificate, is looked for.

A different approach relies on the application of general-purpose SMT solving algorithms~\cite{SMTSolving}.

A tropical Nullstellensatz for tropical ideals, building on~\cite{GP18}, is established in~\cite{rincon}. Other tropical Nullstellens\"atze of different natures have also been established in~\cite{shustin,Bertram2017,JM17,GP20}.

The initial account of the present results appeared in the ISSAC conference paper~\cite{ABG23a}, which also discusses algorithmic aspects. The present paper contains the proofs of the theoretical results of~\cite{ABG23a}.
An implementation of algorithms to decide the feasibility of tropical polynomial systems, using the present sparse Nullstellensatz and Positivstellensatz, is available from~\cite{bereau2023}.

\subsubsection*{Acknowledgments}

We thank Matías Bender for his insightful comments in the comparaison between the classical and the tropical case, as well as for bringing to our attention the Masser-Philippon example (see \Cref{rmk:tight}). We also thank Yue Ren for mentioning a few of the aforementioned potential applications of this work. Finally, we thank the ISSAC reviewers for their helpful feedback.

\section{Generalities on valued fields, polynomials and varieties}

Throughout this paper, we will use the typographic convention to denote tropical objects with normal weight fonts, and reserve bold fonts to denote their classical analogues in a valued field $\vfield$ (most of the time objects in the field of univariate complex generalized Puiseux series).

First, we recall the definition of the various notions we will be working with throughout this paper.
Let $(\Gamma,+,\leq)$ be a totally ordered abelian group. Then the group law and ordering on $\Gamma$ can be extended to the set $\Gamma \cup \lbrace \bot \rbrace$, where $\bot$ denotes the \textit{bottom} element, by setting $\bot \leq v$ and $\bot + v = v + \bot = \bot$ for all $v \in \Gamma$. In $\Gamma \cup \lbrace \bot \rbrace$, the maximum function is defined as usual.

\begin{dfn}\label{def:val}
Let $\vfield$ be a field and $\Gamma$ a totally ordered abelian group. A map $\val$ from $\vfield$ to $\Gamma \cup \lbrace \bot \rbrace$ is called a \textit{valuation} if it satisfies the following properties for all $\puis{x},\puis{y} \in \vfield$:
\begin{enumerate}[label=\it(\roman*)]
\item $\val(\puis{x}) = \bot$ if and only if $\puis{x} = 0$;
\item $\val(\puis{x}\puis{y}) = \val(\puis{x}) + \val(\puis{y})$;
\item $\val(\puis{x}+\puis{y}) \leq \max(\val(\puis{x}),\val(\puis{y}))$.
\end{enumerate}
A field endowed with a valuation is called a \textit{valued field}.
\end{dfn}
Note that in valuation theory, it is more common to call valuation
the opposite of the map $\val$. This would lead us to work which
min-plus type, rather than max-plus type, tropical semifields.

Note that from items (i) and (ii) of the previous definition, we obtain in particular that $\val$ induces a group morphism from $(\vfield^{\ast},\times)$ to $(\Gamma,+)$, where $\vfield^{\ast}:=\vfield\setminus\{0\}$ is the set of invertible elements. This motivates the following definition.

\begin{dfn}
The \textit{value group} of a valued field $\vfield$ is the group $\vgroup := \val(\vfield^{\ast})$.
\end{dfn}

\begin{rmk}
The value group $\vgroup$ is thus a subgroup of $\Gamma$. However, in a lot of cases, we will be working with $\Gamma = \R$. In this case, we say that $\vfield$ is a field with \textit{real valuation} and that $\val$ is a \textit{real valuation}.
\end{rmk}

\begin{expl}
A basic example of valued field is the field $\C \lbrace\!\lbrace t \rbrace\!\rbrace$ of complex \textit{Puiseux series} over one indeterminate. This field is endowed with the valuation given by
\[
\val(\puis{x}) := -\min \lbrace q \in \Q : a_q \neq 0 \rbrace \quad \textrm{if} \quad \puis{x} = \sum_{q \in \Q} \puis{x}_q t^q
\]
and in this case we have $\vgroup = \Q$.

We will also consider the field $\hahnseries{\C}{\R}$ of univariate complex \textit{Hahn series}\footnote{Sometimes the notation $\C((\R))$ is also encountered, notably in \cite{MS15}.}, that is the set of formal sums
\[
\puis{x} = \sum_{r \in \R} \puis{x}_rt^r \quad \textrm{such that} \quad \lbrace r \in \R : \puis{x}_r \neq 0 \rbrace \textrm{ is a well ordered subset of } \R\enspace ,
\]
in which case the quantity
\[
\val(\puis{x}) := -\min \lbrace r \in \R : \puis{x}_r \neq 0 \rbrace
\]
is well-defined and yields a valuation of $\hahnseries{\C}{\R}$, which is surjective in this case, \textit{i.e.} $\vgroup = \R$.
\end{expl}

In the context of tropical algebra, the algebraic structure we typically use for calculations is not a ring, as we are used to, but a semiring. The definition of a semiring is the following.

\begin{dfn}
A \textit{semiring} $(R,+,\cdot,0,1)$ is a set $R$ equipped with addition $+$ and multiplication $\cdot$ such that
\begin{enumerate}[label=\it (\roman*)]
\item $(R,+)$ is a commutative monoid with identity element $0$ called the \textit{zero element};
\item $(R,\cdot)$ is a monoid with identity element $1$ called the \textit{unit element};
\item the multiplication is distributive over the addition;
\item $0$ is the absorbing element for the multiplication.
\end{enumerate}

If moreover, the multiplication is commutative, then $R$ is called a \textit{commutative semiring}, and if every nonzero element has an inverse for the multiplication, then $R$ is called a \textit{semifield}.
\end{dfn}

Note that a semiring $R$ satisfy all the properties of a ring except one: the existence of an opposite element for the addition is not guaranteed. In particular, a ring is a semiring. There are many interesting examples of semirings besides rings, such as the set $\N$ of nonnegative integers, endowed with the usual arithmetic operations, the sets of ideals of a ring, with ideal addition and multiplication, or any Boolean algebra, with $\vee$ as addition and $\wedge$ as multiplication. However, in this paper, we will be focusing on the following semiring.

\begin{dfn}
The \textit{tropical} (or \textit{max-plus}) \textit{semiring} is the semiring $(\R \cup \lbrace -\infty \rbrace, \tplus, \tdot, \zero, \unit)$ with addition $\tplus = \max$, multiplication $\tdot = +$, zero element $\zero = -\infty$ and unit element $\unit = 0$. The operations $\tplus$ and $\tdot$ are respectively refered to as the \textit{tropical addition} and the \textit{tropical multiplication}, and the tropical semiring is denoted by $\T$.
\end{dfn}

\begin{rmk}
In fact, we can similarly define a semiring structure for any totally ordered abelian group with bottom element $\Gamma \cup \lbrace \bot \rbrace$, but in the context of this paper, we will focus on the case where $\Gamma = \R$ once again. Moreover, note that $\T$ is not only a semiring but actually a semifield.
\end{rmk}

The definition of a formal polynomial over a semiring is identical to the usual definition of a formal polynomial over a ring.

\begin{dfn}
A \textit{formal Laurent polynomial} $f$ over a semiring --- and \textit{a fortiori} over a ring --- $(R,+,\cdot,0,1)$ in $n$ variables is a map
\[
\begin{array}{rcl}
\Z^n & \longrightarrow & R\\
\alpha & \longmapsto & \puis{f}_\alpha
\end{array}
\]
such that $\puis{f}_\alpha = 0$ for all $\alpha \in \Z^n$ but a finite number. If $\alpha = (\alpha_1,\ldots,\alpha_n)$, then we use the following notation:
\[
\puis{f} = \sum_{\alpha \in \Z^n} \puis{f}_\alpha X^\alpha = \sum_{\alpha \in \Z^n} \puis{f}_\alpha X_1^{\alpha_1} \cdots X_n^{\alpha_n}
\]
and we denote by $R[X_1^{\pm 1},\ldots,X_n^{\pm 1}]$ the set of formal Laurent polynomials in $n$ variables over the semiring $R$.

The \textit{support} $\supp(\puis{f})$ of a formal Laurent polynomial $\puis{f} \in R[X_1^{\pm 1},\ldots,X_n^{\pm 1}]$ is the set
\[
\lbrace \alpha \in \Z^n : \puis{f}_\alpha \neq 0 \rbrace \enspace .
\]

Finally, the \textit{polynomial function} associated to a formal Laurent polynomial $\puis{f} \in R[X_1^{\pm 1},\ldots,X_n^{\pm 1}]$ is the function
\[
\begin{array}{rcl}
(R^{\times})^n & \longrightarrow & R\\
\puis{x} & \longmapsto & \sum_{\alpha \in \Z^n} \puis{f}_\alpha x^\alpha \enspace ,
\end{array}
\]
where $R^{\times}$ denotes the multiplicative group of invertible elements of $R$ and $\puis{x}^\alpha = \puis{x}_1^{\alpha_1} \times \cdots \times \puis{x}_n^{\alpha_n}$ if $\puis{x} = (\puis{x}_1,\ldots,\puis{x}_n)$ and $\alpha = (\alpha_1,\ldots,\alpha_n)$. More generally, we call polynomial function any function of the previous form.\\

In particular, in the case where the semiring $R$ corresponds to the tropical semiring, then a formal Laurent polynomial over $R$ will be called a \textit{formal tropical Laurent polynomial}, and its associated polynomial function a \textit{tropical polynomial function}.
\end{dfn}

In the remainder of this paper, we will refer to formal Laurent polynomials as \textit{formal polynomials}, or even simply as \textit{polynomials} if the context is unambiguous, and a (formal Laurent) polynomial $f$ with support included in $\N^n$ will be refered to as an \textit{ordinary polynomial} whenever we need to highlight the fact that all the exponents belong in the positive orthant, and we denote by $R[X_1,\ldots,X_n]$ the set of formal ordinary polynomials in $n$ variables over a semiring $R$.

\begin{rmk}
If $f$ is a tropical polynomial in $n$ variables, then the tropical polynomial function associated to $f$ is the function
\[
\begin{array}{rcl}
\R^n & \longrightarrow & \T\\
x & \longmapsto & \tsum_{\alpha \in \Z^n} f_\alpha \tdot x^{\tdot \alpha} = \max_{\alpha \in \Z^n} \left( f_\alpha + \langle x, \alpha \rangle \right) \enspace ,
\end{array}
\]
where $x^{\tdot \alpha} := x_1^{\tdot \alpha_1} \tdot \cdots \tdot x_n^{\tdot \alpha_n} = \langle x, \alpha \rangle$ denotes the usual scalar product of $x = (x_1,\ldots,x_n)$ and $\alpha = (\alpha_1,\ldots,\alpha_n)$ in $\R^n$.
\end{rmk}

Note that while in the case of an infinite field, there is a one-to-one correspondance between polynomial functions and formal polynomials, this is not the case in the tropical setting, where two distinct formal tropical polynomial can have the same tropical polynomial function. Therefore, we will always be explicit
about the nature of the tropical objects we manipulate.

Now we can discuss the notion of root. A \textit{root} of a polynomial $\puis{f}$ over a ring $R$ is an element $\puis{x}$ of $R$ such that evaluating the polynomial function of $\puis{f}$ at $\puis{x}$ gives the zero element. This definition, however, is not suited for polynomials over semirings, and in particular for tropical polynomials, for which we need to give another definition of a root.

\begin{dfn}
Let $f = \tsum_{\alpha \in \Z^n} f_\alpha X^\alpha \in \T[X_1^\pm,\ldots,X_n^\pm]$ be a tropical polynomial. An element $x = (x_1,\ldots,x_n)$ of $\R^n$ is called a (\textit{tropical}) \textit{root} of $f$ whenever the maximum in the expression
\[
\max_{\alpha \in \Z^n} (f_\alpha + \langle x, \alpha \rangle )
\]
is attained for at least two distinct values of $\alpha \in \Z^n$. This is denoted\footnote{This notation is not only used to convey the fact that $f(x) \bal \zero$ is the tropical analogue of the equation $\puis{f}(\puis{x}) = 0$ in a ring, but it also has a deeper meaning as it comes from a natural binary relation that arises in extensions of the tropical semifield. For more details, see for instance Section 4. of \cite{AGG08}.} as $f(x) \bal \zero$.

In the case where $f$ is an ordinary tropical polynomial, then a root $x$ of $f$ is an element $\T^n$ (instead of $\R^n$) satisfying the same property, and in this case the \textit{support} of the root is the set of indices whose associated coefficient is not equal to $\zero$. 
\end{dfn}

Finally, note that for a finite collection $\puis{f} = (\puis{f}_1, \ldots, \puis{f}_k)$ of Laurent polynomials, you can always find a monomial $X^{\alpha}$ such that for all $1 \leq i \leq k$, the polynomial $X^{\alpha} \puis{f}_i$ is an ordinary polynomial. Therefore, we could in most cases only consider systems of ordinary polynomials.

Now, we recall a few definitions from \cite{MS15} regarding the notion of tropicalization.

\begin{dfn}
Let $\vfield$ be a field with valuation $\val$, and consider a Laurent polynomial $\puis{f} = \sum_{\alpha \in \Z^n} \puis{f}_{\alpha} X^{\alpha} \in \vfield[X_1^{\pm 1},\ldots,X_n^{\pm 1}]$. Then the \textit{hypersurface} (or \textit{variety}) associated to $\puis{f}$ is the set
\[
\var_{\vfield}(\puis{f}) := \left\lbrace \puis{x} \in (\vfield^{\ast})^n : \puis{f}(\puis{x}) = 0 \right\rbrace \enspace .
\]

Similarly, consider a tropical Laurent polynomial $f \in \T[X_1^{\pm 1}, \ldots, X_n^{\pm 1}]$. Then the \textit{tropical hypersurface} (or \textit{tropical variety}) associated to the tropical polynomial $f$ is the set
\[
\tvar(f) := \left\lbrace x \in \R^n, f(x) \bal \zero \right\rbrace \enspace ,
\]
\end{dfn}

\begin{dfn}
Let $\vfield$ be a field with valuation $\val$, and consider a Laurent polynomial $\puis{f} = \sum_{\alpha \in \Z^n} \puis{f}_{\alpha} X^{\alpha} \in \vfield[X_1^{\pm 1},\ldots,X_n^{\pm 1}]$.
Then:
\begin{itemize}[label=$\diamond$]
\item the \textit{tropicalization} of $\puis{f}$ is the tropical polynomial function $\trop(\puis{f})$ defined by
\[
\trop(\puis{f}) : \left\lbrace\begin{array}{rcl}
\R^n & \longrightarrow & \T\\
x & \longmapsto & \tsum_{\alpha \in \Z^n} \val(\puis{f}_{\alpha}) \tdot x^{\tdot \alpha} = \max_{\alpha \in \Z^n} (\val(\puis{f}_{\alpha}) + \langle x, \alpha \rangle) \enspace ;\end{array}\right.
\]
\item the \textit{tropicalization} of $\var_{\vfield}(\puis{f})$ or the \textit{tropical hypersurface} associated to the (ordinary) polynomial $\puis{f}$ is the set $\tropi_{\vfield}(\puis{f})$ defined by
\[
\tropi_{\vfield}(\puis{f}) = \tvar(\trop(\puis{f})) \enspace .
\]
\textit{i.e.} it corresponds to the set of points $x$ in $\R^n$ such that the maximum in the expression $\trop(\puis{f})(x)$ is achieved at least twice.
\end{itemize}
\end{dfn}

We can then define the tropicalization of any variety of $(\vfield^{\ast})^n$ as follows.

\begin{dfn}
Let $I$ be an ideal of $\vfield[X_1^{\pm 1},\ldots,X_n^{\pm 1}]$, and $\var_{\vfield}(I)$ the variety it defines in $(\vfield^{\ast})^n$. Then the \textit{tropicalization} of the variety $\var_{\vfield}(I)$, or the \textit{tropical variety} associated to the ideal $I$ is the subset $\tropi_{\vfield}(I)$ of $\R^n$ defined by
\begin{align}
\tropi_{\vfield}(I) = \bigcap_{\puis{f} \in I} \tropi_{\vfield}(\puis{f}) \enspace .\label{e-finite}
\end{align}
More generally, we call \textit{tropical variety} in $\R^n$ any subset of the previous form.
\end{dfn}

In the case where $\vfield$ has a nontrivial valuation, then there is an equivalent way to express $\tropi_{\vfield}(I)$.

\begin{thm}[Fundamental theorem of Tropical Algebraic Geometry, see {\cite[Theorem~3.2.3]{MS15}.}]\label{th-fund}
Let $\vfield$ be an algebraically closed field endowed with a nontrivial valuation $\val$ and let $I$ be an ideal in $\vfield[X_1^{\pm 1},\ldots,X_n^{\pm 1}]$. Then the following two subsets of $\R^n$ coincide:
\begin{enumerate}[label=(\roman*)]
\item the tropical variety $\tropi_{\vfield}(I)$;
\item the closure of the set of coordinatewise valuations of points of $\var_{\vfield}(I)$, \textit{i.e.} of the set
\[
\val(\var_{\vfield}(I)) = \lbrace (\val(\puis{x}_1),\ldots,\val(\puis{x}_n)) \in \R^n : (\puis{x}_1,\ldots,\puis{x}_n) \in \var_{\vfield}(I)\rbrace \enspace .
\]
\end{enumerate}
\end{thm}

\Cref{th-fund} motivates the present work. Indeed, for any finite family of polynomials $\puis{f}_1,\dots,\puis{f}_k$ in the ideal $I$, checking that $\bigcap_{j\in[k]} \tropi_{\vfield}(\puis{f}_k) = \emptyset$, which can be done by
applying the present tropical Nullstellensatz, entails that
$\val(\var_{\vfield}(I))=\emptyset$. Hence, the tropical Nullstellensatz
provides a certificate of emptyness for an algebraic variety over a valued
field. Moreover, it is shown in~\cite{MS15} that the intersection~\eqref{e-finite} is achieved
by considering only a finite subintersection, which entails that the collection
of certificates obtained in this way is complete.

\section{The Sparse Tropical Nullstellensatz}\label{sec:nullstellensatz}

\subsection{Statement of the theorem}

The idea of the main theorem in this section is to reduce the problem of the existence of a solution to a system of polynomials equations to the existence of a solution to a system of tropical linear equations arising from a certain matrix called the Macaulay matrix, which can be constructed using the coefficients of the polynomials $f_1, \ldots, f_k$. From now on, we denote by $f$ the collection $(f_1, \ldots, f_k)$ of polynomials, and by $f \bal \zero$ the system
\[
\tsum_{\alpha \in \calA_i} f_{i,\alpha} \tdot x^{\tdot \alpha} \bal \zero \quad \textrm{for all} \quad 1 \leq i \leq k
\]
of tropical polynomial equations with unknown $x \in \R^n$.

We start by giving a proper setting to talk about tropical linear equations. We call \textit{tropical matrix} a matrix with coefficients in $\T$. For two integers $p, q \in \Npos$, the set of $p \times q$ tropical matrices is denoted by $\T^{p \times q}$. We can define tropical addition $\tplus$ and multiplication $\tdot$ on tropical matrices by replacing the usual operations by their tropical version in the definition of the usual matrix operations. This notably gives a semiring structure to the set $\T^{p \times p}$.

Particularly, for $A = (a_{ij})_{(i, j) \in [p] \times [q]} \in \T^{p \times q}$ and $y = (y_j)_{j \in [q]} \in \T^q$, one has
\begin{equation}\label{eq:tropical-matrix-vector-product}
A \tdot y = \left( \max_{1 \leq j \leq q} a_{ij} + y_j \right)_{i \in [p]} \enspace .
\end{equation}

\begin{dfn}
Let $A$ be a $p \times q$ tropical matrix and let $y \in \T^q$. Then we write that $A \tdot y \bal \zero$ whenever the maximum is attained twice for every coordinate in the righthandside of \eqref{eq:tropical-matrix-vector-product}. The set of vectors $y \in \T^q$ such that $A \tdot y \bal \zero$ is called the \textit{tropical right null space} or \textit{kernel} of the matrix $A$. Moreover, we set by convention that all vectors $y \in \T^q$ are in the tropical kernel of a $0 \times m$ matrix.
\end{dfn}

\begin{rmk}
Note that as in the usual case, the tropical matrix equation $A \tdot y \bal \zero$ can be written as the following $q$-variate tropical polynomial --- linear in fact --- system
\[
\forall i \in [p], \quad \tsum_{j = 1}^q a_{ij} \tdot y_j \bal \zero \enspace .
\]
\end{rmk}

Now, we define the tropical Macaulay matrix associated to the system $f$, which plays a crucial role in the determination of the solvability of a polynomial system, with no restriction on the number of equations.

\begin{dfn}
Given a collection of tropical polynomials $f = (f_1, \ldots, f_k)$, we define the \textit{tropical Macaulay matrix} $\mac$ of the system as such: the rows of $\mac$ are indexed by pairs $(i,\alpha)$ where $1 \leq i \leq k$ and $\alpha \in \Z^n$, the columns of $\mac$ are indexed by integer vectors $\beta \in \Z^n$, and for given $(i,\alpha)$ and $\beta$, we set the entry $\mac_{(i,\alpha),\beta}$ of $\mac$ equal to the coefficient of the monomial $X^{\beta}$ in the polynomial $X^{\alpha}f_i(X)$, or $-\infty$ if no such monomial exists.

Given a Macaulay matrix $\mac$ as above, a {\em nonempty} finite subset $\calE$ of $\Z^n$, and a collection $\calA = (\calA_1, \ldots, \calA_k)$ of subsets of $\Z^n$, 
we denote by $\mac_{\calE}^{\calA}$ the submatrix of $\mac$ consisting only of the columns with indices $\beta \in \calE$, and the rows indexed by pairs $(i,\alpha)$ where $1 \leq i \leq k$ and $\alpha \in \Z^n$ such that $\alpha+\calA_i \subseteq \calE$. When the polynomials $f_i$ have their support equal to $\calA_i$, 
and $\mac$ is associated to $f= (f_1, \ldots, f_k)$, we 
simply write $\mac_{\calE}$ instead of $\mac_{\calE}^{\calA}$.
\end{dfn}

\begin{rmk}
Note that, equivalently, $\mac_{\calE}$ is the submatrix of $\mac$ consisting of the columns with indices $\beta \in \calE$, and the rows that have all their finite entries in these columns. Moreover, if $\calE$ is nonempty but too small, it might be possible that there are no such row of $\mac$, and thus the set of rows of $\mac_{\calE}^{\calA}$ might be empty, in which case by the convention of the previous definition, all vectors are considered to be in the tropical kernel of $\mac_{\calE}^{\calA}$.
\end{rmk}

Now, let us denote as previously by $\calA = (\calA_1, \ldots, \calA_k)$ a collection of subsets of $\Z^n$, and for all $1 \leq i \leq k$, let $Q_i$ be the convex hull of $\calA_i$ and set $Q := Q_1 + \cdots + Q_k$. Take a generic vector $\delta \in V + \Z^n$ where $V \subseteq \R^n$ is the vector space directing the affine hull of $Q$, and consider the set
\[
\calE := (Q + \delta) \cap \Z^n \enspace .
\]
We will refer to sets of this form as \textit{Canny-Emiris subsets} of $\Z^n$ 
associated to the collection $\calA$. Note that for $\delta$ small enough, we always have the inclusion
\[
\relint(Q) \cap \Z^n \subseteq \calE \subseteq Q \cap \Z^n\enspace ,
\]
where $\relint$ denotes the relative interior.

Now, for a collection $f = (f_1, \ldots, f_k)$ of tropical polynomials, we shall consider in particular the collection $\calA = (\calA_1, \ldots, \calA_k)$ where $\calA_i$ is the support of $f_i$ for all $i \in [k]$. In that case, the set $Q_i$ corresponds to the Newton polytope $\NP_{f_i}$ of $f_i$, and we shall also refer to $Q$ as the \textit{Newton polytope} of $f$. Also the Canny-Emiris subsets $\calE$ associated to the collection $\calA$ of supports are refered to as the \textit{Canny-Emiris sets associated to} $f$.

The tropical linear system $\mac_{\calE} \tdot y \bal \zero$ will be of interest to us in the resolution of the previous problem. More precisely, we have the following result, which will be proven in \Cref{sec:proof-nullstellensatz}.

\begin{thm}[Sparse tropical Nullstellensatz] 
\label{thm:nullstellensatz}
There exists a common root $x \in \R^n$ to the system $f(x) \bal \zero$ if and only if there exists a vector $y \in \R^{\mathcal{E'}}$ in the tropical kernel of the submatrix $\mac_{\mathcal{E'}}$ of the Macaulay matrix $\mac$ associated to the collection $f$ --- \emph{i.e.} $\mac_{\mathcal{E'}} \tdot y \bal \zero$ --- where $\mathcal{E'}$ is any subset of $\Z^n$ containing a nonempty Canny-Emiris subset $\calE$ of $\Z^n$ associated to $f$.

Moreover, if $\calE'=\calE$, these conditions are equivalent to the existence of a vector $y \in \T^{\calE} \setminus \{ \zero \}$ such that $\mac_{\calE} \odot y \bal \zero$.
\end{thm}

When we have no particular information on the supports of the polynomials $f_i$ besides that they are ordinary tropical polynomials  with respective degrees $d_i$, we denote by $\mac_N$ the submatrix $\mac_{\overline{\calE}}^{\calA}$ of $\mac$
with $\overline{\calE} = N\Delta \cap \N^n$, and $\calA_i = d_i\Delta \cap \N^n$, where
\[
\Delta := \{ \alpha \in \Rnonneg^n : \vert \alpha \vert = \alpha_1 + \cdots + \alpha_n \leq 1 \}
\]
denotes the unit simplex. In this case, the integer $N$ is called the \textit{truncation degree} of the Macaulay submatrix $\mac_N$.

More generally, if it is only known, for all $i=1,\ldots, k$, that the support of $f_i$ is included in $\calA_i$ (which plays now the role of an \emph{a priori} support), then one shall consider a bigger set
\[
\overline{\calE} := Q  \cap \Z^n \enspace ,
\]
where $Q  = \conv(\calA_1) + \cdots + \conv(\calA_k)$ is defined using the collection $\calA = (\calA_1, \ldots, \calA_k)$. Recall that for $\delta$ small enough $\calE \subseteq \overline{\calE}$. In that case, one has the following theorem, in which we consider the matrices $\mac_{\calE'}^{\calA}$ with $\calE' \supseteq \overline{\calE}$.

\begin{thm}[Nullstellensatz for sparse Tropical Polynomial Systems with \emph{a priori} supports]
\label{thm:null2}
There exists a common root $x \in \R^n$ to the system $f(x) \bal \zero$ if and only if there exists a vector $y \in \R^{\mathcal{E'}}$ in the tropical kernel of the submatrix $\mac^{\calA}_{\mathcal{E'}}$ of the Macaulay matrix $\mac$ associated to the collection $f$, where $\mathcal{E'}$ is any subset of $\Z^n$ containing $\overline{\calE}$ and $\calA$ is a collection of a priori supports of the $f_i$.

Moreover, when the Newton polytope of $f$ has the same dimension as $Q$, one can replace $\overline{\calE}$ by any nonempty Canny-Emiris set $\calE$ associated to $\calA$.
\end{thm}

\begin{proof}
The inclusions $\calE \subseteq \calE'$ and $\supp(f_i) \subseteq \calA_i$ for all $1 \leq i \leq k$ imply that the matrix $\mac_\calE$ is a submatrix of the matrix $\mac^{\calA}_{\calE'}$. More precisely, if $(i, \alpha)$ is the index of a row of $\mac_{\calE}$, then this indicates that the support of polynomial $x^\alpha f_i$ is included in $\calE$, and therefore
\begin{align*}
\alpha \in \conv \left( \sum_{1 \leq j \neq i \leq n} \supp(f_j) \right) \subseteq \conv \left( \sum_{1 \leq j \neq i \leq n} \calA_j \right) \enspace ,
\end{align*}
which shows that $(i, \alpha)$ is also the index of a row of the matrix $\mac_{\calE'}^{\calA}$. The theorem then follows directly from \Cref{thm:nullstellensatz} and from the latter remark.
\end{proof}

When $f = (f_1, \ldots, f_k)$ is a collection of ordinary tropical polynomials $f_i$ with respective degree $d_i$ and the matrix $\mac_N$ is defined as above, applying \Cref{thm:null2} with $\calA_i = d_i\Delta \cap \N^n$, we deduce the following result.

\begin{cor}\label{cor:null}
There exists a common root $x \in \R^n$ to the system $f(x) \bal \zero$ if and only if there exists a vector $y\in \R^{\mathcal{E'}}$ such that $\mac_N \tdot y \bal \zero$, where ${\mathcal{E'}}={N\Delta\cap \N^n}$ and 
$N\geq  d_1 + \cdots + d_k$.
Moreover, when the Newton polytope of $f$ has full dimension, one can replace the above lower bound on $N$ by $N\geq  d_1 + \cdots + d_k-n$.
\end{cor}

\begin{proof}
This result is obtained by applying \Cref{thm:null2} with $\calA_i=d_i\Delta\cap \N^n$, and $\calE'=N\Delta\cap \N^n$, since $\overline{\calE} = (d_1+\cdots +d_k)\Delta\cap \N^n\subset \calE'$. Moreover, for the full-dimensional case, one can perturb the simplex $(d_1+\cdots +d_k)\Delta$ by a perturbation $(\varepsilon, \ldots, \varepsilon)$ for $\varepsilon > 0$ small enough and the resulting Canny-Emiris set is $(d_1+\cdots +d_k - n)\Delta \cap \N^n$.
\end{proof}

\begin{expl}\label{expl:system-nullstellensatz}
Let us illustrate \Cref{thm:nullstellensatz} with some explicit examples. Consider the following two systems:
\[
(\mathcal{S}_1) : \left\{\begin{array}{rcl}
f_1 &=& 1 \tplus 2x_1 \tplus 1x_2 \tplus 1x_1x_2\\
f_2 &=& 0 \tplus 0x_1 \tplus 1x_2\\
f_3 &=& 2x_1 \tplus 0x_2
\end{array}\right. \quad \textrm{and} \quad (\mathcal{S}_2) : \left\{\begin{array}{rcl}
f_1 &=& 1 \tplus 4x_1 \tplus 1x_2 \tplus 3x_1x_2\\
f_2 &=& 0 \tplus 0x_1 \tplus 1x_2\\
f_3 &=& 2x_1 \tplus 0x_2
\end{array}\right.
\]
Both systems have the same supports, and thus yield the same polytope $Q$.
\begin{figure}[h]
\centering
\begin{tikzpicture}[line cap=round,line join=round,>=triangle 45,x=1cm,y=1cm]
\clip(-1.5,-1.5) rectangle (10.5,4.5);
\foreach \x in {-1,...,10}
    \foreach \y in {-1,...,4}
    {
    \draw (\x,\y) circle (0.5pt);
    }
\fill[line width=2pt,fill=black,fill opacity=0.1] (0,0) -- (1,0) -- (1,1) -- (0,1) -- cycle;
\fill[line width=2pt,fill=black,fill opacity=0.1] (2,0) -- (3,0) -- (2,1) -- cycle;
\draw [line width=1pt] (0,0)-- (1,0);
\draw [line width=1pt] (1,0)-- (1,1);
\draw [line width=1pt] (1,1)-- (0,1);
\draw [line width=1pt] (0,1)-- (0,0);
\draw [line width=1pt] (2,0)-- (3,0);
\draw [line width=1pt] (3,0)-- (2,1);
\draw [line width=1pt] (2,1)-- (2,0);
\draw [line width=1pt] (4,1)-- (5,0);
\begin{tiny}
\draw [fill=black] (0,0) circle (1pt);
\draw[color=black] (-0.05,-0.2) node {$(0,0)$};
\draw [fill=black] (1,0) circle (1pt);
\draw[color=black] (1.05,-0.2) node {$(1,0)$};
\draw [fill=black] (1,1) circle (1pt);
\draw[color=black] (1.05,1.2) node {$(1,1)$};
\draw [fill=black] (0,1) circle (1pt);
\draw[color=black] (-0.05,1.2) node {$(0,1)$};
\draw [fill=black] (2,0) circle (1pt);
\draw[color=black] (1.95,-0.2) node {$(0,0)$};
\draw [fill=black] (3,0) circle (1pt);
\draw[color=black] (3.05,-0.2) node {$(1,0)$};
\draw [fill=black] (2,1) circle (1pt);
\draw[color=black] (1.95,1.2) node {$(0,1)$};
\draw [fill=black] (4,1) circle (1pt);
\draw[color=black] (3.95,1.2) node {$(0,1)$};
\draw [fill=black] (5,0) circle (1pt);
\draw[color=black] (5.05,-0.2) node {$(1,0)$};
\end{tiny}
\draw[color=black] (0.5,-0.6) node {$Q_1$};
\draw[color=black] (2.5,-0.6) node {$Q_2$};
\draw[color=black] (4.5,-0.6) node {$Q_3$};

\fill[line width=2pt,fill=black,fill opacity=0.1] (6,3) -- (6,1) -- (7,0) -- (9,0) -- (9,1) -- (7,3) -- cycle;
\draw [line width=1pt] (6,3)-- (6,1);
\draw [line width=1pt] (6,1)-- (7,0);
\draw [line width=1pt] (7,0)-- (9,0);
\draw [line width=1pt] (9,0)-- (9,1);
\draw [line width=1pt] (9,1)-- (7,3);
\draw [line width=1pt] (7,3)-- (6,3);
\begin{tiny}
\draw [fill=black] (6,1) circle (1pt);
\draw[color=black] (5.65,1) node {$(0,1)$};
\draw [fill=black] (7,0) circle (1pt);
\draw[color=black] (7,-0.2) node {$(1,0)$};
\draw [fill=black] (6,3) circle (1pt);
\draw[color=black] (5.9,3.2) node {$(0,3)$};
\draw [fill=black] (7,3) circle (1pt);
\draw[color=black] (7.1,3.2) node {$(1,3)$};
\draw [fill=black] (9,1) circle (1pt);
\draw[color=black] (9.35,1) node {$(3,1)$};
\draw [fill=black] (9,0) circle (1pt);
\draw[color=black] (9,-0.2) node {$(3,0)$};
\draw [fill=black] (6,2) circle (0.5pt);
\draw [fill=black] (7,1) circle (0.5pt);
\draw [fill=black] (7,2) circle (0.5pt);
\draw [fill=black] (8,0) circle (0.5pt);
\draw [fill=black] (8,1) circle (0.5pt);
\draw [fill=black] (8,2) circle (0.5pt);
\end{tiny}
\draw[color=black] (8,-0.6) node {$Q = Q_1 + Q_2 + Q_3$};
\end{tikzpicture}
\caption{The Newton polytopes associated to $f_1$, $f_2$, $f_3$ and their Minkowski sum.}
\end{figure}
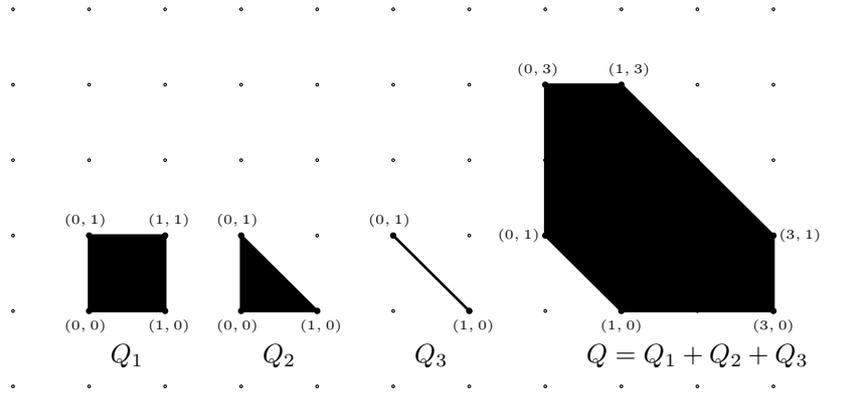\\
For this collection of supports, one can take $\delta = (-1+\varepsilon,-1+\varepsilon)$ with $\varepsilon > 0$ sufficiently small, for instance $\varepsilon = 0.1$, which gives us the Canny-Emiris set
\[
\calE := (Q + \delta) \cap \Z^n = \{ (0,0), (1,0), (0,1), (2,0), (1,1), (0,2) \} \enspace ,
\]
corresponding to the set of monomials
\[
\{ 1, x_1, x_2, x_1^2, x_1x_2, x_2^2 \} \enspace .
\]
Note that we could simply have chosen $\delta$ to be $(\varepsilon,\varepsilon)$ instead, but we chose to also translate $Q$ by $(-1,-1)$ so that $Q+\delta$ contains the origin, thus ensuring that $\calE$ contains smaller degree monomials. We can therefore write the respective submatrices $\mac^{(1)}_{\calE}$ and $\mac^{(2)}_{\calE}$ of the Macaulay matrix associated to the set $\calE$ and we obtain the following $7 \times 6$ matrices
\[
\mac^{(1)}_{\calE} =
\begin{blockarray}{*{7}{c}}
& 1 & x_1 & x_2 & x_1^2 & x_1x_2 & x_2^2\\
\begin{block}{c(*{6}{c})}
f_1 & 1 & 2 & 1 & & 1 &\\
f_2 & 0 & 0 & 1 & & &\\
x_1f_2 & & 0 & & 0 & 1 &\\
x_2f_2 & & & 0 & & 0 & 1\\
f_3 & & 2 & 0 & & &\\
x_1f_3 & & & & 2 & 0 &\\
x_2f_3 & & & & & 2 & 0\\
\end{block}
\end{blockarray}
\quad \textrm{and} \quad
\mac^{(2)}_{\calE} =
\begin{blockarray}{*{7}{c}}
& 1 & x_1 & x_2 & x_1^2 & x_1x_2 & x_2^2\\
\begin{block}{c(*{6}{c})}
f_1 & 1 & 4 & 1 & & 3 &\\
f_2 & 0 & 0 & 1 & & &\\
x_1f_2 & & 0 & & 0 & 1 &\\
x_2f_2 & & & 0 & & 0 & 1\\
f_3 & & 2 & 0 & & &\\
x_1f_3 & & & & 2 & 0 &\\
x_2f_3 & & & & & 2 & 0\\
\end{block}
\end{blockarray} \enspace .
\]
One can check that the system $(\mathcal{S}_1)$ does not have a common root, as the different intersection points of the tropical hypersurfaces associated to $f_1$, $f_2$ and $f_3$ are listed on \Cref{fig:sys1} (and we know that two tropical lines only have at most one intersection point, and a line and a quadric have at most two intersection points).

\begin{figure}[H]
\centering
\begin{multicols}{2}
\scalebox{0.75}{
\begin{tikzpicture}[line cap=round,line join=round,>=triangle 45,x=1cm,y=1cm,scale=0.75]
\clip(-8,-4) rectangle (8,4);
\draw [line width=2pt,color=color3,domain=-12:12] plot(\x,{(--2--1*\x)/1});
\draw [line width=2pt,color=color2,domain=0:12] plot(\x,{(-9--9*\x)/9});
\draw [line width=2pt,color=color2] (0,-1) -- (0,-8);
\draw [line width=2pt,color=color2,domain=-12:0] plot(\x,{(--6-0*\x)/-6});
\draw [line width=2pt,color=color1] (-1,0) -- (-12,0);
\draw [line width=2pt,color=color1] (-1,0) -- (-1,-8);
\draw [line width=2pt,color=color1] (-1,0) -- (0,1);
\draw [line width=2pt,color=color1] (0,1) -- (0,8);
\draw [line width=2pt,color=color1] (0,1) -- (12,1);
\begin{scriptsize}
\draw [fill=color2] (0,2) circle (3.5pt);
\draw[color=color2] (-0.7,2.3) node {$(0,2)$};
\draw [fill=color3] (2,1) circle (3.5pt);
\draw[color=color3] (2.6,0.6) node {$(2,1)$};
\draw [fill=color3] (-1,-1) circle (3.5pt);
\draw[color=color3] (-2,-1.4) node {$(-1,-1)$};
\draw [fill=color2] (-2,0) circle (3.5pt);
\draw[color=color2] (-2.6,0.4) node {$(-2,0)$};
\draw [fill=color1] (-3,-1) circle (3.5pt);
\draw[color=color1] (-4.6,-1.4) node {$(-3,-1)$};
\end{scriptsize}
\begin{large}
\draw[color=color1] (7,0.5) node {$\tvar(f_1)$};
\draw[color=color2] (6,3.5) node {$\tvar(f_2)$};
\draw[color=color3] (-4,-3.5) node {$\tvar(f_3)$};
\end{large}
\end{tikzpicture}
}
\flushright
\begin{tikzpicture}[line cap=round,line join=round,>=triangle 45,x=1cm,y=1cm]
\clip(-1,-1) rectangle (4,4);
\draw[line width=0.5pt,fill=color1,fill opacity=0.1] (0,2) -- (1,1) -- (1,0) -- (0,1) -- cycle;
\draw[line width=0.5pt,fill=color2,fill opacity=0.1] (0,3) -- (1,2) -- (1,1) -- (0,2) -- cycle;
\draw[line width=0.5pt,fill=color2,fill opacity=0.1] (0,3) -- (1,3) -- (2,2) -- (1,2) -- cycle;
\draw[line width=0.5pt,fill=color3,fill opacity=0.1] (2,2) -- (2,1) -- (3,0) -- (3,1) -- cycle;
\draw[line width=0.5pt,fill=color3,fill opacity=0.1] (1,1) -- (2,1) -- (2,0) -- (1,0) -- cycle;
\draw [line width=0.5pt,color=black] (1,2) -- (2,1);
\draw [line width=0.5pt,color=black] (1,0) -- (0,1);
\draw [line width=0.5pt,color=black] (0,1) -- (0,3);
\draw [line width=0.5pt,color=black] (0,3) -- (1,3);
\draw [line width=0.5pt,color=black] (1,3) -- (3,1);
\draw [line width=0.5pt,color=black] (3,1) -- (3,0);
\draw [line width=0.5pt,color=black] (3,0) -- (1,0);
\draw (0,0) circle (1pt);
\draw (2,3) circle (1pt);
\draw (3,3) circle (1pt);
\draw (3,2) circle (1pt);
\begin{scriptsize}
\draw [fill=black] (0,1) circle (1pt);
\draw[color=black] (-0.4,1) node {$(0,1)$};
\draw [fill=black] (1,0) circle (1pt);
\draw[color=black] (1,-0.2) node {$(1,0)$};
\draw [fill=black] (0,3) circle (1pt);
\draw[color=black] (-0.1,3.2) node {$(0,3)$};
\draw [fill=black] (1,3) circle (1pt);
\draw[color=black] (1.1,3.2) node {$(1,3)$};
\draw [fill=black] (1,2) circle (1pt);
\draw [fill=black] (1,1) circle (1pt);
\draw [fill=black] (2,1) circle (1pt);
\draw [fill=black] (3,1) circle (1pt);
\draw[color=black] (3.4,1) node {$(3,1)$};
\draw [fill=black] (3,0) circle (1pt);
\draw[color=black] (3,-0.2) node {$(3,0)$};
\draw [fill=black] (0,2) circle (1pt);
\draw [fill=black] (2,2) circle (1pt);
\draw [fill=black] (2,0) circle (1pt);
\end{scriptsize}
\begin{large}
\draw[color=color1] (0.5,1) node {\ding{182}};
\draw[color=color2] (0.5,2) node {\ding{183}};
\draw[color=color2] (1,2.5) node {\ding{183}};
\draw[color=color3] (2.5,1) node {\ding{184}};
\draw[color=color3] (1.5,0.5) node {\ding{184}};
\end{large}
\end{tikzpicture}
\end{multicols}
\caption{The arrangement of tropical hypersurfaces of the polynomials from the system $(\mathcal{S}_1)$ and the associated subdivision of $Q$.}
\label{fig:sys1}
\end{figure}
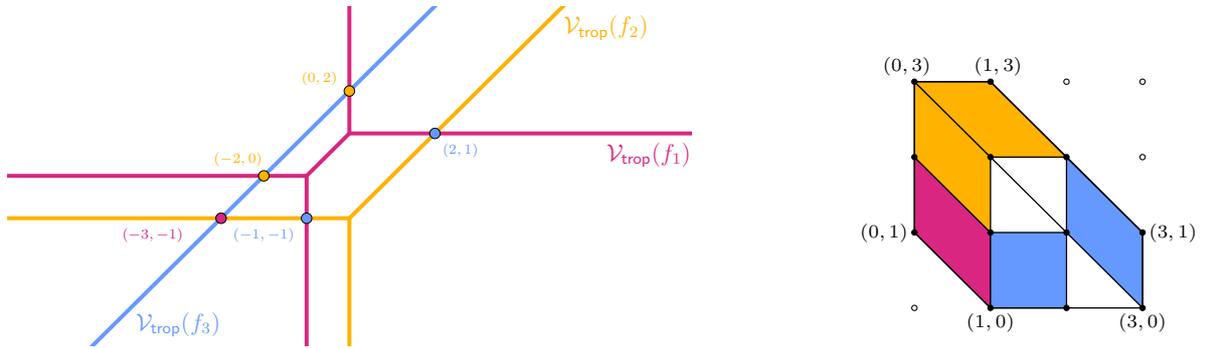

With a similar argument to the first system, one can check that $(-3,-1)$ is the only common root of the system $(\mathcal{S}_2)$ (see \Cref{fig:sys2}), and indeed by choosing
\[
y = \ver(-3,-1) = \begin{pmatrix}
0\\ -3\\ -1\\ -6\\ -4\\ -2
\end{pmatrix} \enspace ,
\]
we observe that
\[
\mac^{(2)}_{\calE} \tdot y \bal \zero \enspace .
\]
Moreover, note that the set of solutions $y \in \R^6$ to the tropical linear system $\mac_\calE^{(2)} \tdot y \bal \zero$ consists precisely in the set $\{ \lambda + \ver(-3, -1) : \lambda \in \R \}$ of tropical multiples of the Veronese embedding \eqref{eq:veronese} of the point $(-3, -1)$, which indeed attests to the uniqueness of the solution $(-3, -1)$, as two distinct solutions would have two non-collinear Veronese embeddings, in the tropical sense.
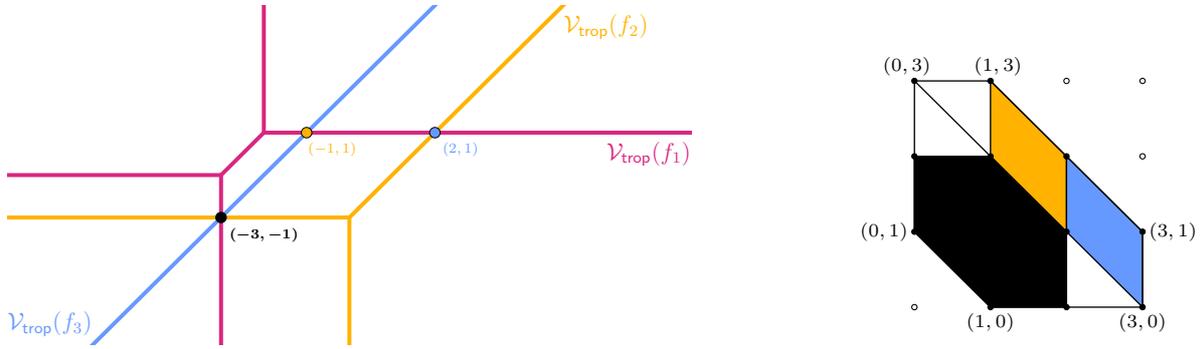
\begin{figure}[h]
\centering
\begin{multicols}{2}
\scalebox{0.75}{
\begin{tikzpicture}[line cap=round,line join=round,>=triangle 45,x=1cm,y=1cm,scale=0.75]
\clip(-8,-4) rectangle (8,4);
\draw [line width=2pt,color=color3,domain=-12:12] plot(\x,{(--2--1*\x)/1});
\draw [line width=2pt,color=color2,domain=0:12] plot(\x,{(-9--9*\x)/9});
\draw [line width=2pt,color=color2] (0,-1) -- (0,-8);
\draw [line width=2pt,color=color2,domain=-12:0] plot(\x,{(--6-0*\x)/-6});
\draw [line width=2pt,color=color1] (-3,0) -- (-12,0);
\draw [line width=2pt,color=color1] (-3,0) -- (-3,-8);
\draw [line width=2pt,color=color1] (-3,0) -- (-2,1);
\draw [line width=2pt,color=color1] (-2,1) -- (-2,8);
\draw [line width=2pt,color=color1] (-2,1) -- (12,1);
\begin{scriptsize}
\draw [fill=color2] (-1,1) circle (3.5pt);
\draw[color=color2] (-0.4,0.6) node {$(-1,1)$};
\draw [fill=color3] (2,1) circle (3.5pt);
\draw[color=color3] (2.6,0.6) node {$(2,1)$};
\draw [fill=black] (-3,-1) circle (3.5pt);
\draw[color=black] (-2,-1.4) node {$\boldsymbol{(-3,-1)}$};
\end{scriptsize}
\begin{large}
\draw[color=color1] (7,0.5) node {$\tvar(f_1)$};
\draw[color=color2] (6,3.5) node {$\tvar(f_2)$};
\draw[color=color3] (-7,-3.5) node {$\tvar(f_3)$};
\end{large}
\end{tikzpicture}
}
\flushright
\begin{tikzpicture}[line cap=round,line join=round,>=triangle 45,x=1cm,y=1cm]
\clip(-1,-1) rectangle (4,4);
\draw[line width=0.5pt,fill=color2,fill opacity=0.1] (1,2) -- (1,3) -- (2,2) -- (2,1) -- cycle;
\draw[line width=0.5pt,fill=color3,fill opacity=0.1] (2,2) -- (2,1) -- (3,0) -- (3,1) -- cycle;
\draw[line width=0.5pt,fill=black,fill opacity=0.1] (0,2) -- (1,2) -- (2,1) -- (2,0) -- (1,0) -- (0,1) -- cycle;
\draw [line width=0.5pt,color=black] (0,3) -- (1,2);
\draw [line width=0.5pt,color=black] (1,0) -- (0,1);
\draw [line width=0.5pt,color=black] (0,1) -- (0,3);
\draw [line width=0.5pt,color=black] (0,3) -- (1,3);
\draw [line width=0.5pt,color=black] (1,3) -- (3,1);
\draw [line width=0.5pt,color=black] (3,1) -- (3,0);
\draw [line width=0.5pt,color=black] (3,0) -- (1,0);
\draw (0,0) circle (1pt);
\draw (2,3) circle (1pt);
\draw (3,3) circle (1pt);
\draw (3,2) circle (1pt);
\begin{scriptsize}
\draw [fill=black] (0,1) circle (1pt);
\draw[color=black] (-0.4,1) node {$(0,1)$};
\draw [fill=black] (1,0) circle (1pt);
\draw[color=black] (1,-0.2) node {$(1,0)$};
\draw [fill=black] (0,3) circle (1pt);
\draw[color=black] (-0.1,3.2) node {$(0,3)$};
\draw [fill=black] (1,3) circle (1pt);
\draw[color=black] (1.1,3.2) node {$(1,3)$};
\draw [fill=black] (1,2) circle (1pt);
\draw [fill=black] (1,1) circle (1pt);
\draw [fill=black] (2,1) circle (1pt);
\draw [fill=black] (3,1) circle (1pt);
\draw[color=black] (3.4,1) node {$(3,1)$};
\draw [fill=black] (3,0) circle (1pt);
\draw[color=black] (3,-0.2) node {$(3,0)$};
\draw [fill=black] (0,2) circle (1pt);
\draw [fill=black] (2,2) circle (1pt);
\draw [fill=black] (2,0) circle (1pt);
\end{scriptsize}
\begin{large}
\draw[color=color2] (1.5,2) node {\ding{183}};
\draw[color=color3] (2.5,1) node {\ding{184}};
\end{large}
\end{tikzpicture}
\end{multicols}
\caption{The arrangement of tropical varieties of the polynomials from the system $(\mathcal{S}_2)$ and the associated subdivision of $Q$.}
\label{fig:sys2}
\end{figure}
\end{expl}

\begin{rmk}\label{rmk:tight} This improves on Grigoriev and Podolskii's Tropical Dual Nullstellensatz from \cite[Theorem 3.3 (i)]{GP18}, which requires $N \geq (n+2)(d_1 + \cdots + d_k)$.
Moreover, under the condition that the Newton polytope of $f$ is full-dimensional, and when $k=n+1$, we retrieve the classical Macaulay bound $N \geq d_1 + \cdots + d_{n+1}-n$ (see \cite{Laz81,Laz83,Giu84}).

In \cite[\S 4.6]{GP18}, the authors provide for all degree $d \geq 2$ and all number $n \geq 2$ of variables the following family of $n+1$ polynomials of degree at most $d$
\[
\begin{array}{rcl}
f_1 &=& 0 \tplus 0x_1\\
f_i &=& 0x_{i-1}^d \tplus 0x_i, \quad 2 \leq i \leq n\\
f_{n+1} &=& 0 \tplus 1x_n
\end{array}
\]
and show that the linearized system with truncation degree $N = (n-1)(d-1)$
\[
\mac_{(n-1)(d-1)} \tdot y \bal \zero
\]
has a solution in $\R^{N+n \choose n}$ while the polynomial system does not have a solution in $\R^n$, showing that our improved bound is tight, as in this example, our bound yields
\[
d_1 + \cdots + d_{n+1} - n = 1 + (n-1)d + 1 - n = (n-1)(d-1) + 1 \enspace .
\]

However, in non-square cases, the lower bound in \Cref{cor:null} is not necessarily optimal.

For instance, in the case of $k > n+1$ degree one polynomials, for all $1 \leq i \leq k$, the tropical polynomial function associated to $f_i$ is simply a tropical affine function
\[
(x_1, \ldots, x_n) \mapsto f_{i0} \tplus f_{i1}x_1 \tplus \cdots \tplus f_{in}x_n \enspace ,
\]
and thus $x = (x_1, \ldots, x_n) \in \R^n$ is a common root of $f_1, \ldots, f_k$ if and only if
\[
\begin{pmatrix}
f_{10} & f_{11} & \cdots & f_{1n}\\
\vdots & \vdots & & \vdots\\
f_{k0} & f_{k1} & \cdots & f_{kn}
\end{pmatrix} \tdot \begin{pmatrix}
1\\
x_1\\
\vdots\\
x_n
\end{pmatrix} \bal \zero \enspace ,
\]
and thus the collection of polynomials $f$ has a common root if and only if the matrix $(f_{ij})_{\substack{1 \leq i \leq k\\ 0 \leq j \leq n}}$ has an element $(y_0,\ldots,y_n) \in \R^{n+1}$ in its right null space, in which case $(y_1 - y_0, \ldots, y_n - y_0) \in \R^n$ is a common root of $f$. But this matrix corresponds to the submatrix of the Macaulay matrix $\mac$ obtained by taking $N = 1$ as the truncation degree for the Macaulay matrix, while the previous bound gives $N = d_1 + \cdots + d_k - n = k - n$.
\end{rmk}

\begin{rmk}\label{rmk:non-toric}
Although \Cref{thm:nullstellensatz} only deals with the toric case, \emph{i.e.} only accounts for solutions in $x \in \R^n$, one can still use it to deal with the non-toric case and find solutions in $x \in \T^n$: for any subset $I$ of $\{ 1, \ldots, n \}$, if $x \in \T^n$ is a solution of a tropical polynomial system such that
\[
\left\{\begin{array}{ll}
x_i \neq \zero &\textrm{for all } i \in I\\
x_i = \zero &\textrm{for all } i \in J := \{ 1, \ldots, n \} \setminus I \enspace ,
\end{array}\right.
\]
then $x_I := (x_i)_{i \in I} \in \R^I$ is a root of the tropical polynomial system obtained by removing all the monomials in which the variables $X_j$ for $j \in J$ appear.

Note however that in \cite[Theorems 3.3 (ii) and 4.20]{GP18}, it is shown that the linearization remains valid with $-\infty$ but at the price of an exponential blow up of the truncation degree, which becomes $N = 2(n+2)^2k(4d)^{\min(n,k)+2}$ , and thus for practical applications, enumerating the $2^n$ possible supports of a solution leads to a faster method.
\end{rmk}

\begin{rmk}
The assumption that the considered Canny-Emiris set $\calE$ is nonempty is needed because it is possible to find systems both with and without a common root for which the empty set is a Canny-Emiris subset associated to the system. For instance for $n=3$ and $k=2$, consider the system
\[
(S_1) : \left\{\begin{array}{rcl}
f_1 &=& 0 \tplus 0x_1 \tplus 0x_2 \tplus 0x_3\\
f_2 &=& 0 \enspace .
\end{array}\right.
\]
In this case, $Q$ is simply the tetrahedron with vertices $(0,0,0)$, $(0,0,1)$, $(0,1,0)$ and $(1,0,0)$. Now if we take $\delta = (\varepsilon,\varepsilon,\varepsilon)$ for $\varepsilon > 0$ small enough, we obtain a Canny-Emiris set $\calE := (Q+\delta) \cap \Z^n$ which is empty, as illustrated in \Cref{fig:poly}.
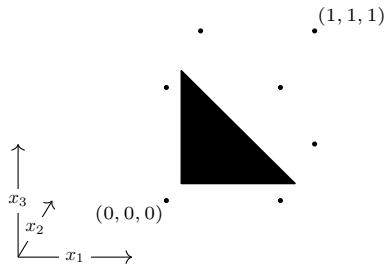
\begin{figure}[h]
\centering
\begin{tikzpicture}[x={(1cm,0cm)},y={(0.3cm,0.5cm)},z={(0cm,1cm)},scale=1.5]
\draw[->] (-1,-1,0) -- node[fill=white,scale=0.7] {$x_1$} ++ (1,0,0);
\draw[->] (-1,-1,0) -- node[fill=white,scale=0.7] {$x_2$} ++ (0,1,0);
\draw[->] (-1,-1,0) -- node[fill=white,scale=0.7] {$x_3$} ++ (0,0,1);

\begin{scriptsize}
\draw[fill=black] (0,0,0) circle (0.5pt);
\draw (-0.25,-0.25,0) node {$(0,0,0)$};
\draw[fill=black] (0,0,1) circle (0.5pt);
\draw (0,1,0) circle (0.5pt);
\draw[fill=black] (0,1,1) circle (0.5pt);
\draw[fill=black] (1,0,0) circle (0.5pt);
\draw[fill=black] (1,0,1) circle (0.5pt);
\draw[fill=black] (1,1,0) circle (0.5pt);
\draw[fill=black] (1,1,1) circle (0.5pt);
\draw (1.25,1.25,1) node {$(1,1,1)$};
\end{scriptsize}

\draw [line width=0.5pt] (0.1,0.1,0.1) -- (1.1,0.1,0.1);
\draw [line width=0.5pt,dashed] (0.1,0.1,0.1) -- (0.1,1.1,0.1);
\draw [line width=0.5pt] (0.1,0.1,0.1) -- (0.1,0.1,1.1);
\draw [line width=0.5pt,dashed] (1.1,0.1,0.1) -- (0.1,1.1,0.1);
\draw [line width=0.5pt,dashed] (0.1,1.1,0.1) -- (0.1,0.1,1.1);
\draw [line width=0.5pt] (0.1,0.1,1.1) -- (1.1,0.1,0.1);

\fill[line width=2pt,color=black,fill=black,fill opacity=0.1] (0.1,0.1,0.1) -- (1.1,0.1,0.1) -- (0.1,0.1,1.1) -- cycle;
\end{tikzpicture}
\caption{The polytope $Q+\delta$ with $\varepsilon=0.1$ for the system $(S_1)$.}
\label{fig:poly}
\end{figure}
Of course, the system $(S_1)$ does not have a common root since $f_2$ is a constant.

Now, still for $n=3$ and $k=2$, consider the system
\[
(S_2) : \left\{\begin{array}{rcl}
f_1 &=& 0 \tplus 0x_1 \tplus 0x_3\\
f_2 &=& 0 \tplus 0x_2 \enspace .
\end{array}\right.
\]
Now, $Q$ is the triangular prism with vertices $(0,0,0)$, $(1,0,0)$, $(0,0,1)$, $(0,1,0)$, $(1,1,0)$ and $(0,1,1)$. Once again, if we take $\delta = (\varepsilon, \varepsilon, \varepsilon)$ for $\varepsilon > 0$ small enough, we obtain a Canny-Emiris set which is empty, as illustrated in \Cref{fig:poly2}.
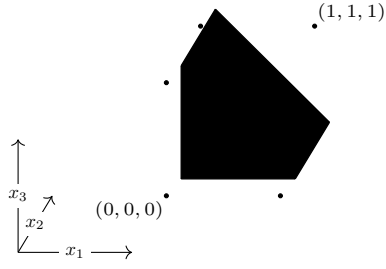
\begin{figure}[h]
\centering
\begin{tikzpicture}[x={(1cm,0cm)},y={(0.3cm,0.5cm)},z={(0cm,1cm)},scale=1.5]
\draw[->] (-1,-1,0) -- node[fill=white,scale=0.7] {$x_1$} ++ (1,0,0);
\draw[->] (-1,-1,0) -- node[fill=white,scale=0.7] {$x_2$} ++ (0,1,0);
\draw[->] (-1,-1,0) -- node[fill=white,scale=0.7] {$x_3$} ++ (0,0,1);

\begin{scriptsize}
\draw[fill=black] (0,0,0) circle (0.5pt);
\draw (-0.25,-0.25,0) node {$(0,0,0)$};
\draw[fill=black] (0,0,1) circle (0.5pt);
\draw (0,1,0) circle (0.5pt);
\draw[fill=black] (0,1,1) circle (0.5pt);
\draw[fill=black] (1,0,0) circle (0.5pt);
\draw[fill=black] (1,0,1) circle (0.5pt);
\draw (1,1,0) circle (0.5pt);
\draw[fill=black] (1,1,1) circle (0.5pt);
\draw (1.25,1.25,1) node {$(1,1,1)$};
\end{scriptsize}

\draw [line width=0.5pt] (0.1,0.1,0.1) -- (1.1,0.1,0.1);
\draw [line width=0.5pt,dashed] (0.1,1.1,0.1) -- (1.1,1.1,0.1);
\draw [line width=0.5pt] (0.1,0.1,0.1) -- (0.1,0.1,1.1);
\draw [line width=0.5pt,dashed] (0.1,1.1,0.1) -- (0.1,1.1,1.1);
\draw [line width=0.5pt] (0.1,0.1,1.1) -- (1.1,0.1,0.1);
\draw [line width=0.5pt] (0.1,1.1,1.1) -- (1.1,1.1,0.1);
\draw [line width=0.5pt,dashed] (0.1,0.1,0.1) -- (0.1,1.1,0.1);
\draw [line width=0.5pt] (1.1,0.1,0.1) -- (1.1,1.1,0.1);
\draw [line width=0.5pt] (0.1,0.1,1.1) -- (0.1,1.1,1.1);

\fill[line width=2pt,color=black,fill=black,fill opacity=0.1] (0.1,0.1,0.1) -- (1.1,0.1,0.1) -- (1.1,1.1,0.1) -- (0.1,1.1,1.1) -- (0.1,0.1,1.1) -- cycle;
\end{tikzpicture}
\caption{The polytope $Q+\delta$ with $\varepsilon=0.1$ for the system $(S_2)$.}
\label{fig:poly2}
\end{figure}
This time however, the system $(S_2)$ has a common root, namely $(0,0,0)$.

Therefore, it is \textit{a priori} not possible to conclude if the considered Canny-Emiris set is taken empty, although excluding the case where some of the polynomials of the system are monomials is rather degenerate. Excluding this particular case, one can wonder if it is then possible to reach a conclusion in the case of systems where there exists a empty associated Canny-Emiris set.
\end{rmk}

\subsection{Preliminary results}\label{sec:prel}

In this section, we state and prove a number of lemmas which will be used in order to prove \Cref{thm:nullstellensatz}.

\subsubsection{Nonsingular and diagonally dominant tropical matrices}

We first recall the definition of nonsingularity and diagonal dominance for tropical matrices. Set two integers $p, q \in \Npos$.

\begin{dfn}
Let $A = (a_{ij})_{(i, j) \in [p] \times [q]}$ be a $p \times q$ tropical matrix. Then the matrix $A$ is said to be \textit{tropically nonsingular} whenever the only solution to the equation $A \tdot y \bal \zero$ of unknown $y \in \T^q$ is $y = \zero$.
\end{dfn}

\begin{rmk}\label{rmk:tdet}
In the case where $A \in \T^{p \times p}$ is a square matrix, one can consider its \textit{tropical determinant} $\tdet(A)$ which is given by
\[
\tdet(A) = \max_{\sigma \in \mathfrak{S}_p} a_{1\sigma(1)} + \cdots + a_{p\sigma(p)} \enspace .
\]
If the maximum in the previous expression is attained exactly once, hence if $\tdet(A) \nbal \zero$, then as a direct consequence of Corollary 6.12 of \cite{AGG08}, the only possible solution to the equation $A \tdot y \bal \zero$ is $y = \zero$, \textit{i.e.} the matrix $A$ is tropically nonsingular.
\end{rmk}

\begin{dfn}
A matrix $A = (a_{ij})_{(i, j) \in [p] \times [p]} \T^{p \times p}$ is said to be \textit{weakly diagonally dominant} in the tropical sense whenever we have the inequalities
\[
a_{ii} \geq a_{ij} \quad \textrm{for all} \quad 1 \leq i, j \leq p \text{ such that } i \neq j \enspace ,
\]
and we say that it is \textit{(strictly) diagonally dominant} if these inequalities are strict.
\end{dfn}

\begin{rmk}
This notion of tropical diagonal dominance just corresponds to the tropical version of classical diagonal dominance, as the inequality of the above definition is equivalent to
\[
a_{ii} \geq \tsum_{1 \leq j \neq i \leq p} a_{ij} = \max_{1 \leq j \neq i \leq p} a_{ij} \quad \textrm{for all} \quad 1 \leq i \leq p \enspace .
\]
\end{rmk}

A notable fact about diagonally dominant tropical matrices, which will play a crucial role in the proof of \Cref{thm:nullstellensatz}, is the fact that similarly to classical diagonally dominant matrices, these matrices are non-singular. More precisely, we have the following result.

\begin{lem}\label{lem:ddns}
Let $A \in \T^{p \times p} = (a_{ij})_{(i, j) \in [p] \times [p]}$ be a diagonally dominant tropical matrix. Then $A$ is tropically nonsingular.
\end{lem}

\begin{proof}
Let $y = (y_1, \ldots, y_p) \in \T^p$ be such that $A \tdot y \bal \zero$, and consider $1 \leq i \leq p$ such that $y_i = \max_{1 \leq j \leq p} y_j$. Then from the relation $A \tdot y \bal \zero$, it follows in particular that the maximum in the expression
\[
\max_{1 \leq j \leq p} (a_{ij} + y_j)
\]
is attained twice, but since for all $1 \leq j \leq p$, we have
\[
a_{ii} > a_{ij} \quad \textrm{and} \quad y_i \geq y_j \enspace ,
\]
the only possible way such that the maximum in the previous expression is attained twice is that
\[
y_i = -\infty
\]
and thus
\[
y = \zero \enspace . \qedhere
\]
\end{proof}

\begin{rmk}
Alternatively, one can retrieve the previous lemma with the following argument: since $A$ is diagonally dominant, it means that the maximum in the expression
\[
\tdet(A) = \max_{\sigma \in \mathfrak{S}_d} a_{1\sigma(1)} + \cdots + a_{d\sigma(d)}
\]
is attained exactly once, hence $\tdet(A) \nbal \zero$, and thus by the previous remark, $A$ is tropically nonsingular.
\end{rmk}

Finally we will also make use of the following two lemmas.

\begin{lem}\label{lem:nsg1}
Let $A = (a_{ij})_{(i,j) \in [p] \times [q]}$ be a $p \times q$ tropical matrix. Fix for $1 \leq j \leq q$, $\varepsilon_j \in \R$, and set $\widetilde{A} = (\widetilde{a}_{ij})_{(i,j) \in [p] \times [q]} \in \T^{p \times q}$ with $\widetilde{a}_{ij} = a_{ij} + \varepsilon_j$ for all $1 \leq i \leq p$ and $1 \leq j \leq m$. Then $A$ is tropically nonsingular if and only if $\widetilde{A}$ is tropically nonsingular.
\end{lem}

\begin{proof}
Assume that $A$ is nonsingular, and let $\widetilde{y} = (\widetilde{y}_j)_{1 \leq j \leq q} \in \T^q$ be such that $\widetilde{A} \tdot \widetilde{y} \bal \zero$. Then, this means that for all $1 \leq i \leq p$, the maximum in the expression
\[
\max_{1 \leq j \leq q} \left( \widetilde{a}_{ij} + \widetilde{y}_j \right) = \max_{1 \leq j \leq q} \left( a_{ij} + (\widetilde{y}_j + \varepsilon_j) \right)
\]
is attained twice. In other words, setting $y = (\widetilde{y}_j + \varepsilon_j)_{1 \leq j \leq q}$, we obtain that $A \tdot y \bal \zero$. Therefore, by nonsingularity of $A$, we must have $y = \zero$, and since $\varepsilon_j$ is finite for all $1 \leq j \leq q$, this implies that $\widetilde{y} = \zero$, hence $\widetilde{A}$ is tropically nonsingular.

As for the converse implication, it is obtained by swapping $A$ and $\widetilde{A}$, and by changing $(\varepsilon_j)_{1 \leq j \leq q}$ to $(-\varepsilon_j)_{1 \leq j \leq m}$.
\end{proof}

\begin{lem}\label{lem:nsg2}
Let $A$ be a $p \times q$ tropical matrix, and assume that $A$ can be written by block as a lower-triangular matrix
\[
A = \begin{pmatrix}
A^{(m)} & \zero\\
\ast & \ast
\end{pmatrix}
\]
with $A^{(m)}$ a $m \times m$ square submatrix with $0 < m \leq p, q$. Moreover, assume that $A^{(m)}$ is tropically nonsingular. Then the equation $A \tdot y \bal \zero$ of unknown $y$ has no solution in $\R^q$.
\end{lem}

\begin{proof}
Let $y = (y_j)_{1 \leq j \leq q} \in \T^q$ be such that $A \tdot y \bal \zero$. Then by setting $y^{(m)} = (y_j)_{1 \leq j \leq m} \in \T^m$, we obtain in particular, by looking at the first $m$ rows of the product $A \tdot y$, that $A^{(m)} \tdot y^{(m)} \bal \zero$, which implies by nonsingularity of $A^{(m)}$ that $y^{(m)} = \zero$ and thus $y$ does not belong to $\R^q$.
\end{proof}

\subsubsection{Generalities on sup-convolution and Minkowski sums}

In order to write the proof of \Cref{thm:nullstellensatz}, we also need to introduce the following definition, which simply corresponds to the tropical equivalent of the convolution product.
The \textit{support} of a function $h : \R^n \to \R \cup \{\pm \infty \}$ is defined
by $\supp(h)=\cl(\{ x \in \R^n : h(x) > -\infty \})$.
The {\em hypograph} of $h$ is the set $\hypo(h)=\{(x,t)\in \R^n\times \R : t\leq h(x)\}$.
\begin{dfn}
The \textit{Minkowski sum} of two subsets $E,F\subset \R^n$
is defined as $E+F=\{x+y : x\in E,\,y\in F\}$.
  The \textit{sup-convolution} is the binary operator $\supco$ defined for all functions $f, g : \R^n \to \R\cup\{\pm\infty\}$ by
\[
f \supco g(x) = \sup_{y+z=x} f(y) + g(z) \enspace ,
\]
with the convention $(-\infty)+(+\infty)=-\infty$.
\end{dfn}
In particular, if $f$ and $g$ are upper semicontinuous, take
values in $\R \cup \{ -\infty \}$, and have compact support, then, the
supremum in the expression of $f\supco g(x)$ is achieved and $f \supco g$ also has compact support. The operations of sup-convolution and Minkowski sum are commutative and associative, and that we have the following immediate properties:
\begin{prt}\label{prt:conv}
  Let $E_1,\dots,E_\ell$ be a collection of subsets of $\R^n$ and let $h_1, \ldots, h_\ell$ be a family of upper semi-continuous functions with compact support from $\R^n$ to $\R\cup\{\pm\-\infty\}$. 
  Let $E=E_1+\dots+E_\ell$ and $h = h_1 \supco \cdots \supco h_\ell$.
Then,
\begin{enumerate}[label = (\alph*)]
\item $E=\{x_1+\dots +x_\ell : x_i\in E_i \, \textrm{ for all } \; 1 \leq i \leq \ell\}$;
\item For all $q \in \R^n$, $h(q) = \max_{q_1 + \cdots + q_k = q} h_1(q_1) + \cdots + h_\ell(q_\ell)$;
\item $\hypo(h) = \hypo(h_1)+\dots +\hypo(h_\ell)$ and $\supp(h) = \supp(h_1)+\dots +\supp(h_\ell)$;
\item For $1 \leq i \leq \ell$, let $\widehat{h}_i := h_1 \supco \cdots \supco h_{i-1} \supco h_{i+1} \supco \cdots \supco h_\ell$. Then we have $h_i \supco \widehat{h}_i = h$.
\end{enumerate}
\end{prt}

\begin{dfn}
The {\em normal cone} $N_x(P)$ to a polyhedron $P\subset \R^n$ \textit{at a point} $x \in P$ is defined by
\[
N_x(P) =  \{y \in \R^n : \<y,z-x>\leq 0,\,\forall z\in P\} \enspace .
\]
The \textit{normal cone} $N_F(P)$ to $P$ \textit{at a face} $F$ of $P$ is defined as $N_x(P)$ where $x$ is an arbitrary point in the relative interior of $F$ --- this object is well-defined because you can easily check that two points in the relative interior of the same face of $P$ yield the same normal fan. Moreover, the collection
\[
\mathscr{N}(P) = \{ N_F(P) : F \textrm{ face of } P \}
\]
of normal cones at every face of $P$ is called the \textit{normal fan} of $P$.
\end{dfn}

\begin{rmk}
Note that the normal fan of a polyhedron $P \subseteq \R^n$ is the same as the normal fan of any translation $P+b$ of $P$. More precisely, if $b \in \R^n$, then the map $F \mapsto F+b$ gives us a correspondance between the faces of $P$ and the faces of $P+b$, and moreover we can easily show that $N_{F+b}(P+b) = N_F(P)$
\end{rmk}

The normal fan $\scrN(P)$ of a polyhedron forms a lattice for the partial order given by the inclusion, and is therefore refered to as the \textit{face lattice} of $P$. Moreover, this lattice is endowed with a grading given by the dimension. The same property holds for the set $\scrF(P)$ of faces of $P$, except that the order and the grading are reversed. In fact, these two lattices are related by the following duality result.

\begin{prp}
Let $P$ be a polyhedron in $\R^n$. Then the map
\[
\begin{array}{rcl}
\mathscr{F}(P) & \longto & \mathscr{N}(P)\\
F & \longmapsto & N_F(P)
\end{array}
\]
is a poset anti-isomorphism from the lattice of faces of $P$ to the lattice of normal cones of $P$. In particular, we have $\dim(F) + \dim(N_F(P)) = n$.
\end{prp}

\begin{proof}
This is a quite standard result which is stated, amongst others, in Section 7 of \cite{Tho06}.
\end{proof}

We deduce the following corollary from the previous proposition, which will help us deal with the case where the polyhedron we are working with is not full-dimensional.

\begin{cor}\label{cor:dico}
Let $P$ be a polyhedron in $\R^n$ and let $W$ be the vector subspace of $\R^n$ directing the affine hull of $P$. Then for any face $F$ of $P$, one has
\[
\dim(N_F(P) \cap W) = \dim(P) - \dim(F) \enspace .
\]
\end{cor}

\begin{proof}
Since translating a polyhedron does not affect its normal fan, we can assume without loss of generality that $P$ is included in $W$. Then we notice that for any $x \in P$,
\[
N_x(P) \cap W = \{ y \in W : \<y, z-x> \leq 0, \forall z \in P \} \enspace ,
\]
This means that $N_x(P) \cap W$ corresponds to the normal cone at point $x$ of the polyhedron $P$ seen as a full dimensional polyhedron of $W$. We can then apply the previous proposition and we obtain that
\[
\dim(F) + \dim(N_F(P) \cap W) = \dim(W) = \dim(P) \enspace .
\]
\end{proof}

\begin{dfn}
A concave function $\R^n\to\R \cup \{ -\infty \}$ is said to be \textit{polyhedral}
if its hypograph is a (closed) polyhedron.
We say that a convex subset $F$ of $\R^n\times \R$ is {\em vertical} if $(0, \ldots, 0, 1)$ belongs to the vector space directing the affine hull of $F$.
\end{dfn}

\begin{expl}\label{expl:poly}
The main example of concave polyhedral functions that will be of interest in this paper are the functions obtained by taking the concavification of the coefficient map $\omega : \R^n \to \R \cup \{ -\infty \}$ --- that is the infimum of all concave functions greater than or equal to $\omega$ --- defined by
\[
\omega(\alpha) =  \left\{\begin{array}{ll}
f_\alpha & \textrm{if} \quad \alpha \in \supp(f),\\
\zero & \textrm{otherwise}.
\end{array}\right.
\]
of a tropical polynomial $f$. These functions satisfy in particular the property that the projection of the singularities of their graph onto $\R^n \times \{ 0 \}$ is a rational polyhedral complex. Moreover, the projection of 
vertices of their hypograph are elements of $\supp (f)$.
\end{expl}

For all functions $h : \R^n \to \R \cup \{ -\infty \}$, and for all $x\in \R^n$, we set
\[
\cell (x,h) := \argmax_{q\in \R^n} ( \<q,x>+h(q)) \quad \text{and} \quad \face (x,h) := \{ (q,h(q)) : q\in \cell(x,h)\} \enspace .
\]

We make the following crucial observations.
\begin{obs}\label{rmk:remk}\ 
\begin{enumerate}[label=(\alph*)]
\item \label{rmk:remk-a}
If $h$ is a concave polyhedral function nonidentically $-\infty$,
then $\face(x,h)$ is the face of the hypograph of $h$, which is
obtained as the intersection of this hypograph
with a supporting hyperplane of outer normal vector $(x,1)$.
In particular, this face is non-vertical, and hence it is a proper face of $\hypo(h)$, \textit{i.e.} $\face(x,h) \subsetneq \hypo(h)$.
\item \label{rmk:remk-b} If $h$ is the concavification of the coefficient map $\omega$ as in \Cref{expl:poly}, then $\cell(x,h)$ is the convex hull of
\[ \cell(x,\omega)= \argmax_{\alpha \in \Z^n} (f_\alpha + \<x, \alpha> )\enspace ,\]
which coincides with the intersection of $\cell(x,h)$ with the elements
$\alpha\in  \supp(f)$ such that $h(\alpha)=f_\alpha$.
Then, $\face(x,h)$ is the convex hull of $\face(x,\omega)$.
We also have that $\face(x,h)$ is the convex hull of its intersection
with the set of vertices of $\hypo(h)$.

\item \label{rmk:remk-c}
Moreover, when $h$ is the concavification of the coefficient map $\omega$, if $F$ is a non-vertical face of $\hypo(h)$, then $F = \face(x,h)$ if and only if $(x,1)$ is in the relative interior of $N_F(\hypo(h))$. In particular, if $F$ is a facet of $\hypo(h)$, then there exists a unique vector $x \in \R^n$ such that $F = \face(x,h)$ and $x$ is in the vector space directing the affine hull of $Q := \supp(h)$. Indeed, from \Cref{cor:dico}, we have $\dim(N_F(\hypo(h)) \cap W) = 1$ where $W = V \times \R$ is the vector space directing the affine hull of $\hypo(h)$, and thus if $(x,1)$ and $(x',1)$ are both in the half-line $N_F(\hypo(h)) \cap W$, then it follows that $x' = x$.
\end{enumerate}
\end{obs}

We now state a useful lemma on convex polyhedra.

\begin{lem}
Let $P_1, \ldots, P_\ell$ be a finite collection of convex polyhedra, and denote by $P$ their Minkowski sum $P_1 + \cdots + P_\ell$. Let $F$ be a face of $P$ and let $y$ be in the relative interior of $N_F(P)$. If $p = p_1 + \cdots + p_\ell \in F$ with $p_i \in P_i$ for all $1 \leq i \leq \ell$, then for all $1 \leq i \leq \ell$, $y \in N_{p_i}(P_i)$, \textit{i.e.} $p_i \in \argmax_{p'_i \in P_i} \<p'_i, y>$.
\end{lem}

\begin{proof}
Saying that $y$ is in the relative interior of $N_F(P)$ is equivalent to saying that $F = \argmax_{p' \in P} \<p', y>$. Therefore, you have
\[
\begin{array}{rcl}
\<p_1, y> + \cdots + \<p_\ell, y> &=& \<p, y>\\
&=& \max_{p' \in P} \<p', y>\\
&=& \max_{p'_1 \in P_1, \ldots, p'_\ell \in P_\ell} \left( \<p'_1, y> + \cdots + \<p'_\ell, y> \right)\\
&=& \max_{p'_1 \in P_1} \<p'_1, y> + \cdots + \max_{p'_\ell \in P_\ell} \<p'_\ell, y> \enspace ,
\end{array}
\]
therefore $\<p_i, y> = \max_{p'_i \in P_i} \<p'_i, y>$ for all $1 \leq i \leq \ell$, \textit{i.e.} $y \in N_{p_i}(P_i)$ for all $1 \leq i \leq \ell$. 
\end{proof}

\begin{rmk}
In particular, from the previous lemma, by setting $F_i := \argmax_{p'_i \in P_i} \<p'_i, y>$, we obtain that the decomposition of $F$ as a sum of faces of the $P_i$ is precisely $F = F_1 + \cdots + F_\ell$.
\end{rmk}

The following result is a direct corollary of the previous lemma.

\begin{cor}\label{cor:faces}
Let $h_1,\dots,h_\ell$ denote concave non-identically $-\infty$ polyhedral functions from $\R^n$ to $\R \cup \{ -\infty \}$. Consider the sup-convolution
\[
h = h_1 \supco \cdots \supco h_\ell
\enspace .
\]
Let $F$ be a non-vertical face of $\hypo(h)$, and let $x \in \R^n$ be such that $(x,1)$ is in the relative interior of the normal cone of $\hypo(h)$ at this face. Consider a point $(q,\lambda) \in F$ such that
\[
(q,\lambda) = (q_1, \lambda_1) + \cdots + (q_\ell, \lambda_\ell) \quad \textrm{with} \quad (q_i, \lambda_i) \in \hypo(h_i) \textrm{ for all } 1 \leq i \leq \ell \enspace .
\]
Then for all $1 \leq i \leq \ell$, $(q_i, \lambda_i) \in \face(x,h_i)$.

In particular,
\[
F=\face(x,h) =\face(x,h_1)+\dots+\face(x,h_\ell)
\]
and moreover, for all $1\leq i\leq \ell$, $(x,1)$ is in the normal cone of $\hypo(h_i)$ at point $(q_i, \lambda_i)$.
\end{cor}

\begin{proof}
We apply the previous lemma for $P_i = \hypo(h_i)$ for all $1 \leq i \leq \ell$ and $P = \hypo(h)$. Since $(x,1)$ is in the relative interior of $N_F(P)$, then we have $F = \face(x,h)$ by definition, and moreover from the previous lemma, we have $(x,1) \in N_{(q_i, \lambda_i)}(P_i)$, or equivalently $(q_i, \lambda_i) \in \face(x,h_i)$ for all $1 \leq i \leq \ell$. The remainder follows immediately.
\end{proof}

\subsection{Proving the Tropical Nullstellensatz}

\subsubsection{The Canny-Emiris construction}\label{sec:caem}

In the following, we assume given the polynomials $f_1,\ldots,f_k$, a nonempty Canny-Emiris subset $\calE$ associated to the system $f$, and a subset $\calE'$ of $\Z^n$ containing $\calE$. We can now describe the construction of Canny and Emiris from \cite{CE93} and \cite{Emi05}, which was generalized by Sturmfels in part 3 of \cite{Stu94}, which we will apply to the system $f = (f_1, \ldots, f_k)$ in the particular case where the polynomials $f_i$ do not share a common root, in order to prove \Cref{thm:nullstellensatz}. In particular, contrary to the previous constructions, we do not assume that the coefficients satisfy any genericity condition. This construction is illustrated on an example in \Cref{expl:canny-emiris-construction} below.

First of all, let us settle some notations for the rest of this section. For all $1 \leq i \leq k$, let $\calA_i$ denote the support of $f_i$,  $Q_i$ the convex hull of $\calA_i$, $Q := Q_1 + \cdots + Q_k$. The Canny-Emiris subset $\calE$  of $\Z^n$ is such that $\calE:= \Z^n \cap (Q + \delta)$, for some generic vector $\delta \in V + \Z^n$, where $V \subseteq \R^n$ is the vector space directing the affine hull of $Q$.
Let $\omega = (\omega_i)_{1 \leq i \leq k}$ be the collection of coefficient maps of the $f_i$, that is defined by
\[
\omega_i(\alpha) = \left\{\begin{array}{ll}
f_{i,\alpha} & \textrm{if} \quad \alpha \in \calA_i\\
\zero & \textrm{else},
\end{array}\right.
\]
and let $h_i : \R^n \to \R \cup \{ -\infty \}$ denote the concave hull of $\omega_i$. Then we can consider the following liftings of the polytopes $Q, Q_1, \ldots, Q_k$. Let $Q^\mathsf{lift}_i := \hypo(h_i)$ (this is refered to by Grigoriev and Podolskii in \cite{GP18} as the \textit{extended Newton polytope} of $f_i$) and $Q^\mathsf{lift} := Q^\mathsf{lift}_1 + \cdots + Q^\mathsf{lift}_k = \hypo(h)$ with $h := h_1 \supco \cdots \supco h_k$.

\begin{obs}\label{obs:ess}
Note that for all $1 \leq i \leq k$, by construction of the maps $h_i$, we have
\[
h_i(\alpha) \geq f_{i,\alpha} \quad \forall \alpha \in \calA_i \enspace ,
\]
with equality whenever the monomial $f_{i,\alpha} X^\alpha$ is \textit{essential} in $f_i$, that is whenever there exists a point $x \in \R^n$ at which the monomial of exponent $\alpha$ is the only one achieving the maximum over all the monomials of $f_i$. Equivalently,
the monomial $f_{i,\alpha}X^\alpha$ of $f_i$ is essential if, and only if,
the point $(\alpha,f_{i,\alpha})$ is a vertex of $\hypo(h_i)$.
Note also that if $(q,h(q))$ is an extreme point of $Q^\mathsf{lift}$, then $h(q)$ corresponds to the coefficient of $X^q$ in the product $f_1 \cdots f_k$.
\end{obs}

Now let us apply the Canny-Emiris construction, to the collection $\omega$ of maps --- note that as opposed to what is done in \cite{Stu94}, $\omega$ is given and might not be generic. The projection of $Q^\mathsf{lift}$ onto $Q$ induces a mixed coherent subdivision $\Delta_{\omega}$ of $Q$, given by the points of non-differentiability of $h$.
The following observation follows readily from the genericity of $\delta$.
\begin{obs}\label{obs:construct-x}
Given $p \in \calE$, the couple $(p-\delta,h(p-\delta))$ lies in the relative interior of a unique non-vertical facet $F$ of $Q^\mathsf{lift}$, or equivalently, $p-\delta$ lies in the relative interior of the associated cell $C$ of $Q$, obtained as the image of $F$ by the projection mapping. Moreover, there is a unique $x \in V$ such that $F = \face(x, h)$ and $C = \cell(x,h)$.
\end{obs}

We construct a matrix $\mac_{\mathcal{EE}'}$, whose rows are indexed by $\calE$ and whose columns are indexed by $\calE'$. We proceed as follows. First, for $p \in \calE$, consider the associated $x \in V \subseteq \R^n$ and $F$ as above, and for all $1 \leq i \leq k$, let $F_i$ denote the face $\face(x, h_i)$ of $Q^\mathsf{lift}_i$, and $C_i$ the cell $\cell(x,h_i)$ of the associated subdivision of $Q_i$. Then by \Cref{cor:faces}, we have
\[
F = F_1 + \cdots + F_k \quad \textrm{and} \quad C = C_1 + \cdots + C_k \enspace .
\]

Moreover, under the assumption that the tropical polynomials $f_1, \ldots, f_k$ do not share a common root, we know that at least one of the $C_i$ is a singleton, and we let $j$ be the maximal index such that $C_j = \{ a_j \}$ is a singleton. We call the couple $(j, a_j)$ the \textit{row content} of $p$.
Note that this row content only depends on the cell $C$ containing $p - \delta$ in its interior. Then since $p - \delta \in C$, thus it can be written as
\begin{equation}\label{eqn:qsum}
p - \delta = q_1 + \cdots + q_j + \cdots + q_k \quad \textrm{with} \quad \left\{ \begin{array}{rcl}
q_i & \in & C_i \quad \textrm{for all} \quad 1 \leq i \neq j \leq k\\
q_j & = & a_j \enspace .
\end{array}\right.
\end{equation}
We then construct the matrix $\mac_{\mathcal{EE}'}$ as follows : for every $p \in \calE$, we associate to $p$ its row content $(j, a_j)$, and then we put the row $(j, p-a_j)$ of the Canny-Emiris submatrix $\mac_{\calE'}$ in $\mac_{\mathcal{EE}'}$, whose coefficients are given by the coefficients of the polynomial $X^{p-a_j}f_j(X)$. Note that all exponents appearing in the support of $X^{p-a_j}f_j(X)$ belong to $\calE$ and thus to $\calE'$. Indeed, if $a'_j \in \calA_j$, then
\[
\begin{array}{rcl}
p-a_j+a'_j &=& \delta + q_1 + \cdots + q_j + \cdots + q_k - a_j + a'_j\\
&=& \delta + q_1 + \cdots + a'_j + \cdots + q_k \in (Q+\delta) \cap \Z^n = \calE \subseteq \calE' \enspace .
\end{array}
\]
Thus, we end up with a matrix $\mac_{\mathcal{EE}'} = \left(m_{pp'}\right)_{(p,p') \in \calE \times \calE'}$ indexed by $\calE \times \calE'$. We shall show in the proof of \Cref{thm:nullstellensatz} that this matrix $\mac_{\mathcal{EE}'}$ is actually a submatrix of $\mac_{\calE'}$ --- it might not be the case because the rows selected by the row content are \textit{a priori} not necessarily distinct. Also notice that by grouping together the columns of $\mac_{\mathcal{EE}'}$ indexed by $\calE$ and by $\calE' \setminus \calE$, the matrix $\mac_{\mathcal{EE}'}$ can be written as a block matrix
\[
\mac_{\mathcal{EE}'} = \begin{pmatrix} \mac_{\mathcal{EE}} & \zero \end{pmatrix}
\]
where $\mac_{\mathcal{EE}} = (m_{pp'})_{(p,p') \in \calE \times \calE}$ is a square matrix indexed by $\calE \times \calE$, which we shall show is a submatrix of $\mac_{\calE}$.

\subsubsection{The proof of \Cref{thm:nullstellensatz}}\label{sec:proof-nullstellensatz}

In order to prove \Cref{thm:nullstellensatz}, we will make use of the following lemma.

\begin{lem}\label{lem:ddiag}
Consider $h$ and $h_1, \ldots, h_k$ as defined as in \Cref{sec:caem}, and moreover for all $1 \leq j \leq k$, set $\widehat{h}_j = h_1 \supco \cdots \supco h_{j-1} \supco h_{j+1} \supco \cdots \supco h_k$. Let $p \in \calE$ and let $(j, a_j)$ be its row content. Then for all $p' \in \calE$ and $a'_j \in \Z^n$ such that $p' = p - a_j + a'_j$, we have
\begin{equation}
h(p'-\delta) \geq h_j(a'_j) + \widehat{h}_j(p-\delta-a_j),
\label{eqn:ndiag}
\end{equation}
with equality if and only if $p'=p$ and $a'_j = a_j$.
\end{lem}

\begin{proof}
  Since $p' - \delta = a'_j + (p - \delta - a_j) $,
  the inequality~\eqref{eqn:ndiag} follows from the definition of sup-convolution, noting that $h=h_j \supco \widehat{h}_j$.

We next show that the equality holds if $p'=p$, which entails that $a'_j=a_j$. Indeed, since
\[
(p-\delta, h(p-\delta)) \in Q^\mathsf{lift} = \hypo(h_j) + \hypo(\widehat{h}_j) \enspace ,
\]
we can write 
\begin{equation}\label{eqn:faces}
(p-\delta, h(p-\delta)) = (q_j, \lambda_j) + (\widehat{q}_j, \widehat{\lambda}_j)
\end{equation}
with $(q_j, \lambda_j) \in \hypo(h_j)$ and $(\widehat{q}_j, \widehat{\lambda}_j) \in \hypo(\widehat{h}_j)$. Then by \Cref{cor:faces}, defining $x$ as in~\Cref{obs:construct-x},
it follows that $(q_j, \lambda_j) \in \face(x,h_j)$. In particular, since $(j,a_j)$ is the row content of $p$, we have $\face(x,h_j) = \{ (a_j, h_j(a_j)) \}$, hence,
\[
(q_j, \lambda_j) = (a_j, h_j(a_j)) \enspace .
\]
\Cref{obs:construct-x} also entails that $(\widehat{q}_j, \widehat{\lambda}_j) \in \face(x,\widehat{h}_j)$
and therefore that $(\widehat{q}_j, \widehat{\lambda}_j) = (\widehat{q}_j, \widehat{h}_j(\widehat{q}_j))$. Moreover, since $\widehat{q}_j = p-\delta-q_j = p-\delta-a_j$, it follows that
\[
(\widehat{q}_j, \widehat{\lambda}_j) = (p-\delta-a_j, \widehat{h}_j(p-\delta-a_j)) \enspace .
\]
Hence, we deduce from (\ref{eqn:faces}) that
\begin{equation}\label{eqn:diag}
h(p-\delta) = h_j(a_j) + \widehat{h}_j (p - \delta - a_j) \enspace .
\end{equation}

Keeping the same $p$ as above, we now show that (\ref{eqn:ndiag}) is strict whenever $p' \neq p$. Indeed, assume that the equality is achieved. Then this implies that
\[
(p'-\delta, h(p'-\delta)) = (a'_j, h_j(a'_j)) + (p - \delta - a_j, \widehat{h}_j(p - \delta - a_j)) \enspace .
\]
Now consider $x' \in \R^n$ such that $F' = \face(x',h)$ is the facet in the interior of which $(p'-\delta, h(p'-\delta))$ lies. Then from \Cref{cor:faces}, we have
\[
(a'_j, h_j(a'_j)) \in \face(x', h_j) \quad \textrm{and} \quad (p - \delta - a_j, \widehat{h}_j(p - \delta - a_j)) \in \face(x', \widehat{h}_j) \enspace .
\]
However, we also know from equality (\ref{eqn:diag}) and \Cref{cor:faces} that
\[
(p - \delta - a_j, \widehat{h}_j(p - \delta - a_j)) \in \face(x, \widehat{h}_j) \enspace ,
\]
and since $\face(x, h_j) = \{ (a_j, h_j(a_j)) \}$ is a singleton, then $\face(x, \widehat{h}_j)$ is simply a translation of $\face(x, h)$, and moreover since $(p-\delta, h(p-\delta))$ is in the relative interior of $\face(x, h)$, then it means that $(p - \delta - a_j, \widehat{h}_j(p - \delta - a_j))$ is in the relative interior of $\face(x, \widehat{h}_j)$, and that $\face(x, \widehat{h}_j)$ is a facet of $\widehat{Q}^\mathsf{lift}_j := \hypo(\widehat{h}_j)$. Therefore, since $(p-\delta-a_j, \widehat{h}_j(p-\delta-a_j))$ is in both $\face(x, \widehat{h}_j)$ and $\face(x', \widehat{h}_j)$ and since it is in particular in the relative interior of the first face, we deduce that
\[
\face(x, \widehat{h}_j) \subseteq \face(x', \widehat{h}_j) \ \enspace .
\]
Since by \Cref{rmk:remk} (a), $\face(x', \widehat{h}_j)$ is a proper face of $\widehat{Q}^\mathsf{lift}_j$, then it implies that it is also a facet of $\widehat{Q}^\mathsf{lift}_j$, and thus that
\[
\face(x, \widehat{h}_j) = \face(x', \widehat{h}_j) \enspace .
\]
Hence, it follows from \Cref{rmk:remk} (b) that $x = x'$, and therefore $(a'_j,h_j(a'_j)) \in \face(x, h_j) = \{ (a_j,h_j(a_j)) \}$, hence $a'_j = a_j$ and thus $p' = p$.
\end{proof}

Using \Cref{lem:ddiag}, we are now all set to give a proof of \Cref{thm:nullstellensatz}.

\begin{proof}[Proof of \Cref{thm:nullstellensatz}]
The `only if' implication is straightforward, for if $x$ is a common root of the system $f$, then its image by the Veronese embedding below
\begin{equation}\label{eq:veronese}
\ver : \left\{\begin{array}{rcl}
\R^n & \longto & \R^{\calE'}\\
x & \longmapsto & \ver(x) = (x^{\nu})_{\nu \in \calE'}
\end{array}\right.
\end{equation}
yields a finite vector in the right null space of the matrix $\mac_{\calE'}$. Note that by convention, this also holds when $\mac_{\calE'}$ is empty.

For the converse implication, we rather show the contrapositive. Assume that the tropical polynomial system $f$ does not have a solution, and consider the matrix $\mac_{\mathcal{EE}'} = (m_{pp'})_{(p,p') \in \calE \times \calE'}$ obtained from the Canny-Emiris construction described in \Cref{sec:caem}.
Note that this construction implies that the matrix $\mac_{\mathcal{EE}}$ and \textit{a fortiori}
the matrix $\mac_{\calE}$ are nonempty, leading to a contradiction when $\mac_{\calE}$ is empty, thus proving the `if' implication in that case.
We now assume that $\mac_{\calE}$ and thus $\mac_{\calE'}$ are nonempty.

Let
\[
\widetilde{\mac}_{\mathcal{EE}'} = \left(\widetilde{m}_{pp'}\right)_{(p,p') \in \calE \times \calE'} \quad \textrm{with} \quad \widetilde{m}_{pp'} = m_{pp'} - h(p'-\delta) \enspace .
\]
Similarly to $\mac_{\mathcal{EE}'}$, this matrix can also be written as a block matrix in the following way:
\[
\widetilde{\mac}_{\mathcal{EE}'} = \begin{pmatrix} \widetilde{\mac}_{\mathcal{EE}} & \zero \end{pmatrix}
\]
where $\widetilde{\mac}_{\mathcal{EE}} = (\widetilde{m}_{pp'})_{(p,p') \in \calE \times \calE}$ is a square matrix indexed by $\calE \times \calE$. We show that the tropical matrix $\widetilde{\mac}_{\mathcal{EE}}$ is diagonally dominant. Indeed, for all $p, p' \in \calE$, and let $(j, a_j)$ be the row content of $p$. Then, we have for $a'_j \in \Z^n$ such that $p' = p - a_j + a'_j$
\[
m_{pp'} = \left\{\begin{array}{ll}
f_{j,a'_j} & \textrm{if} \quad a'_j \in \calA_j\\
\zero & \textrm{otherwise},
\end{array}\right.
\]
or in other words, $m_{pp'} = \omega_j(a'_j)$.
In particular, for $p'=p$, we have
\[
\widetilde{m}_{pp} = f_{j,a_j} - h(p-\delta) \enspace .
\]
Moreover, by assumption, the monomial $f_{j,a_j} X^{a_j}$ of $f_j$ is essential. By \Cref{obs:ess}, we have therefore $h_j(a_j) = f_{j,a_j}$, hence we obtain from the equality case of \Cref{lem:ddiag} that
\[
\widetilde{m}_{pp} = -\widehat{h}_j(p - \delta - a_j) \enspace .
\]
Now for the case of $p' \neq p$, again since $h_j(a'_j) \geq f_{j,a'_j}$, then by rearranging the inequality of \Cref{lem:ddiag}, we obtain that
\[
\widetilde{m}_{pp} = -\widehat{h}_j(p - \delta - a_j) \geq f_{j,a'_j} - h(p' - \delta) = \widetilde{m}_{pp'} \enspace ,
\]
with equality if and only if $p' = p$.
Thus, the matrix $\widetilde{\mac}_{\mathcal{EE}}$ is tropically diagonally dominant. In particular, it follows that any two rows of $\widetilde{\mac}_{\mathcal{EE}'}$ are distinct, and thus it is also true for $\mac_{\mathcal{EE}'}$ because $\widetilde{\mac}_{\mathcal{EE}'}$ was obtained from $\mac_{\mathcal{EE}'}$ by adding the same vector to all of its rows. Therefore, it means that the matrix $\mac_{\mathcal{EE}'}$ is indeed a submatrix of $\mac_{\calE'}$ --- and likewise the rows of $\mac_{\mathcal{EE}}$ are pairwise distinct, thus $\mac_{\mathcal{EE}}$ is a (square) submatrix of $\mac_{\calE}$.

Finally, since $\widetilde{\mac}_\mathcal{EE}$ is tropically diagonally dominant, by \Cref{lem:ddns}, it also is nonsingular, and thus applying \Cref{lem:nsg1} to $\mac_{\mathcal{EE}}$, we obtain that $\mac_{\mathcal{EE}}$ is also nonsingular. Therefore, the matrix $\mac_{\calE'}$ has the following form
\[
\mac_{\calE'} = \begin{pmatrix}
\mac_{\mathcal{EE}} & \zero\\
\ast & \ast
\end{pmatrix}
\]
and thus \Cref{lem:nsg2} entails that the equation $\mac_{\calE'} \tdot y \bal \zero$ has no solution $y$ in $\R^{\calE'}$, \textit{i.e.} there cannot exist a finite vector in the tropical right null space of $\mac_{\calE'}$.
\end{proof}

\begin{rmk}
Notice that in the case where $\calE' = \calE$, following the same proof, you obtain a weaker condition on the right null space of $\mac_{\calE}$ in order to find a finite solution of the system $f$. More precisely, you can show that there exists a solution $x \in \R^n$ to the system $f$ if and only if there exists a vector $y \in \T^{\calE} \setminus \{ \zero \}$ such that $\mac_{\calE} \tdot y \bal \zero$. In other words, in the case where $\calE' = \calE$, the existence of a nonzero vector in the right null space of $\mac_{\calE}$, even possibly with some coordinates equal to $\zero$, is enough to guarantee the existence of a finite solution of the system $f$.
\end{rmk}

\begin{expl}\label{expl:canny-emiris-construction}
We illustrate the use of the Canny-Emiris construction in our proof by applying it to system
\[ (\mathcal{S}_1) : \left\{\begin{array}{rcl}
f_1 &=& 1 \tplus 2x_1 \tplus 1x_2 \tplus 1x_1x_2\\
f_2 &=& 0 \tplus 0x_1 \tplus 1x_2\\
f_3 &=& 2x_1 \tplus 0x_2 \enspace ,
\end{array}\right. \]
which was shown in \Cref{expl:system-nullstellensatz} not to have any solution. First, we obtain a subdivision of $Q$ by projecting the non-vertical faces of the Minkowski sum of the hypographs of $h_1$, $h_2$ and $h_3$ onto the horizontal hyperplane $\R^n \times \{ 0 \}$ as shown on \Cref{fig:proj}.
\begin{figure}[H]
\centering
\begin{tikzpicture}[x={(0.8cm,0.25cm)},y={(-0.4cm,0.4cm)},z={(0cm,0.9cm)},scale=0.7]
\draw[->] (-3,-3,-3) -- node[fill=white,scale=0.7] {$x_1$} ++ (2,0,0);
\draw[->] (-3,-3,-3) -- node[fill=white,,scale=0.7] {$x_2$} ++ (0,2,0);
\draw[->] (-3,-3,-3) -- node[fill=white,scale=0.7] {$h(x)$} ++ (0,0,2);
\draw [line width=0.5pt,color=color1] (0,0,1) -- (0,1,1);
\draw [line width=0.5pt,color=color1] (0,0,1) -- (1,0,2);
\draw [line width=0.5pt,color=color1] (0,1,1) -- (1,0,2);
\draw [line width=0.5pt,color=color1,densely dashed] (0,1,1) -- (1,1,1);
\draw [line width=0.5pt,color=color1,densely dashed] (1,0,2) -- (1,1,1);
\draw [line width=0.5pt,color=color1,loosely dotted] (0,0,1) -- (0,0,-1);
\draw [line width=0.5pt,color=color1,loosely dotted] (0,1,1) -- (0,1,-1);
\draw [line width=0.5pt,color=color1,loosely dotted] (1,0,2) -- (1,0,-1);
\draw [line width=0.5pt,color=color1,loosely dotted] (1,1,1) -- (1,1,-1);
\draw [line width=0.2pt,draw=color1,fill=color1,fill opacity=0.1] (0,0,-1) -- (0,1,-1) -- (1,0,-1) -- cycle;
\draw [line width=0.2pt,draw=color1,fill=color1,fill opacity=0.1] (1,1,-1) -- (0,1,-1) -- (1,0,-1) -- cycle;
\draw[color=color1] (0.5,0.5,-3) node {$Q_1$};
\draw[color=color1] (0.5,0.5,3) node {$Q^\mathsf{lift}_1$};

\draw [line width=0.5pt,color=color2] (3,0,0) -- (3,1,1);
\draw [line width=0.5pt,color=color2] (3,0,0) -- (4,0,0);
\draw [line width=0.5pt,color=color2] (3,1,1) -- (4,0,0);
\draw [line width=0.5pt,color=color2,loosely dotted] (3,0,0) -- (3,0,-1);
\draw [line width=0.5pt,color=color2,loosely dotted] (3,1,1) -- (3,1,-1);
\draw [line width=0.5pt,color=color2,loosely dotted] (4,0,0) -- (4,0,-1);
\draw [line width=0.2pt,draw=color2,fill=color2,fill opacity=0.1] (3,0,-1) -- (3,1,-1) -- (4,0,-1) -- cycle;
\draw[color=color2] (3.5,0.5,-3) node {$Q_2$};
\draw[color=color2] (3.5,0.5,3) node {$Q^\mathsf{lift}_2$};

\draw [line width=0.5pt,color=color3] (6,1,0) -- (7,0,2);
\draw [line width=0.5pt,color=color3,loosely dotted] (6,1,0) -- (6,1,-1);
\draw [line width=0.5pt,color=color3,loosely dotted] (7,0,2) -- (7,0,-1);
\draw [line width=0.2pt,color=color3] (6,1,-1) -- (7,0,-1);
\draw[color=color3] (6.5,0.5,-3) node {$Q_3$};
\draw[color=color3] (6.5,0.5,3) node {$Q^\mathsf{lift}_3$};

\draw [line width=0.5pt,color=black] (9,1,1) -- (9,2,2);
\draw [line width=0.5pt,color=black] (9,2,2) -- (10,1,4);
\draw [line width=0.5pt,color=black] (10,1,4) -- (10,0,3);
\draw [line width=0.5pt,color=black] (10,0,3) -- (9,1,1);
\draw [line width=0.5pt,color=black] (9,2,2) -- (9,3,2);
\draw [line width=0.5pt,color=black] (9,3,2) -- (10,2,4);
\draw [line width=0.5pt,color=black] (10,2,4) -- (10,1,4);
\draw [line width=0.5pt,color=black] (10,1,4) -- (9,2,2);
\draw [line width=0.5pt,color=black] (10,2,4) -- (9,3,2);
\draw [line width=0.5pt,color=black,densely dashed] (9,3,2) -- (10,3,2);
\draw [line width=0.5pt,color=black,densely dashed] (10,3,2) -- (11,2,4);
\draw [line width=0.5pt,color=black,densely dashed] (11,2,4) -- (10,2,4);
\draw [line width=0.5pt,color=black,densely dashed] (11,2,4) -- (12,1,3);
\draw [line width=0.5pt,color=black,densely dashed] (12,1,3) -- (12,0,4);
\draw [line width=0.5pt,color=black] (12,0,4) -- (11,1,5);
\draw [line width=0.5pt,color=black,densely dashed] (11,1,5) -- (11,2,4);
\draw [line width=0.5pt,color=black] (10,0,3) -- (11,0,4);
\draw [line width=0.5pt,color=black] (11,0,4) -- (11,1,5);
\draw [line width=0.5pt,color=black] (11,1,5) -- (10,1,4);
\draw [line width=0.5pt,color=black] (10,1,4) -- (10,0,3);
\draw [line width=0.5pt,color=black] (11,0,4) -- (12,0,4);
\draw [line width=0.5pt,color=black] (10,2,4) -- (11,1,5);
\draw [line width=0.5pt,color=black,loosely dotted] (9,1,1) -- (9,1,-1);
\draw [line width=0.5pt,color=black,loosely dotted] (9,2,2) -- (9,2,-1);
\draw [line width=0.5pt,color=black,loosely dotted] (9,3,2) -- (9,3,-1);
\draw [line width=0.5pt,color=black,loosely dotted] (10,1,4) -- (10,1,-1);
\draw [line width=0.5pt,color=black,loosely dotted] (10,2,4) -- (10,2,-1);
\draw [line width=0.5pt,color=black,loosely dotted] (10,3,2) -- (10,3,-1);
\draw [line width=0.5pt,color=black,loosely dotted] (11,2,4) -- (11,2,-1);
\draw [line width=0.5pt,color=black,loosely dotted] (12,1,3) -- (12,1,-1);
\draw [line width=0.5pt,color=black,loosely dotted] (12,0,4) -- (12,0,-1);
\draw [line width=0.5pt,color=black,loosely dotted] (11,0,4) -- (11,0,-1);
\draw [line width=0.5pt,color=black,loosely dotted] (10,0,3) -- (10,0,-1);
\draw [line width=0.5pt,color=black,loosely dotted] (11,1,5) -- (11,1,-1);
\draw[line width=0.2pt,fill=color1,fill opacity=0.1] (9,1,-1) -- (9,2,-1) -- (10,1,-1) -- (10,0,-1) -- cycle;
\draw[line width=0.2pt,fill=color2,fill opacity=0.1] (9,2,-1) -- (9,3,-1) -- (10,2,-1) -- (10,1,-1) -- cycle;
\draw[line width=0.2pt,fill=color2,fill opacity=0.1] (9,3,-1) -- (10,3,-1) -- (11,2,-1) -- (10,2,-1) -- cycle;
\draw[line width=0.2pt,fill=color3,fill opacity=0.1] (11,2,-1) -- (12,1,-1) -- (12,0,-1) -- (11,1,-1) -- cycle;
\draw[line width=0.2pt,fill=color3,fill opacity=0.1] (10,0,-1) -- (10,1,-1) -- (11,1,-1) -- (11,0,-1) -- cycle;
\draw [line width=0.2pt,color=black] (11,0,-1) -- (12,0,-1);
\draw [line width=0.2pt,color=black] (10,2,-1) -- (11,1,-1);
\draw[color=black] (11.5,-0.5,-3) node {$Q = \textcolor{color1}{Q_1} + \textcolor{color2}{Q_2} + \textcolor{color3}{Q_3}$};
\draw[color=black] (11.5,-0.5,7) node {$Q^\mathsf{lift} = \textcolor{color1}{Q^\mathsf{lift}_1} + \textcolor{color2}{Q^\mathsf{lift}_2} + \textcolor{color3}{Q^\mathsf{lift}_3}$};
\end{tikzpicture}
\caption{The subdivision of $Q$ arises from the projection of the Minkowski sum of the hypographs of the $h_i$.}
\label{fig:proj}
\end{figure}
With this subdivision, we associate to every point $p \in \calE$ its row content $i, a_i$, which is univocally determined by the maximal-dimensional cell of the decomposition of $Q$ to which $p-\delta$ belongs. This process is illustrated in \Cref{fig:rowc}.

More precisely, in the following table, for each point $p$ of $\calE$ in the first row:
\begin{itemize}[label=$\diamond$]
\item the second row displays the monomial $x^p$ which corresponds to a column of the Macaulay matrix,
\item the third row displays the row content $i, a_i$ of $p$,
\item the fourth row displays the polynomial $x^{p-a_i}f_i$ which corresponds to a row of the Macaulay matrix,
\item and finally the last row displays the tropical scaling factor $h(p-\delta)$ which must be substracted (in the usual sense) to the column $p$ of the matrix $\mac^{(1)}_{\mathcal{EE}}$ in order to obtain the matrix $\widetilde{\mac}^{(1)}_{\mathcal{EE}}$.
\end{itemize}

\begin{figure}[H]
\centering
\begin{tabular}{|c||c c c c c c|}
\hline
$p \in \calE$ & $(0,0)$ & $(1,0)$ & $(0,1)$ & $(2,0)$ & $(1,1)$ & $(0,2)$\\
\hline
$x^p$ & $1$ & $x_1$ & $x_2$ & $x_1^2$ & $x_1x_2$ & $x_2^2$\\
\hline
$i, a_i$ & $1,(0,0)$ & $3,(1,0)$ & $2,(0,1)$ & $3,(1,0)$ & $3,(1,0)$ & $2,(0,1)$\\
\hline
$x^{p-a_i}f_i$ & $f_1$ & $f_3$ & $f_2$ & $x_1f_3$ & $x_2f_3$ & $x_2f_2$\\
\hline
$h(p-\delta)$ & $3.6$ & $4.8$ & $3.8$ & $3.3$ & $4.1$ & $2.2$\\
\hline
\end{tabular}
\end{figure}

\begin{figure}[h]
\centering
\begin{tikzpicture}[line cap=round,line join=round,>=triangle 45,x=1cm,y=1cm,scale=1.5]
\clip(-1,-1) rectangle (4,4);
\fill[line width=1pt,color=black,fill=color1,fill opacity=0.15] (0,2) -- (1,1) -- (1,0) -- (0,1) -- cycle;
\fill[line width=1pt,color=black,fill=color2,fill opacity=0.15] (0,3) -- (1,2) -- (1,1) -- (0,2) -- cycle;
\fill[line width=1pt,color=black,fill=color2,fill opacity=0.15] (0,3) -- (1,3) -- (2,2) -- (1,2) -- cycle;
\fill[line width=1pt,color=black,fill=color3,fill opacity=0.15] (2,2) -- (2,1) -- (3,0) -- (3,1) -- cycle;
\fill[line width=1pt,color=black,fill=color3,fill opacity=0.15] (1,1) -- (2,1) -- (2,0) -- (1,0) -- cycle;
\draw [line width=0.5pt] (1,2) -- (2,1);
\draw [line width=0.5pt] (1,0) -- (0,1);
\draw [line width=0.5pt] (0,1) -- (0,3);
\draw [line width=0.5pt] (0,3) -- (1,3);
\draw [line width=0.5pt] (1,3) -- (3,1);
\draw [line width=0.5pt] (3,1) -- (3,0);
\draw [line width=0.5pt] (3,0) -- (1,0);
\draw [line width=1pt] (0,2) -- (1,1);
\draw [line width=1pt] (1,1) -- (1,0);
\draw [line width=1pt] (1,0) -- (0,1);
\draw [line width=1pt] (0,1) -- (0,2);
\draw [line width=1pt] (0,3) -- (1,2);
\draw [line width=1pt] (1,2) -- (1,1);
\draw [line width=1pt] (1,1) -- (0,2);
\draw [line width=1pt] (0,2) -- (0,3);
\draw [line width=1pt] (0,3) -- (1,3);
\draw [line width=1pt] (1,3) -- (2,2);
\draw [line width=1pt] (2,2) -- (1,2);
\draw [line width=1pt] (1,2) -- (0,3);
\draw [line width=1pt] (2,2) -- (2,1);
\draw [line width=1pt] (2,1) -- (3,0);
\draw [line width=1pt] (3,0) -- (3,1);
\draw [line width=1pt] (3,1) -- (2,2);
\draw [line width=1pt] (1,1) -- (2,1);
\draw [line width=1pt] (2,1) -- (2,0);
\draw [line width=1pt] (2,0) -- (1,0);
\draw [line width=1pt] (1,0) -- (1,1);

\begin{scriptsize}
\draw [fill=black] (0.9,0.9) circle (1pt);
\draw[color=black] (0.65,0.75) node {$1,(0,0)$};
\draw [fill=black] (0.9,1.9) circle (1pt);
\draw[color=black] (0.65,1.75) node {$2,(0,1)$};
\draw [fill=black] (0.9,2.9) circle (1pt);
\draw[color=black] (0.65,2.75) node {$2,(0,1)$};
\draw [fill=black] (1.9,0.9) circle (1pt);
\draw[color=black] (1.65,0.75) node {$3,(1,0)$};
\draw [fill=black] (1.9,1.9) circle (1pt);
\draw[color=black] (1.65,1.75) node {$3,(1,0)$};
\draw [fill=black] (2.9,0.9) circle (1pt);
\draw[color=black] (2.65,0.75) node {$3,(1,0)$};
\end{scriptsize}

\draw (-0.1,-0.1) circle (1pt);
\draw (-0.1,0.9) circle (1pt);
\draw (-0.1,1.9) circle (1pt);
\draw (-0.1,2.9) circle (1pt);
\draw (0.9,-0.1) circle (1pt);
\draw (1.9,-0.1) circle (1pt);
\draw (1.9,2.9) circle (1pt);
\draw (2.9,-0.1) circle (1pt);
\draw (2.9,1.9) circle (1pt);
\draw (2.9,2.9) circle (1pt);
\end{tikzpicture}
\caption{The polytope $Q+\delta$, with the integer points inside the maximal dimensional cells of the decomposition of $Q+\delta$ labelled by the row content the cell they belong to.}
\label{fig:rowc}
\end{figure}

With the information from the previous table, we obtain the following $6 \times 6$ square submatrix $\mac^{(1)}_{\mathcal{EE}}$ of $\mac^{(1)}_{\mathcal{E}}$
\[
\mac^{(1)}_{\mathcal{EE}} = \begin{blockarray}{*{7}{c}}
& 1 & x_1 & x_2 & x_1^2 & x_1x_2 & x_2^2\\
\begin{block}{c(*{6}{c})}
f_1 & 1 & 2 & 1 & & 1 &\\
f_3 & & 2 & 0 & & &\\
f_2 & 0 & 0 & 1 & & &\\
x_1f_3 & & & & 2 & 0 &\\
x_2f_3 & & & & & 2 & 0\\
x_2f_2 & & & 0 & & 0 & 1\\
\end{block}
\end{blockarray} \enspace ,
\]
and after applying the tropical scaling of factor $h(p-\delta)$ to the column $p$ of the previous matrix for all $p \in \calE$, we obtain the following matrix, in which we highlighted the diagonal coefficients by writing them in bold
\[
\widetilde{\mac}^{(1)}_{\mathcal{EE}} = \begin{blockarray}{*{7}{c}}
& 1 & x_1 & x_2 & x_1^2 & x_1x_2 & x_2^2\\
\begin{block}{c(*{6}{c})}
f_1 & \boldsymbol{-2.6} & -2.8 & -2.8 & & -3.1 &\\
f_3 & & \boldsymbol{-2.8} & -3.8 & & &\\
f_2 & -3.6 & -4.8 & \boldsymbol{-2.8} & & &\\
x_1f_3 & & & & \boldsymbol{-1.3} & -4.1 &\\
x_2f_3 & & & & & \boldsymbol{-2.1} & -2.2\\
x_2f_2 & & & -3.8 & & -4.1 & \boldsymbol{-1.2}\\
\end{block}
\end{blockarray} \enspace .
\]
We finally observe that the matrix $\widetilde{\mac}^{(1)}_{\mathcal{EE}}$ is tropically diagonally dominant, and therefore it follows from Lemmas \ref{lem:ddns} and \ref{lem:nsg1} that $\mac^{(1)}_{\mathcal{EE}}$ is also nonsingular\footnote{Alternatively, one can verify that $\tdet(\mac^{(1)}_{\mathcal{EE}}) \nbal \zero$ if $\tdet$ denotes the tropical determinant defined in \Cref{rmk:tdet}.}, thus there is no solution $y \in \R^{\calE}$ to the equation $\mac^{(1)}_{\calE} \tdot y \bal \zero$.

Besides, note that by choosing a different ordering of the polynomials, in which $f_3$ comes before $f_2$, then the row content of $(1,1)$ would be $2,(1,0)$ and thus the fifth row of $\mac^{(1)}_{\mathcal{EE}}$ would change to be
\[
\begin{blockarray}{*{7}{c}}
& 1 & x_1 & x_2 & x_1^2 & x_1x_2 & x_2^2\\
\begin{block}{c(*{6}{c})}
x_1f_2 & & 0 & & 0 & 1 &\\
\end{block}
\end{blockarray}
\]
and we can still check likewise that the right null space of this matrix is simply equal to $\{ \zero \}$.
\end{expl}

\section{The Tropical Positivstellensatz for two-sided systems}

\subsection{Statement of the theorem}\label{sec:pos}

Now we formulate a result similar to \Cref{thm:nullstellensatz}, for systems mixing equalities and weak and strict inequalities between tropical polynomial functions (also called `two sided' systems). More precisely, consider $f^+ = (f_1^+, \ldots, f_k^+)$ and $f^- = (f_1^-, \ldots, f_k^-)$ two collections of $k$ tropical polynomials. For $1 \leq i \leq k$, we denote by $\calA_i^+$ and $\calA_i^-$ respectively the supports of $f_i^+$ and $f_i^-$, and set $\calA_i = \calA_i^+ \cup \calA_i^-$, and $\calA = (\calA_1, \ldots, \calA_k)$.
Equivalently, $\calA$ is the collection of supports of the polynomials $f_i := f^+_i \tplus f^-_i$.

Consider also a collection $\rhd=(\rhd_1,\ldots, \rhd_k)$ of relations $\rhd_i \in \{ \geq, =, > \}$, for $1\leq i\leq k$. We denote by $f^+ (x)\rhd f^-(x)$ the system
\begin{align}
\max_{\alpha \in \calA_i^+} \left( f_{i,\alpha}^+ + \<\alpha, x> \right) \rhd_i \max_{\alpha \in \calA_i^-} \left( f_{i,\alpha}^- + \<\alpha, x> \right) \, \textrm{for all} \; 1 \leq i \leq k ,
\label{eq:ineq}
\end{align}
of unknown $x \in \T^n$. 
Note that this is equivalent to the expressions $f^+_i(x) \rhd_i f^-_i(x)$, $i = 1, \ldots, k$. Finally, we denote by $\mac^+$ and $\mac^-$ the Macaulay matrices associated to $f^+$ and $f^-$ respectively, so with entries $\mac^{\pm}_{(i,\alpha), \beta}=f^{\pm}_{i,\beta-\alpha}$. Then, for any subset $\calE$ of $\Z^n$, we denote by $\mac^+_{\calE}$ and $\mac^-_{\calE}$ the submatrices associated to $\calE$ and the collection $\calA$ defined above, that is $\mac^+_{\calE}=(\mac^+)^{\calA}_{\calE}$ and $\mac^-_{\calE}=(\mac^-)^{\calA}_{\calE}$, and we likewise set for $N \in \N$, $\mac^+_N = (\mac^+)^{\calA}_N$ and $\mac^-_N = (\mac^-)^{\calA}_N$.

Finally, we set for $1 \leq i \leq k$:
\[
r_i = \left\{\begin{array}{ll}
\dim(\aff(\calA^-_i)) + 1 & \textrm{if } \rhd_i \in \{ \geq, > \}\\
\max\left(\dim(\aff(\calA^-_i)), \dim(\aff(\calA^+_i))\right) + 1 & \textrm{if } \rhd_i \in \{ = \} \enspace .
\end{array}\right.
\]

Notice that in particular, this definition implies the following inequality
\[
r_i \leq \left\{\begin{array}{ll}
\min(\vert \calA^-_i \vert, n+1) & \textrm{if } \rhd_i \in \{ \geq, > \}\\
\min(\max(\vert \calA^-_i \vert, \vert \calA^+_i \vert), n+1) & \textrm{if } \rhd_i \in \{ = \} \enspace .
\end{array}\right.
\]

We now call \textit{Canny-Emiris subsets} of $\Z^n$ associated to the system $f^+ \rhd f^-$ any set $\calE$ of the form
\[
\calE := \left( \widetilde{Q} + \delta \right) \cap \Z^n \quad \textrm{with} \quad \widetilde{Q} = r_1 Q_1 + \cdots + r_k Q_k \enspace ,
\]
where $Q_i = \conv(\calA_i)$ for $1 \leq i \leq k$, and $\delta$ is a generic vector in $V + \Z^n$, with $V$ the direction of the affine hull of $\widetilde{Q}$ (note that this is the same as for $Q_1 + \cdots + Q_k$). Finally, for any subset $\calE'$ of $\Z^n$ containing a Canny-Emiris subset $\calE$ associated to $f^+ \rhd f^-$ and $y \in \R^{\calE'}$, we denote by $\mac^+_{\calE'} \tdot y \rhd \mac^-_{\calE'} \tdot y$ the following system of tropical linear equalities:
\[
\max_{\beta \in \calE'} \left( m^+_{(i,\alpha), \beta} + y_\beta \right) \rhd_i \max_{\beta \in \calE'} \left( m^-_{(i,\alpha), \beta} + y_\beta \right) \quad \textrm{for all} \quad 1 \leq i \leq k \quad \textrm{and} \quad \alpha \in \calA_i \enspace .
\]

We now state a Positivstellensatz for the case of tropical polynomial systems allowing both weak and strict inequalities, and equalities.

\begin{thm}[Sparse tropical Positivstellensatz]
\label{thm:positivstellensatz}
There exists a solution $x \in \R^n$ to the system $f^+ \rhd f^-$ if and only if there exists a vector $y \in \R^{\calE'}$ satisfying $\mac_{\calE'}^+ \tdot y \rhd \mac_{\calE'}^- \tdot y$, where $\calE'$ is any subset of $\Z^n$ containing a nonempty Canny-Emiris subset $\calE$ of $\Z^n$ associated to the system $f^+ \rhd f^-$.
\end{thm}

\begin{rmk}\label{rmk:positivstellensatz}
The following example illustrates how the present tropical Positivstellensatz can be applied to the problem of deciding the inclusion of \textit{tropical basic semialgebraic sets}. By analogy with classical semialgebraic sets, a tropical basic semialgebraic subset of $\R^n$ is defined to be the set of solutions of a collection of inequalities
of the form $f_i^+ \rhd_i f_i^-, i\in[k]$ where $f_1^\pm, \dots,f_k^\pm$
are pairs of tropical polynomials, and $\rhd_1,\dots,\rhd_k\in \{\geq, >\}$.
Then, \Cref{thm:positivstellensatz} provides an effective way to check the inclusion of two tropical basic semialgebraic sets. Let us illustrate this by the following
typical special case: given $f^\pm_1, \ldots, f^\pm_{k+1}$  a collection of pairs of tropical polynomials, check whether the following implication holds:
\begin{equation}\label{eq:positivstellensatz}\tag{$\mathscr{P}$}
\left\{\begin{array}{rcl}
f^+_1(x) & \geq & f^-_1(x)\\
& \vdots &\\
f^+_k(x) & \geq & f^-_k(x)
\end{array}\right. \overset{}{\implies} f^+_{k+1}(x) \geq f^-_{k+1}(x) \quad \forall x \in \R^n \enspace .
\end{equation}
Checking \eqref{eq:positivstellensatz} is equivalent to showing that the following system
\[
(\mathscr{S}) : \left\{\begin{array}{rcl}
f^+_1(x) & \geq & f^-_1(x)\\
& \vdots &\\
f^+_k(x) & \geq & f^-_k(x)\\
f^+_{k+1}(x) & < & f^-_{k+1}(x)
\end{array}\right.
\]
has no solution $x \in \R^n$, which can be done using~\Cref{thm:positivstellensatz},
by reduction to the problem of the unsolvability a system of linear tropical (in) equations,
allowing both strict and weak inequalities. Recall that the latter system
reduces to a mean-payoff game~\cite[Theorem~4.7]{uli2013}, see also~\cite[Theorem~18]{AGK11b}.
\end{rmk}

\begin{expl}
Now let us illustrate \Cref{thm:nullstellensatz} with an example. Consider the following problem
\begin{equation}\tag{$\mathcal{P}$}
\label{ex:positivstellensatz}
\left\{\begin{array}{lcl}
0x_1 \tplus 0x_1x_2 & \geq & 0x_2\\
2 \tplus 0x_1 & \geq & 1x_2\\
3 & \geq & 0x_1
\end{array}\right. \quad \implies \quad 1x_1 \;\; \geq \;\; 0x_2 \tplus (-3)x_1^2 \enspace .
\end{equation}

We want to show that the implication in \eqref{ex:positivstellensatz} holds. This is the case if and only if the system
\begin{equation}\tag{$\mathcal{S}$}
\left\{\begin{array}{rcl}
0x_1 \tplus 0x_1x_2 & \geq & 0x_2\\
2 \tplus 0x_1 & \geq & 1x_2\\
3 & \geq & 0x_1\\
0x_2 \tplus (-3)x_1^2 & > & 1x_1
\end{array}\right.
\end{equation}
does not have a solution $x \in \R^2$. We can turn the strict inequality into a weak inequality by noticing that the latter system does not have a solution on $\R^2$ if and only if for all $\lambda > \unit = 0$, the system
\begin{equation}\tag{$\mathcal{S}'_\lambda$}
\label{ex:posi-lambda}
\left\{\begin{array}{rcl}
0x_1 \tplus 0x_1x_2 & \geq & 0x_2\\
2 \tplus 0x_1 & \geq & 1x_2\\
3 & \geq & 0x_1\\
0x_2 \tplus (-3)x_1^2 & \geq & \lambda \tdot 1x_1
\end{array}\right.
\end{equation}
does not have a solution $x \in \R^2$.

We rewrite the system \eqref{ex:positivstellensatz} as $\forall i \in \{1, 2, 3 \}, \, f^+_i(x) \geq f^-_i(x) \implies f^+_4(x) \geq f^-_4(x)$ by setting
\[
\begin{array}{rclcrclcrcl}
f^+_1 &=& 0x_1 \tplus 0x_1x_2 & \quad & f^-_1 &=& 0x_2 & \quad & f_1 &=& 0x_1 \tplus 0x_2 \tplus 0x_1x_2\\
f^+_2 &=& 2 \tplus 0x_1 & \quad & f^-_2 &=& 1x_2 & \quad & f_2 &=& 2 \tplus 0x_1 \tplus 1x_2\\
f^+_3 &=& 3 & \quad & f^-_3 &=& 0x_1 & \quad & f_3 &=& 3 \tplus 0x_1\\
f^+_4 &=& 1x_1 & \quad & f^-_4 &=& 0x_2 \tplus (-3)x_1^2 & \quad & f_4 &=& 1x_1 \tplus 0x_2 \tplus (-3)x_1^2
\end{array}
\]

On \Cref{fig:poly-ineq}, we show the Newton polytopes associated to the polynomials $f_1$, $f_2$, $f_3$ and $f_4$ as well as their Minkowski sum. Moreover, on Figures \ref{fig:sys-ineq-1} and \ref{fig:sys-ineq-2}, we show on the left the configuration of hypersurfaces associated respectively to \eqref{ex:positivstellensatz} and \eqref{ex:posi-lambda}. In this configuration, we colored every intersection point by the colors of the curves that do not take part in the intersection, and we reported the coloring on the matching cell of the decomposition of $Q$ on the right, as to illustrate the duality between the hypersurface configuration and the subdivision of $Q$.

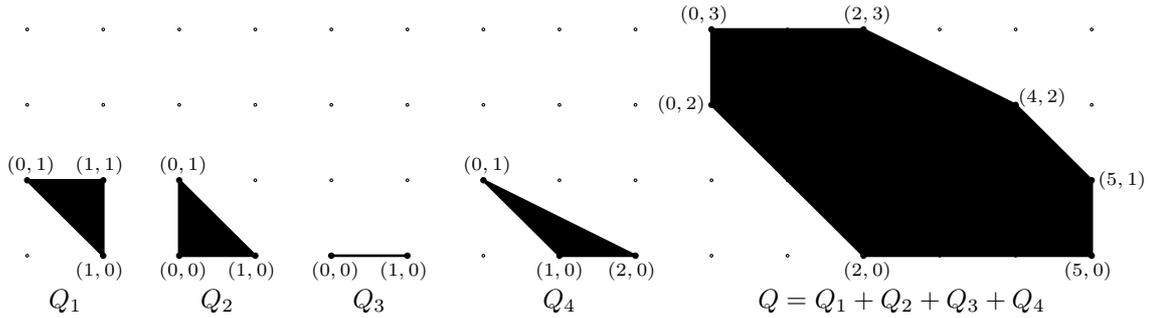
\begin{figure}[H]
\centering
\begin{tikzpicture}[line cap=round,line join=round,>=triangle 45,x=1cm,y=1cm]
\clip(-0.75,-0.75) rectangle (14.75,3.75);
\foreach \x in {-1,...,15}
    \foreach \y in {-1,...,4}
    {
    \draw (\x,\y) circle (0.5pt);
    }
\fill[line width=2pt,fill=black,fill opacity=0.1] (0,1) -- (1,0) -- (1,1) -- cycle;
\fill[line width=2pt,fill=black,fill opacity=0.1] (2,1) -- (2,0) -- (3,0) -- cycle;
\fill[line width=2pt,fill=black,fill opacity=0.1] (8,0) -- (7,0) -- (6,1) -- cycle;
\fill[line width=2pt,fill=black,fill opacity=0.1] (11,0) -- (14,0) -- (14,1) -- (13,2) -- (11,3) -- (9,3) -- (9,2) -- cycle;
\draw [line width=1pt] (0,1)-- (1,0);
\draw [line width=1pt] (1,0)-- (1,1);
\draw [line width=1pt] (1,1)-- (0,1);
\draw [line width=1pt] (2,1)-- (2,0);
\draw [line width=1pt] (2,0)-- (3,0);
\draw [line width=1pt] (3,0)-- (2,1);
\draw [line width=1pt] (5,0)-- (4,0);
\draw [line width=1pt] (8,0)-- (7,0);
\draw [line width=1pt] (7,0)-- (6,1);
\draw [line width=1pt] (6,1)-- (8,0);
\draw [line width=1pt] (11,0)-- (14,0);
\draw [line width=1pt] (14,0)-- (14,1);
\draw [line width=1pt] (14,1)-- (13,2);
\draw [line width=1pt] (13,2)-- (11,3);
\draw [line width=1pt] (11,3)-- (9,3);
\draw [line width=1pt] (9,3)-- (9,2);
\draw [line width=1pt] (9,2)-- (11,0);
\draw[color=black] (0.5,-0.6) node {$Q_1$};
\draw[color=black] (2.5,-0.6) node {$Q_2$};
\draw[color=black] (4.5,-0.6) node {$Q_3$};
\draw[color=black] (7,-0.6) node {$Q_4$};
\draw[color=black] (11.5,-0.6) node {$Q = Q_1 + Q_2 + Q_3 + Q_4$};
\begin{scriptsize}
\draw [fill=black] (0,1) circle (1pt);
\draw[color=black] (0.05,1.2) node {$(0,1)$};
\draw [fill=black] (1,0) circle (1pt);
\draw[color=black] (0.95,-0.2) node {$(1,0)$};
\draw [fill=black] (1,1) circle (1pt);
\draw[color=black] (0.95,1.2) node {$(1,1)$};
\draw [fill=black] (2,1) circle (1pt);
\draw[color=black] (2.05,1.2) node {$(0,1)$};
\draw [fill=black] (2,0) circle (1pt);
\draw[color=black] (2.05,-0.2) node {$(0,0)$};
\draw [fill=black] (3,0) circle (1pt);
\draw[color=black] (2.95,-0.2) node {$(1,0)$};
\draw [fill=black] (5,0) circle (1pt);
\draw[color=black] (4.95,-0.2) node {$(1,0)$};
\draw [fill=black] (4,0) circle (1pt);
\draw[color=black] (4.05,-0.2) node {$(0,0)$};
\draw [fill=black] (8,0) circle (1pt);
\draw[color=black] (7.95,-0.2) node {$(2,0)$};
\draw [fill=black] (7,0) circle (1pt);
\draw[color=black] (7,-0.2) node {$(1,0)$};
\draw [fill=black] (6,1) circle (1pt);
\draw[color=black] (6.05,1.2) node {$(0,1)$};
\draw [fill=black] (11,0) circle (1pt);
\draw[color=black] (11.05,-0.2) node {$(2,0)$};
\draw [fill=black] (14,0) circle (1pt);
\draw[color=black] (13.95,-0.2) node {$(5,0)$};
\draw [fill=black] (14,1) circle (1pt);
\draw[color=black] (14.4,1) node {$(5,1)$};
\draw [fill=black] (13,2) circle (1pt);
\draw[color=black] (13.35,2.1) node {$(4,2)$};
\draw [fill=black] (11,3) circle (1pt);
\draw[color=black] (11.05,3.2) node {$(2,3)$};
\draw [fill=black] (9,3) circle (1pt);
\draw[color=black] (8.9,3.2) node {$(0,3)$};
\draw [fill=black] (9,2) circle (1pt);
\draw[color=black] (8.6,2) node {$(0,2)$};
\end{scriptsize}
\end{tikzpicture}
\caption{The Newton polytopes associated $f_1$, $f_2$, $f_3$, $f_4$ and their Minkowski sum.}\label{fig:poly-ineq}
\end{figure}

\begin{figure}[H]
\centering
\begin{multicols}{2}
\scalebox{0.75}{
\begin{tikzpicture}[line cap=round,line join=round,>=triangle 45,x=1cm,y=1cm,scale=0.75]
\clip(-5.5,-2) rectangle (9,9);

\fill[line width=0pt,color=color1,fill=color1,fill opacity=0.1] (0,8)--(0,0)--(-2,-2)--(9,-2)--(9,8)--cycle;
\fill[line width=0pt,color=color2,fill=color2,fill opacity=0.1] (9,8)--(2,1)--(-3,1)--(-3,-2)--(9,-2)--cycle;
\fill[line width=0pt,color=color3,fill=color3,fill opacity=0.1] (-3,8)--(-3,-2)--(3,-2)--(3,8)--cycle;
\fill[line width=0pt,color=color4,fill=color4,fill opacity=0.1] (4,5)--(-3,-2)--(4,-2)--(4,-1)--cycle;

\draw [line width=0.75pt,color=color1,domain=-3:0] plot(\x,{(-0-1*\x)/-1});
\draw [line width=0.75pt,dashed,color=color1,domain=0:9] plot(\x,{(-0-0*\x)/1});
\draw [line width=0.75pt,color=color1] (0,0) -- (0,8);

\draw [line width=0.75pt,dashed,color=color2] (2,1) -- (2,-2);
\draw [line width=0.75pt,color=color2,domain=-3:2] plot(\x,{(-2-0*\x)/-2});
\draw [line width=0.75pt,color=color2,domain=2:9] plot(\x,{(-1--1*\x)/1});

\draw [line width=0.75pt,color=color3] (3,-2) -- (3,8);

\draw [line width=0.75pt,color=color4,domain=-3:4] plot(\x,{(-4-4*\x)/-4});
\draw [line width=0.75pt,color=color4] (4,5) -- (4,-2);
\draw [line width=0.75pt,dashed,color=color4,domain=4:5.5] plot(\x,{(-3--2*\x)/1});

\draw [line width=2pt,color=color1,domain=-3:0] plot(\x,{(-0-1*\x)/-1});
\draw [line width=2pt,color=color1] (0,0) -- (0,1);

\draw [line width=2pt,color=color2,domain=0:2] plot(\x,{(-2-0*\x)/-2});
\draw [line width=2pt,color=color2,domain=2:3] plot(\x,{(-1--1*\x)/1});

\draw [line width=2pt,color=color3] (3,-2) -- (3,2);

\draw [line width=2pt,color=color4,domain=-3:4] plot(\x,{(-4-4*\x)/-4});
\draw [line width=2pt,color=color4] (4,5) -- (4,-2);

\begin{large}
\draw[color=color1] (0, 8.5) node {$\tvar(f_1)$};
\draw[color=color2] (-4.25, 1) node {$\tvar(f_2)$};
\draw[color=color3] (3, 8.5) node {$\tvar(f_3)$};
\draw[color=color4] (6, 8.5) node {$\tvar(f_4)$};
\end{large}

\draw[fill=color4] (2,0) circle (3pt);
\draw[fill=color3] (2,0) -- ($(90:3pt)+(2,0)$) arc (90:270:3pt) -- cycle;
\draw[fill=color4] (3,0) circle (3pt);
\draw[fill=color2] (3,0) -- ($(90:3pt)+(3,0)$) arc (90:270:3pt) -- cycle;
\draw[fill=color3] (4,0) circle (3pt);
\draw[fill=color2] (4,0) -- ($(90:3pt)+(4,0)$) arc (90:270:3pt) -- cycle;
\draw[fill=color3] (0,1) circle (3pt);
\draw[fill=color2] (3,4) circle (3pt);
\draw[fill=color1] (3,4) -- ($(90:3pt)+(3,4)$) arc (90:270:3pt) -- cycle;
\draw[fill=color4] (3,2) circle (3pt);
\draw[fill=color1] (3,2) -- ($(90:3pt)+(3,2)$) arc (90:270:3pt) -- cycle;
\draw[fill=color3] (4,3) circle (3pt);
\draw[fill=color1] (4,3) -- ($(90:3pt)+(4,3)$) arc (90:270:3pt) -- cycle;
\end{tikzpicture}
}
\flushright
\begin{tikzpicture}[line cap=round,line join=round,>=triangle 45,x=1cm,y=1cm]
\clip(-1,-1) rectangle (6,4);

\draw [line width=0.5pt] (2,0)-- (5,0);
\draw [line width=0.5pt] (5,0)-- (5,1);
\draw [line width=0.5pt] (5,1)-- (4,2);
\draw [line width=0.5pt] (4,2)-- (2,3);
\draw [line width=0.5pt] (2,3)-- (0,3);
\draw [line width=0.5pt] (0,3)-- (0,2);
\draw [line width=0.5pt] (0,2)-- (2,0);

\draw[line width=0.5pt] (0,3) -- (1,3) -- (2,2) -- (2,1) -- (1,1) -- (0,2) -- cycle;
\fill[fill=color3,opacity=0.3] (0,3) -- (1,3) -- (2,2) -- (2,1) -- (1,1) -- (0,2) -- cycle;
\draw[line width=0.5pt] (1,3) -- (2,2) -- (3,2) -- (2,3) -- cycle;
\fill[pattern={Lines[angle=-45,distance=4pt,line width=2pt]},pattern color=color1,opacity=0.3] (1,3) -- (2,2) -- (3,2) -- (2,3) -- cycle;
\fill[pattern={Lines[xshift=4/sqrt(2),angle=-45,distance=4pt,line width=2pt]},pattern color=color2,opacity=0.3] (1,3) -- (2,2) -- (3,2) -- (2,3) -- cycle;
\draw[line width=0.5pt] (2,2) -- (3,1) -- (4,1) -- (3,2) -- cycle;
\fill[pattern={Lines[angle=-45,distance=4pt,line width=2pt]},pattern color=color1,opacity=0.3] (2,2) -- (3,1) -- (4,1) -- (3,2) -- cycle;
\fill[pattern={Lines[xshift=4/sqrt(2),angle=-45,distance=4pt,line width=2pt]},pattern color=color4,opacity=0.3] (2,2) -- (3,1) -- (4,1) -- (3,2) -- cycle;
\draw[line width=0.5pt] (3,2) -- (4,1) -- (5,1) -- (4,2) -- cycle;
\fill[pattern={Lines[angle=-45,distance=4pt,line width=2pt]},pattern color=color1,opacity=0.3] (3,2) -- (4,1) -- (5,1) -- (4,2) -- cycle;
\fill[pattern={Lines[xshift=4/sqrt(2),angle=-45,distance=4pt,line width=2pt]},pattern color=color3,opacity=0.3] (3,2) -- (4,1) -- (5,1) -- (4,2) -- cycle;
\draw[line width=0.5pt] (2,1) -- (2,0) -- (3,0) -- (3,1) -- cycle;
\fill[pattern={Lines[angle=-45,distance=4pt,line width=2pt]},pattern color=color3,opacity=0.3] (2,1) -- (2,0) -- (3,0) -- (3,1) -- cycle;
\fill[pattern={Lines[xshift=4/sqrt(2),angle=-45,distance=4pt,line width=2pt]},pattern color=color4,opacity=0.3] (2,1) -- (2,0) -- (3,0) -- (3,1) -- cycle;
\draw[line width=0.5pt] (3,1) -- (3,0) -- (4,0) -- (4,1) -- cycle;
\fill[pattern={Lines[angle=-45,distance=4pt,line width=2pt]},pattern color=color2,opacity=0.3] (3,1) -- (3,0) -- (4,0) -- (4,1) -- cycle;
\fill[pattern={Lines[xshift=4/sqrt(2),angle=-45,distance=4pt,line width=2pt]},pattern color=color4,opacity=0.3] (3,1) -- (3,0) -- (4,0) -- (4,1) -- cycle;
\draw[line width=0.5pt] (4,1) -- (4,0) -- (5,0) -- (5,1) -- cycle;
\fill[pattern={Lines[angle=-45,distance=4pt,line width=2pt]},pattern color=color2,opacity=0.3] (4,1) -- (4,0) -- (5,0) -- (5,1) -- cycle;
\fill[pattern={Lines[xshift=4/sqrt(2),angle=-45,distance=4pt,line width=2pt]},pattern color=color3,opacity=0.3] (4,1) -- (4,0) -- (5,0) -- (5,1) -- cycle;

\begin{scriptsize}
\draw [fill=black] (2,0) circle (1pt);
\draw[color=black] (2.05,-0.2) node {$(2,0)$};
\draw [fill=black] (5,0) circle (1pt);
\draw[color=black] (4.95,-0.2) node {$(5,0)$};
\draw [fill=black] (5,1) circle (1pt);
\draw[color=black] (5.4,1) node {$(5,1)$};
\draw [fill=black] (4,2) circle (1pt);
\draw[color=black] (4.35,2.1) node {$(4,2)$};
\draw [fill=black] (2,3) circle (1pt);
\draw[color=black] (2.05,3.2) node {$(2,3)$};
\draw [fill=black] (0,3) circle (1pt);
\draw[color=black] (-0.1,3.2) node {$(0,3)$};
\draw [fill=black] (0,2) circle (1pt);
\draw[color=black] (-0.4,2) node {$(0,2)$};
\end{scriptsize}
\draw (0,0) circle (1pt);
\draw (1,0) circle (1pt);
\draw (0,1) circle (1pt);
\draw (3,3) circle (1pt);
\draw (4,3) circle (1pt);
\draw (5,2) circle (1pt);
\draw (5,3) circle (1pt);
\draw [fill=black] (1,2) circle (1pt);
\draw [fill=black] (1,3) circle (1pt);
\draw [fill=black] (2,2) circle (1pt);
\draw [fill=black] (2,1) circle (1pt);
\draw [fill=black] (1,1) circle (1pt);
\draw [fill=black] (3,2) circle (1pt);
\draw [fill=black] (3,1) circle (1pt);
\draw [fill=black] (4,1) circle (1pt);
\draw [fill=black] (3,0) circle (1pt);
\draw [fill=black] (4,0) circle (1pt);
\end{tikzpicture}
\end{multicols}
\caption{Problem \eqref{ex:positivstellensatz} is illustrated here as the tropical basic semialgebraic set $\{ \textcolor{color1}{f^+_1 \geq f^-_1} \} \cap \{ \textcolor{color2}{f^+_2 \geq f^-_2} \} \cap \{ \textcolor{color3}{f^+_3 \geq f^-_3} \}$ is indeed included in the set $\{ \textcolor{color4!75!black}{f^+_4 \geq f^-_4} \}$. The subdivision of $Q$ associated to this system is displayed on the right.}\label{fig:sys-ineq-1}
\end{figure}
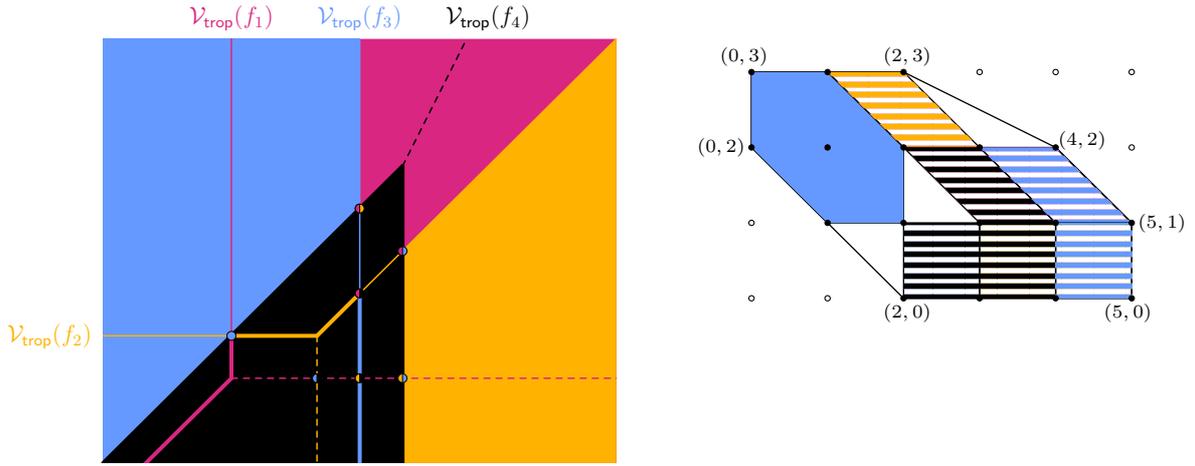

\begin{figure}[H]
\centering
\begin{multicols}{2}
\scalebox{0.75}{
\begin{tikzpicture}[line cap=round,line join=round,>=triangle 45,x=1cm,y=1cm,scale=0.75]
\clip(-5.5,-2) rectangle (9,9);

\fill[line width=0pt,color=color1,fill=color1,fill opacity=0.1] (0,8)--(0,0)--(-2,-2)--(9,-2)--(9,8)--cycle;
\fill[line width=0pt,color=color2,fill=color2,fill opacity=0.1] (9,8)--(2,1)--(-3,1)--(-3,-2)--(9,-2)--cycle;
\fill[line width=0pt,color=color3,fill=color3,fill opacity=0.1] (-3,8)--(-3,-2)--(3,-2)--(3,8)--cycle;
\fill[line width=0pt,color=color4,fill=color4,fill opacity=0.1] (4.5,6)--(-3,-1.5)--(-3,-2)--(4.5,-2)--cycle;

\draw [line width=0.75pt,dashed,color=color1,domain=-3:0] plot(\x,{(-0-1*\x)/-1});
\draw [line width=0.75pt,dashed,color=color1,domain=0:9] plot(\x,{(-0-0*\x)/1});
\draw [line width=0.75pt,color=color1] (0,0) -- (0,8);

\draw [line width=0.75pt,dashed,color=color2] (2,1) -- (2,-2);
\draw [line width=0.75pt,color=color2,domain=-3:2] plot(\x,{(-2-0*\x)/-2});
\draw [line width=0.75pt,color=color2,domain=2:9] plot(\x,{(-1--1*\x)/1});

\draw [line width=0.75pt,color=color3] (3,-2) -- (3,8);

\draw [line width=0.75pt,dashed,color=color4,domain=4.5:5.5] plot(\x,{(-3--2*\x)/1});

\draw [line width=2pt,color=color1,domain=-3:0] plot(\x,{(-0-1*\x)/-1});
\draw [line width=2pt,color=color1] (0,0) -- (0,1);

\draw [line width=2pt,color=color2,domain=0:2] plot(\x,{(-2-0*\x)/-2});
\draw [line width=2pt,color=color2,domain=2:3] plot(\x,{(-1--1*\x)/1});

\draw [line width=2pt,color=color3] (3,-2) -- (3,2);

\draw [line width=2pt,dashed,color=color4,domain=-3:4.5] plot(\x,{(-6-4*\x)/-4});
\draw [line width=2pt,dashed,color=color4] (4.5,6) -- (4.5,-2);

\begin{large}
\draw[color=color1] (0, 8.5) node {$\tvar(f_1)$};
\draw[color=color2] (-4.25, 1) node {$\tvar(f_2)$};
\draw[color=color3] (3, 8.5) node {$\tvar(f_3)$};
\draw[color=color4] (6.75, 8.5) node {$\tvar(f^+_4 \tplus \lambda \tdot f^-_4)$};
\end{large}

\draw[fill=color4] (2,0) circle (3pt);
\draw[fill=color3] (2,0) -- ($(90:3pt)+(2,0)$) arc (90:270:3pt) -- cycle;
\draw[fill=color4] (3,0) circle (3pt);
\draw[fill=color2] (3,0) -- ($(90:3pt)+(3,0)$) arc (90:270:3pt) -- cycle;
\draw[fill=color3] (4.5,0) circle (3pt);
\draw[fill=color2] (4.5,0) -- ($(90:3pt)+(4.5,0)$) arc (90:270:3pt) -- cycle;
\draw[fill=color3] (-0.5,1) circle (3pt);
\draw[fill=color1] (-0.5,1) -- ($(90:3pt)+(-0.5,1)$) arc (90:270:3pt) -- cycle;
\draw[fill=color3] (0,1.5) circle (3pt);
\draw[fill=color2] (0,1.5) -- ($(90:3pt)+(0,1.5)$) arc (90:270:3pt) -- cycle;
\draw[fill=color4] (0,1) circle (3pt);
\draw[fill=color3] (0,1) -- ($(90:3pt)+(0,1)$) arc (90:270:3pt) -- cycle;
\draw[fill=color2] (3,4.5) circle (3pt);
\draw[fill=color1] (3,4.5) -- ($(90:3pt)+(3,4.5)$) arc (90:270:3pt) -- cycle;
\draw[fill=color4] (3,2) circle (3pt);
\draw[fill=color1] (3,2) -- ($(90:3pt)+(3,2)$) arc (90:270:3pt) -- cycle;
\draw[fill=color3] (4.5,3.5) circle (3pt);
\draw[fill=color1] (4.5,3.5) -- ($(90:3pt)+(4.5,3.5)$) arc (90:270:3pt) -- cycle;
\end{tikzpicture}
}
\flushright
\begin{tikzpicture}[line cap=round,line join=round,>=triangle 45,x=1cm,y=1cm]
\clip(-1,-1) rectangle (6,4);

\draw [line width=0.5pt] (2,0)-- (5,0);
\draw [line width=0.5pt] (5,0)-- (5,1);
\draw [line width=0.5pt] (5,1)-- (4,2);
\draw [line width=0.5pt] (4,2)-- (2,3);
\draw [line width=0.5pt] (2,3)-- (0,3);
\draw [line width=0.5pt] (0,3)-- (0,2);
\draw [line width=0.5pt] (0,2)-- (2,0);

\draw[line width=0.5pt] (0,3) -- (1,3) -- (2,2) -- (1,2) -- cycle;
\fill[pattern={Lines[angle=-45,distance=4pt,line width=2pt]},pattern color=color2,opacity=0.3] (0,3) -- (1,3) -- (2,2) -- (1,2) -- cycle;
\fill[pattern={Lines[xshift=4/sqrt(2),angle=-45,distance=4pt,line width=2pt]},pattern color=color3,opacity=0.3] (0,3) -- (1,3) -- (2,2) -- (1,2) -- cycle;
\draw[line width=0.5pt] (0,3) -- (1,2) -- (1,1) -- (0,2) -- cycle;
\fill[pattern={Lines[angle=-45,distance=4pt,line width=2pt]},pattern color=color1,opacity=0.3] (0,3) -- (1,2) -- (1,1) -- (0,2) -- cycle;
\fill[pattern={Lines[xshift=4/sqrt(2),angle=-45,distance=4pt,line width=2pt]},pattern color=color3,opacity=0.3] (0,3) -- (1,2) -- (1,1) -- (0,2) -- cycle;
\draw[line width=0.5pt] (1,2) -- (2,2) -- (2,1) -- (1,1) -- cycle;
\fill[pattern={Lines[angle=-45,distance=4pt,line width=2pt]},pattern color=color3,opacity=0.3] (1,2) -- (2,2) -- (2,1) -- (1,1) -- cycle;
\fill[pattern={Lines[xshift=4/sqrt(2),angle=-45,distance=4pt,line width=2pt]},pattern color=color4,opacity=0.3] (1,2) -- (2,2) -- (2,1) -- (1,1) -- cycle;
\draw[line width=0.5pt] (1,3) -- (2,2) -- (3,2) -- (2,3) -- cycle;
\fill[pattern={Lines[angle=-45,distance=4pt,line width=2pt]},pattern color=color1,opacity=0.3] (1,3) -- (2,2) -- (3,2) -- (2,3) -- cycle;
\fill[pattern={Lines[xshift=4/sqrt(2),angle=-45,distance=4pt,line width=2pt]},pattern color=color2,opacity=0.3] (1,3) -- (2,2) -- (3,2) -- (2,3) -- cycle;
\draw[line width=0.5pt] (2,2) -- (3,1) -- (4,1) -- (3,2) -- cycle;
\fill[pattern={Lines[angle=-45,distance=4pt,line width=2pt]},pattern color=color1,opacity=0.3] (2,2) -- (3,1) -- (4,1) -- (3,2) -- cycle;
\fill[pattern={Lines[xshift=4/sqrt(2),angle=-45,distance=4pt,line width=2pt]},pattern color=color4,opacity=0.3] (2,2) -- (3,1) -- (4,1) -- (3,2) -- cycle;
\draw[line width=0.5pt] (3,2) -- (4,1) -- (5,1) -- (4,2) -- cycle;
\fill[pattern={Lines[angle=-45,distance=4pt,line width=2pt]},pattern color=color1,opacity=0.3] (3,2) -- (4,1) -- (5,1) -- (4,2) -- cycle;
\fill[pattern={Lines[xshift=4/sqrt(2),angle=-45,distance=4pt,line width=2pt]},pattern color=color3,opacity=0.3] (3,2) -- (4,1) -- (5,1) -- (4,2) -- cycle;
\draw[line width=0.5pt] (2,1) -- (2,0) -- (3,0) -- (3,1) -- cycle;
\fill[pattern={Lines[angle=-45,distance=4pt,line width=2pt]},pattern color=color3,opacity=0.3] (2,1) -- (2,0) -- (3,0) -- (3,1) -- cycle;
\fill[pattern={Lines[xshift=4/sqrt(2),angle=-45,distance=4pt,line width=2pt]},pattern color=color4,opacity=0.3] (2,1) -- (2,0) -- (3,0) -- (3,1) -- cycle;
\draw[line width=0.5pt] (3,1) -- (3,0) -- (4,0) -- (4,1) -- cycle;
\fill[pattern={Lines[angle=-45,distance=4pt,line width=2pt]},pattern color=color2,opacity=0.3] (3,1) -- (3,0) -- (4,0) -- (4,1) -- cycle;
\fill[pattern={Lines[xshift=4/sqrt(2),angle=-45,distance=4pt,line width=2pt]},pattern color=color4,opacity=0.3] (3,1) -- (3,0) -- (4,0) -- (4,1) -- cycle;
\draw[line width=0.5pt] (4,1) -- (4,0) -- (5,0) -- (5,1) -- cycle;
\fill[pattern={Lines[angle=-45,distance=4pt,line width=2pt]},pattern color=color2,opacity=0.3] (4,1) -- (4,0) -- (5,0) -- (5,1) -- cycle;
\fill[pattern={Lines[xshift=4/sqrt(2),angle=-45,distance=4pt,line width=2pt]},pattern color=color3,opacity=0.3] (4,1) -- (4,0) -- (5,0) -- (5,1) -- cycle;

\begin{scriptsize}
\draw [fill=black] (2,0) circle (1pt);
\draw[color=black] (2.05,-0.2) node {$(2,0)$};
\draw [fill=black] (5,0) circle (1pt);
\draw[color=black] (4.95,-0.2) node {$(5,0)$};
\draw [fill=black] (5,1) circle (1pt);
\draw[color=black] (5.4,1) node {$(5,1)$};
\draw [fill=black] (4,2) circle (1pt);
\draw[color=black] (4.35,2.1) node {$(4,2)$};
\draw [fill=black] (2,3) circle (1pt);
\draw[color=black] (2.05,3.2) node {$(2,3)$};
\draw [fill=black] (0,3) circle (1pt);
\draw[color=black] (-0.1,3.2) node {$(0,3)$};
\draw [fill=black] (0,2) circle (1pt);
\draw[color=black] (-0.4,2) node {$(0,2)$};
\end{scriptsize}
\draw (0,0) circle (1pt);
\draw (1,0) circle (1pt);
\draw (0,1) circle (1pt);
\draw (3,3) circle (1pt);
\draw (4,3) circle (1pt);
\draw (5,2) circle (1pt);
\draw (5,3) circle (1pt);
\draw [fill=black] (1,2) circle (1pt);
\draw [fill=black] (1,3) circle (1pt);
\draw [fill=black] (2,2) circle (1pt);
\draw [fill=black] (2,1) circle (1pt);
\draw [fill=black] (1,1) circle (1pt);
\draw [fill=black] (3,2) circle (1pt);
\draw [fill=black] (3,1) circle (1pt);
\draw [fill=black] (4,1) circle (1pt);
\draw [fill=black] (3,0) circle (1pt);
\draw [fill=black] (4,0) circle (1pt);
\end{tikzpicture}
\end{multicols}
\caption{For all $\lambda > 0$, the tropical basic semialgebraic set $\{ \textcolor{color1}{f^+_1 \geq f^-_1} \} \cap \{ \textcolor{color2}{f^+_2 \geq f^-_2} \} \cap \{ \textcolor{color3}{f^+_3 \geq f^-_3} \}$ is indeed included in the set $\{ \textcolor{color4!75!black}{\lambda \tdot f^+_4 > f^-_4} \}$, showing that system \eqref{ex:posi-lambda} does not have a solution in $\R^2$. The subdivision of $Q$ associated to this system is displayed on the right.}\label{fig:sys-ineq-2}
\end{figure}
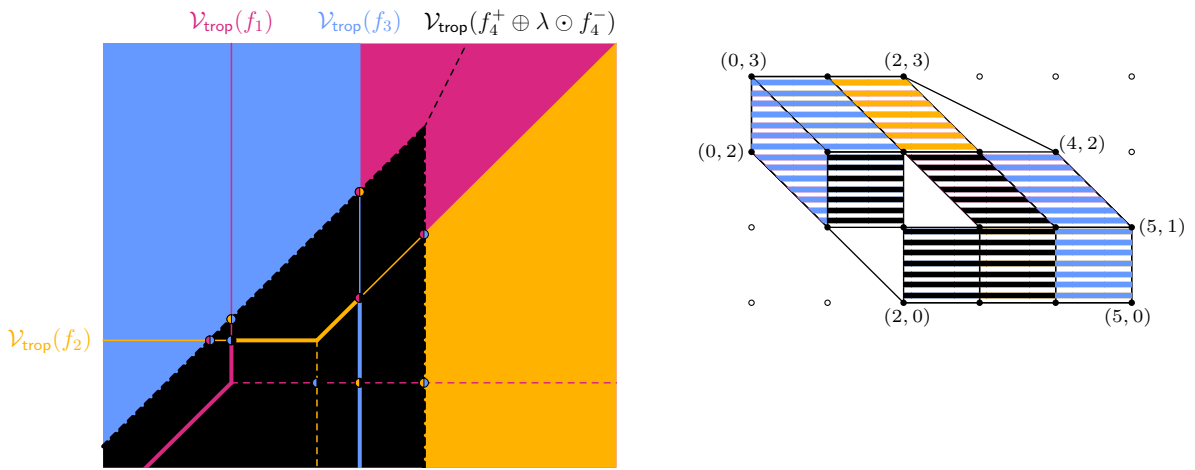

\begin{figure}[H]
\centering
\scalebox{0.95}{
\begin{tikzpicture}[line cap=round,line join=round,>=triangle 45,x=1cm,y=1cm,scale=1.5]
\clip(-0.2,-0.2) rectangle (5,3);

\draw [line width=0.5pt] (2,0)-- (5,0);
\draw [line width=0.5pt] (5,0)-- (5,1);
\draw [line width=0.5pt] (5,1)-- (4,2);
\draw [line width=0.5pt] (4,2)-- (2,3);
\draw [line width=0.5pt] (2,3)-- (0,3);
\draw [line width=0.5pt] (0,3)-- (0,2);
\draw [line width=0.5pt] (0,2)-- (2,0);

\draw[line width=1pt] (0,3) -- (1,3) -- (2,2) -- (1,2) -- cycle;
\fill[pattern={Lines[angle=-45,distance=4pt,line width=2pt]},pattern color=color2, opacity=0.3] (0,3) -- (1,3) -- (2,2) -- (1,2) -- cycle;
\fill[pattern={Lines[xshift=4/sqrt(2),angle=-45,distance=4pt,line width=2pt]},pattern color=color3,opacity=0.3] (0,3) -- (1,3) -- (2,2) -- (1,2) -- cycle;
\draw[line width=1pt] (0,3) -- (1,2) -- (1,1) -- (0,2) -- cycle;
\fill[pattern={Lines[angle=-45,distance=4pt,line width=2pt]},pattern color=color1,opacity=0.3] (0,3) -- (1,2) -- (1,1) -- (0,2) -- cycle;
\fill[pattern={Lines[xshift=4/sqrt(2),angle=-45,distance=4pt,line width=2pt]},pattern color=color3,opacity=0.3] (0,3) -- (1,2) -- (1,1) -- (0,2) -- cycle;
\draw[line width=1pt] (1,2) -- (2,2) -- (2,1) -- (1,1) -- cycle;
\fill[pattern={Lines[angle=-45,distance=4pt,line width=2pt]},pattern color=color3,opacity=0.3] (1,2) -- (2,2) -- (2,1) -- (1,1) -- cycle;
\fill[pattern={Lines[xshift=4/sqrt(2),angle=-45,distance=4pt,line width=2pt]},pattern color=color4,opacity=0.3] (1,2) -- (2,2) -- (2,1) -- (1,1) -- cycle;
\draw[line width=1pt] (1,3) -- (2,2) -- (3,2) -- (2,3) -- cycle;
\fill[pattern={Lines[angle=-45,distance=4pt,line width=2pt]},pattern color=color1,opacity=0.3] (1,3) -- (2,2) -- (3,2) -- (2,3) -- cycle;
\fill[pattern={Lines[xshift=4/sqrt(2),angle=-45,distance=4pt,line width=2pt]},pattern color=color2,opacity=0.3] (1,3) -- (2,2) -- (3,2) -- (2,3) -- cycle;
\draw[line width=1pt] (2,2) -- (3,1) -- (4,1) -- (3,2) -- cycle;
\fill[pattern={Lines[angle=-45,distance=4pt,line width=2pt]},pattern color=color1,opacity=0.3] (2,2) -- (3,1) -- (4,1) -- (3,2) -- cycle;
\fill[pattern={Lines[xshift=4/sqrt(2),angle=-45,distance=4pt,line width=2pt]},pattern color=color4,opacity=0.3] (2,2) -- (3,1) -- (4,1) -- (3,2) -- cycle;
\draw[line width=1pt] (3,2) -- (4,1) -- (5,1) -- (4,2) -- cycle;
\fill[pattern={Lines[angle=-45,distance=4pt,line width=2pt]},pattern color=color1,opacity=0.3] (3,2) -- (4,1) -- (5,1) -- (4,2) -- cycle;
\fill[pattern={Lines[xshift=4/sqrt(2),angle=-45,distance=4pt,line width=2pt]},pattern color=color3,opacity=0.3] (3,2) -- (4,1) -- (5,1) -- (4,2) -- cycle;
\draw[line width=1pt] (2,1) -- (2,0) -- (3,0) -- (3,1) -- cycle;
\fill[pattern={Lines[angle=-45,distance=4pt,line width=2pt]},pattern color=color3,opacity=0.3] (2,1) -- (2,0) -- (3,0) -- (3,1) -- cycle;
\fill[pattern={Lines[xshift=4/sqrt(2),angle=-45,distance=4pt,line width=2pt]},pattern color=color4,opacity=0.3] (2,1) -- (2,0) -- (3,0) -- (3,1) -- cycle;
\draw[line width=1pt] (3,1) -- (3,0) -- (4,0) -- (4,1) -- cycle;
\fill[pattern={Lines[angle=-45,distance=4pt,line width=2pt]},pattern color=color2,opacity=0.3] (3,1) -- (3,0) -- (4,0) -- (4,1) -- cycle;
\fill[pattern={Lines[xshift=4/sqrt(2),angle=-45,distance=4pt,line width=2pt]},pattern color=color4,opacity=0.3] (3,1) -- (3,0) -- (4,0) -- (4,1) -- cycle;
\draw[line width=1pt] (4,1) -- (4,0) -- (5,0) -- (5,1) -- cycle;
\fill[pattern={Lines[angle=-45,distance=4pt,line width=2pt]},pattern color=color2,opacity=0.3] (4,1) -- (4,0) -- (5,0) -- (5,1) -- cycle;
\fill[pattern={Lines[xshift=4/sqrt(2),angle=-45,distance=4pt,line width=2pt]},pattern color=color3,opacity=0.3] (4,1) -- (4,0) -- (5,0) -- (5,1) -- cycle;

\begin{scriptsize}
\draw (-0.1,-0.1) circle (1pt);
\draw (0.9,-0.1) circle (1pt);
\draw (-0.1,0.9) circle (1pt);
\draw (1.9,-0.1) circle (1pt);
\draw (0.9,0.9) circle (1pt);
\draw (-0.1,1.9) circle (1pt);
\draw (2.9,-0.1) circle (1pt);
\draw [fill=black] (1.9,0.9) circle (1pt);
\draw[color=black] (1.65,0.75) node {$4,(1,0)$};
\draw [fill=black] (0.9,1.9) circle (1pt);
\draw[color=black] (0.65,1.75) node {$1,(0,1)$};
\draw (-0.1,2.9) circle (1pt);
\draw (3.9,-0.1) circle (1pt);
\draw [fill=black] (2.9,0.9) circle (1pt);
\draw[color=black] (2.65,0.75) node {$4,(1,0)$};
\draw [fill=black] (1.9,1.9) circle (1pt);
\draw[color=black] (1.65,1.75) node {$4,(1,0)$};
\draw [fill=black] (0.9,2.9) circle (1pt);
\draw[color=black] (0.65,2.75) node {$2,(0,1)$};
\draw (4.9,-0.1) circle (1pt);
\draw [fill=black] (3.9,0.9) circle (1pt);
\draw[color=black] (3.65,0.75) node {$4,(1,0)$};
\draw [fill=black] (2.9,1.9) circle (1pt);
\draw[color=black] (2.65,1.75) node {$4,(1,0)$};
\draw [fill=black] (1.9,2.9) circle (1pt);
\draw[color=black] (1.65,2.75) node {$2,(0,1)$};
\draw [fill=black] (4.9,0.9) circle (1pt);
\draw[color=black] (4.65,0.75) node {$3,(1,0)$};
\draw [fill=black] (3.9,1.9) circle (1pt);
\draw[color=black] (3.65,1.75) node {$3,(1,0)$};
\draw (2.9,2.9) circle (1pt);
\draw (4.9,1.9) circle (1pt);
\draw (3.9,2.9) circle (1pt);
\draw (4.9,2.9) circle (1pt);
\end{scriptsize}
\end{tikzpicture}
}
\caption{The polytope $Q+\delta$, with the integer points inside the maximal dimensional cells of the decomposition of $Q+\delta$ labelled by the row content the cell they belong to.}\label{fig:rowc-2}
\end{figure}
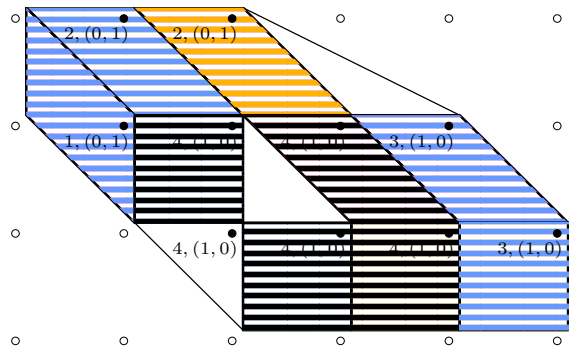

For this collection of supports, one can once again take $\delta = (-1+\varepsilon,-1+\varepsilon)$ with $\varepsilon = \frac{1}{10}$, which gives us the Canny-Emiris set
\[
\calE = \{ (1,0), (0,1), (2,0), (1,1), (0,2), (3,0), (2,1), (1,2), (4,0), (3,1) \}
\]
corresponding to the set of monomials
\[
\{ x_1, x_2, x_1^2, x_1x_2, x_2^2, x_1^3, x_1^2x_2, x_1x_2^2, x_1^4, x_1^3x_2 \} \enspace .
\]

We thus obtain the $21 \times 10$ submatrices $\mac^{+}_\calE$ and $\mac^{\boldsymbol{-}}_\calE$ of the Macaulay matrix associated to the set $\calE$, which we combined in the single following matrix for the sake of readability and to save some space, with the normal font weight coeffficients corresponding to the coefficients of $\mac^{+}_\calE$, and the bold ones corresponding to $\mac^{\boldsymbol{-}}_\calE$

\def\cpm{\mathbin{\ThisStyle{\ensurestackMath{\abovebaseline[-\dimexpr1pt+2.4\LMpt]{%
  \stackunder[-\dimexpr1pt+2.5\LMpt]{\SavedStyle+}{%
  \boldsymbol{\SavedStyle-}}}}}}}

\[
\mac_{\mathcal{E}}^{\cpm} =
\begin{blockarray}{*{11}{c}}
& x_1 & x_2 & x_1^2 & x_1x_2 & x_2^2 & x_1^3 & x_1^2x_2 & x_1x_2^2 & x_1^4 & x_1^3x_2\\
\begin{block}{c(*{10}{c})}
f_1 & 0 & \boldsymbol{0} & & 0 & & & & & &\\
x_1f_1 & & & 0 & \boldsymbol{0} & & & 0 & & &\\
x_2f_1 & & & & 0 & \boldsymbol{0} & & & 0 & &\\
x_1^2f_1 & & & & & & 0 & \boldsymbol{0} & & & 0\\
x_1f_2 & 2 & & 0 & \boldsymbol{1} & & & & & &\\
x_2f_2 & & 2 & & 0 & \boldsymbol{1} & & & & &\\
x_1^2f_2 & & & 2 & & & 0 & \boldsymbol{1} & & &\\
x_1x_2f_2 & & & & 2 & & & 0 & \boldsymbol{1} & &\\
x_1^3f_2 & & & & & & 2 & & & 0 & \boldsymbol{1}\\
x_1f_3 & 3 & & \boldsymbol{0} & & & & & & &\\
x_2f_3 & & 3 & & \boldsymbol{0} & & & & & &\\
x_1^2f_3 & & & 3 & & & \boldsymbol{0} & & & &\\
x_1x_2f_3 & & & & 3 & & & \boldsymbol{0} & & &\\
x_2^2f_3 & & & & & 3 & & & \boldsymbol{0} & &\\
x_1^3f_3 & & & & & & 3 & & & \boldsymbol{0} &\\
x_1^2x_2f_3 & & & & & & & 3 & & & \boldsymbol{0}\\
f_4 & \boldsymbol{1 + \lambda} & 0 & -3 & & & & & & &\\
x_1f_4 & & & \boldsymbol{1 + \lambda} & 0 & & -3 & & & &\\
x_2f_4 & & & & \boldsymbol{1 + \lambda} & 0 & & -3 & & &\\
x_1^2f_4 & & & & & & \boldsymbol{1 + \lambda} & 0 & & -3 &\\
x_1x_2f_4 & & & & & & & \boldsymbol{1 + \lambda} & 0 & & -3\\
\end{block}
\end{blockarray} \enspace .
\]

Applying the Canny-Emiris construction to the collection of polynomials in the system \eqref{ex:posi-lambda}, we obtain the subdivision of $Q$ which allows us to associate to every point $p \in \calE$ its row content $i, a_i$, which we recall is univocally determined by the maximal-dimensional cell of the decomposition of $Q$ to which $p-\delta$ belongs, as illustrated in \Cref{fig:rowc-2}.

More precisely, in the following table, for each point $p$ of $\calE$ in the first row:
\begin{itemize}[label=$\diamond$]
\item the second row displays the monomial $x^p$ which corresponds to a column of the Macaulay matrix,
\item the third row displays the row content $i, a_i$ of $p$,
\item the fourth row displays the polynomial $x^{p-a_i}f_i$ which corresponds to a row of the Macaulay matrix,
\item and finally the last row displays the tropical scaling factor $h(p-\delta)$ which must be substracted (in the usual sense) to the column $p$ of the matrices $\mac^{\cpm}_{\mathcal{EE}}$ in order to obtain the matrix $\widetilde{\mac}^{\cpm}_{\mathcal{EE}}$.
\end{itemize}

\begin{center}
\scalebox{0.87}
{
\begin{tabular}{|c||*{10}{c}|}
\hline
$p \in \calE$ & $(1,0)$ & $(0,1)$ & $(2,0)$ & $(1,1)$ & $(0,2)$ & $(3,0)$ & $(2,1)$ & $(1,2)$ & $(4,0)$ & $(3,1)$\\
\hline
$x^p$ & $x_1$ & $x_2$ & $x_1^2$ & $x_1x_2$ & $x_2^2$ & $x_1^3$ & $x_1^2x_2$ & $x_1x_2^2$ & $x_1^4$ & $x_1^3x_2$\\
\hline
$i, a_i$ & $4,(1,0)$ & $1,(0,1)$ & $4,(1,0)$ & $4,(1,0)$ & $2,(0,1)$ & $4,(1,0)$ & $4,(1,0)$ & $2,(0,1)$ & $3,(1,0)$ & $3,(1,0)$\\
\hline
$x^{p-a_i}f_i$ & $f_4$ & $f_1$ & $x_1f_4$ & $x_2f_4$ & $x_2f_2$ & $x_1^2f_4$ & $x_1x_2f_4$ & $x_1x_2f_2$ & $x_1^3f_3$ & $x_1^2x_2f_3$\\
\hline
$h(p-\delta)$ & $6+\lambda$ & $\frac{51+9\lambda}{10}$ & $\frac{42}{10}+\lambda$ & $\frac{51}{10}+\lambda$ & $\frac{41+\lambda}{10}$ & $\frac{13}{10}+\lambda$ & $\frac{25}{10}+\lambda$ & $\frac{17+\lambda}{10}$ & $\frac{-26+\lambda}{10}$ & $\frac{-13+2\lambda}{10}$\\
\hline
\end{tabular}
}
\end{center}

With the information from the previous table, we obtain the following $10 \times 10$ pair of square submatrices $\mac^{\cpm}_{\mathcal{EE}}$ of $\mac^{\cpm}_{\mathcal{E}}$

\[
\mac_{\mathcal{EE}}^{\cpm} =
\begin{blockarray}{*{11}{c}}
& x_1 & x_2 & x_1^2 & x_1x_2 & x_2^2 & x_1^3 & x_1^2x_2 & x_1x_2^2 & x_1^4 & x_1^3x_2\\
\begin{block}{c(*{10}{c})}
f_4 & \boldsymbol{1 + \lambda} & 0 & -3 & & & & & & &\\
f_1 & 0 & \boldsymbol{0} & & 0 & & & & & &\\
x_1f_4 & & & \boldsymbol{1 + \lambda} & 0 & & -3 & & & &\\
x_2f_4 & & & & \boldsymbol{1 + \lambda} & 0 & & -3 & & &\\
x_2f_2 & & 2 & & 0 & \boldsymbol{1} & & & & &\\
x_1^2f_4 & & & & & & \boldsymbol{1 + \lambda} & 0 & & -3 &\\
x_1x_2f_4 & & & & & & & \boldsymbol{1 + \lambda} & 0 & & -3\\
x_1x_2f_2 & & & & 2 & & & 0 & \boldsymbol{1} & &\\
x_1^3f_3 & & & & & & 3 & & & \boldsymbol{0} &\\
x_1^2x_2f_3 & & & & & & & 3 & & & \boldsymbol{0}\\
\end{block}
\end{blockarray} \enspace .
\]

Finally, after applying the tropical scaling of factor $h(p-\delta)$ to the column $p$ of the previous matrices for all $p \in \calE$, we obtain the following pair of matrices

\[
\widetilde{\mac}_{\mathcal{EE}}^{\cpm} = -\frac{1}{10}
\begin{tiny}
\begin{pmatrix}
\boldsymbol{50} & 51 + 9\lambda & 72 + 10\lambda & & & & & & &\\
60 + 10\lambda & \boldsymbol{51 + 9\lambda} & & 51 + 10\lambda & & & & & &\\
& & \boldsymbol{32} & 51 + 10\lambda & & 43 + 10\lambda & & & &\\
& & & \boldsymbol{41} & 41 + \lambda & & 55 + 10\lambda & & &\\
& 31 + 9\lambda & & 51 + 10\lambda & \boldsymbol{31 + \lambda} & & & & &\\
& & & & & \boldsymbol{3} & 25 + 10\lambda & & 4 + \lambda &\\
& & & & & & \boldsymbol{15} & 17 + \lambda & & 17 + 2\lambda\\
& & & 31 + 10\lambda & & & 25 + 10\lambda & \boldsymbol{7 + \lambda} & &\\
& & & & & -17 + 10\lambda & & & \boldsymbol{-26 + \lambda} &\\
& & & & & & -5 + 10\lambda & & & \boldsymbol{-13 + 2\lambda}\\
\end{pmatrix}
\end{tiny}
\]
and we indeed observe for every $\lambda > 0$ that for each row, the diagonal coefficient of $\widetilde{\mac}^{\boldsymbol{-}}_{\calE\calE}$ is strictly greater than all the coefficients of $\widetilde{\mac}^{+}_{\calE\calE}$ in the same row, or in other words that $\widetilde{\mac}^{+}_{\calE\calE}$ is diagonally dominated by $\widetilde{\mac}^{\boldsymbol{-}}_{\calE\calE}$, and thus from Lemmas \ref{lem:ddom} and \ref{lem-chang-diag}, the only solution to the equation $\mac^{+}_\calE \tdot y \geq \mac^{\boldsymbol{-}}_\calE \tdot y$ in $\R^{10}$ is $y = \zero$.
\end{expl}

\begin{rmk}
  We may define a tropical semialgebraic subset of $\R^n$ as a finite union of tropical basic semialgebraic subsets. This leads to a more general class of sets than the one
  arising by considering the images
  by the valuation of semialgebraic sets
  over a real closed non-archimedean field. Indeed,
  it is shown in~\cite[Theorem~3.1]{AGS20} that the image by a non-trivial and convex valuation of a semialgebraic set over a non-archimedean field is a closed semilinear set. In particular, when the value group is $\R$, this image is a finite union of {\em closed} tropical basic semialgebraic subsets, whereas our definition allows more generally tropical basic semialgebraic subsets not be closed, owing to the presence of strict inequalities.
\end{rmk}
\begin{rmk}
Once again, the proof of the implication that if there is a solution $x \in \R^n$ to the system $f^+ \rhd f^-$, then there exists a vector $y \in \R^{\calE'}$ such that $\mac^+ \tdot y \rhd \mac^- \tdot y$ is immediate by choosing $y$ to be the Veronese embedding $\ver(x)$ of $x$, which we recall is equal to the vector $(x^{\nu})_{\nu \in \calE'}$. Therefore, we will only focus on the converse implication in what follows, or rather on its contrapositive.
\end{rmk}

\subsection{Further preliminary results}

In this section, we adapt the results from \Cref{sec:prel} to the two-sided case in order to prove \Cref{thm:positivstellensatz}.

\subsubsection{Diagonal dominance for a pair of matrices}

In order to prove this theorem, we introduce a notion of diagonal dominance adapted to systems of inequalities and equalities.

\begin{dfn}
Let $A = (a_{ij})_{(i, j) \in [p] \times [p]}$ and $B = (b_{ij})_{(i, j) \in [p] \times [p]}$ be a pair of matrices in $\T^{p \times p}$. One says that $B$ \textit{diagonally dominates} $A$, or that $A$ is \textit{diagonally dominated} by $B$ whenever
\[
b_{ii} > a_{ij} \quad \textrm{for all} \quad 1 \leq i, j \leq p \enspace .
\]
\end{dfn}

\begin{lem}\label{lem:ddom}
Let $A$ and $B$ be a pair of matrices in $\T^{p \times p}$ such that $A$ is diagonally dominated by $B$. Then the only solution to the inequation $A \tdot y \geq B \tdot y$ of unknown $y \in \T^p$ is $y = \zero$.
\end{lem}

\begin{proof}
Let $y = (y_1,\ldots,y_p) \in \T^p$ be such that $A \tdot y \geq B \tdot y$ and consider $1 \leq i \leq p$ such that $y_i = \max_{1 \leq j \leq p} y_j$. Assume that $y \neq \zero$, and thus that $y_i > -\infty$. Then from the relation $A \tdot y \geq B \tdot y$, it follows in particular that
\[
\max_{1 \leq j \leq p} (a_{ij} + y_j) \geq \max_{1 \leq j \leq p} (b_{ij} + y_j) \geq b_{ii} + y_i \enspace ,
\]
but by choice of $i$ and since $A$ is diagonally dominated by $B$, we have
\[
b_{ii} + y_i > a_{ij} + y_j \quad \textrm{for all} \quad 1 \leq j \leq p \enspace ,
\]
hence we get a contradiction.
\end{proof}

\begin{cor}\label{cor:ddom}
Let $A$ and $B$ be a pair of matrices in $\T^{p \times p}$ such that for all $1 \leq i \leq p$, we have either
\begin{equation}\label{eq:perddom}
b_{ii} > \max_{1 \leq j \leq p} a_{ij} \quad \textrm{or} \quad a_{ii} > \max_{1 \leq j \leq p} b_{ij} \enspace .
\end{equation}
Then the only solution to the inequation $A \tdot y = B \tdot y$ of unknown $y \in \T^p$ is $y = \zero$.
\end{cor}

\begin{proof}
The idea of the proof is simply that by swapping some of the rows of $A$ and $B$, we can obtain a pair of matrices $\widetilde{A}$ and $\widetilde{B}$ such that $\widetilde{A}$ is diagonally dominated by $\widetilde{B}$. More precisely, set 
\[
I = \{ 1 \leq i \leq p : b_{ii} > a_{ij} \, \textrm{for all} \; 1 \leq j \leq p \} \quad \textrm{and} \quad J = \{ 1, \ldots, p \} \setminus I \enspace ,
\]
and let
\[
\widetilde{A} = (\widetilde{a}_{ij})_{(i, j) \in [p] \times [p]} \quad \textrm{with} \quad \widetilde{a}_{ij} = \left\{\begin{array}{rcl}
a_{ij} & \textrm{if} & i \in I\\
b_{ij} & \textrm{if} & i \in J
\end{array} \right. 
\]
and
\[
\widetilde{B} = (\widetilde{b}_{ij})_{(i, j) \in [p] \times [p]} \quad \textrm{with} \quad \widetilde{b}_{ij} = \left\{\begin{array}{rcl}
b_{ij} & \textrm{if} & i \in I\\
a_{ij} & \textrm{if} & i \in J \enspace .
\end{array} \right. 
\]
By construction, notice that for $y \in \R^p$, we have
\[
\widetilde{A} \tdot y = \widetilde{B} \tdot y \iff A \tdot y = B \tdot y \enspace ,
\]
and moreover $\widetilde{A}$ is diagonally dominated by $\widetilde{B}$, thus by the previous lemma, the only solution to the equality $\widetilde{A} \tdot y = \widetilde{B} \tdot y$ is $y = \zero$, and thus $\zero$ is also the only solution to the equality $A \tdot y = B \tdot y$.
\end{proof}

Finally we will also make use of the following two lemmas, which are immediate adaptations of \Cref{lem:nsg1} and \Cref{lem:nsg2} to the two-sided case.

\begin{lem}\label{lem-chang-diag}
Let $\rhd\in \{\geq,=\}$. 
Let $A = (a_{ij})_{(i,j) \in [p] \times [q]}$ and $B = (b_{ij})_{(i,j) \in [p] \times [q]} \in \T^{p \times q}$ be two $p \times q$ tropical matrices. Fix for $1 \leq j \leq q$, $\varepsilon_j \in \R$, and set $\widetilde{A} = (\widetilde{a}_{ij})_{(i,j) \in [p] \times [q]} \in \T^{p \times q}$ and $\widetilde{B} = (\widetilde{b}_{ij})_{(i,j) \in [p] \times [q]} \in \T^{p \times q}$ with $\widetilde{a}_{ij} = a_{ij} + \varepsilon_j$ and $\widetilde{b}_{ij} = b_{ij} + \varepsilon_j$ for all $1 \leq i \leq p$ and $1 \leq j \leq q$. Then the (in)equality $A \tdot y \rhd B \tdot y$ of unknown $y \in \T^q$ has no nonzero solution if and only if the (in)equality $\widetilde{A} \tdot \widetilde{y} \rhd \widetilde{B} \tdot \widetilde{y}$ of unknown $\widetilde{y} \in \T^q$ has no nonzero solution.
\end{lem}

\begin{lem}\label{lem-sousmatrice}
Let $\rhd\in \{\geq,=\}$. 
Let $A$ and $B$ be two $p \times q$ tropical matrices, and assume that $A$ and $B$ can both be written by block as lower-triangular matrices
\[
A = \begin{pmatrix}
A^{(m)} & \zero\\
\ast & \ast
\end{pmatrix}
\quad \textrm{and} \quad
B = \begin{pmatrix}
B^{(m)} & \zero\\
\ast & \ast
\end{pmatrix}
\]
with $A^{(m)}$ and $B^{(m)}$ two $m \times m$ square submatrices, with $0 < m \leq p,q$.
Moreover, assume that the only solution to the equation $A^{(m)} \tdot y^{(m)} \rhd B^{(m)} \tdot y^{(m)}$ of unknown $y^{(m)} \in \T^m$ is $y^{(m)} = \zero$. Then the equation $A \tdot y \rhd B \tdot y$ of unknown $y$ has no solution in $\R^q$.
\end{lem}

\subsubsection{The Shapley-Folkman lemma}

In this brief subsection, we recall the following lemma which will allow us to define a notion of row content adapted to the two-sided case.

\begin{lem}[Shapley-Folkman, see {\cite[Theorem 3.1.2]{Sch13}}]\label{lem:shapley-folkman}
Let $A_1, \ldots, A_k \subseteq \R^n$, and let
\[
x \in \sum_{i=1}^k \conv(A_i) \enspace .
\]
Then there is an index set $I \subseteq \{ 1, \ldots, k \}$ with $\vert I \vert \leq n$ such that
\[
x \in \sum_{i \in I} \conv(A_i) + \sum_{i \in \{ 1, \ldots, k \} \setminus I} A_i \enspace .
\]
\end{lem}

In the case where the Minkowski sum of the $A_i$ is not full-dimensional, one has the following direct corollary of the previous result.

\begin{cor}\label{cor:shapley-folkman}
If in \Cref{lem:shapley-folkman} $\sum_{i=1}^k \conv(A_i)$ has (affine) dimension $1 \leq d < n$,
then the index set $I$ can be choosen such that $\vert I \vert \leq d$.
\end{cor}

\subsection{Proving the Tropical Positivstellensatz}

Before detailing the proof of \Cref{thm:positivstellensatz}, notice that we can simply consider the case where all relations in system \eqref{eq:ineq} are of the form $\rhd_i \in \{ \geq, = \}$. Indeed, suppose the system $f^+(x) \rhd f^-(x)$ comprises relations of the form $f^+_i(x) > f^-_i(x)$ for $i \in I \subseteq \{ 1, \ldots, k \}$. Let $\lambda > \unit = 0$, and let us consider the transformed system $\mathcal{S}(\lambda)$, in which every relation
\[
f_i^+ (x) > f^-_i(x)
\]
is replaced by
\[
f^+_i(x) \geq \lambda \tdot f_i^-(x) \enspace .
\]
Let $\mac^-_{\calE'}(\lambda)$ denote the negative Macaulay matrix associated to $\mathcal{S}(\lambda)$, and let us take $\rhd'_j$ to be equal to $\geq$ for $j \in I$, and to $\rhd_j$ otherwise.

Then, the relation $\mac^+_{\calE'} \tdot y \rhd \mac^-_{\calE'} \tdot y$ holds iff $\mac^+_{\calE'} \tdot y \rhd' \mac^-_{\calE'}(\lambda) \tdot y$ holds for some $\lambda > \unit$. Thus, if the theorem is proven for systems without strict inequalities, it follows that the latter relation is equivalent to the solvability of $\mathcal{S}(\lambda)$, which proves the result for systems which include strict inequalities.

In the following, we will therefore simply focus on systems such that $\rhd_i \in \{ \geq, = \}$ for all $1 \leq i \leq k$.

\subsubsection{A row content adapted to two-sided relations}\label{sec-row-ineq}

We consider the collection $f = (f_1, \ldots, f_k)$ of tropical polynomials, their associated Newton polytopes $Q_i$ as in \Cref{sec:pos} and the constants $r_i$ defined in \Cref{sec:pos}. We also consider the maps $h_1, \ldots, h_k$ defined as in \Cref{sec:caem}. Moreover, we set
\[
\widetilde{h} = h_1^{\supco r_1} \supco \cdots \supco h_k^{\supco r_k} \enspace ,
\]
and note that $\widetilde{Q} = \supp(\widetilde{h})$.

Recall that now $\calE$ is a Canny-Emiris set associated to the system $f^+ \rhd f^-$, that is $\calE = (\widetilde{Q} + \delta) \cap \Z^n$, where $\delta$ is a generic vector in $V + \Z^n$ and $V$ is the direction of the affine hull of $\widetilde{Q}$. 
Assume that there is no solution $x \in \R^n$ to the system $f^+(x) \rhd f^-(x)$, \textit{i.e.} that for all $x \in \R^n$, there exists $1 \leq j \leq k$ such that
\begin{equation}\label{eqn:ineq}
f^+_j(x) \nrhd_j f^-_j(x) \enspace .
\end{equation}
More precisely, \eqref{eqn:ineq} implies that either
\begin{subequations}\label{eqn:ineqbis}
\begin{equation}\label{eqn:ineq1} 
f^-_j(x) > f^+_j(x) \enspace ,
\end{equation}
or
\begin{equation}\label{eqn:ineq2} 
\rhd_j \textrm{ is an equality \quad and} \quad f^-_j(x) < f^+_j(x) \enspace .
\end{equation}
\end{subequations}
Then for $p \in \calE$, $(p-\delta, \widetilde{h}(p-\delta))$ is in the relative interior of a facet $\widetilde{F}$ of $\hypo(\widetilde{h}) = \sum_{i=1}^k r_i\hypo(h_i)$. This facet satisfies $\widetilde{F} =  \face(x, \widetilde{h})$ for some $x \in V$,
and then
\[
\widetilde{F} = r_1F_1 + \cdots + r_kF_k \enspace ,
\]
with $F_i = \face(x,h_i)$. This means that we can write $p - \delta = r_1q_1 + \cdots + r_kq_k$ with $(q_i, h_i(q_i)) \in F_i$ for all $1 \leq i \leq k$. Set $j$ to be the maximal index such that \eqref{eqn:ineq} is satisfied. We have 
\[
r_j(q_j, h_j(q_j)) \in r_jF_j = F_j + \cdots + F_j
\]
where the sum has $r_j$ terms. Moreover, $F_j$ is 
isomorphic to its 
projection $C_j := \cell(x,h_j)$ on $\R^n$, and the set of extremal points of $C_j$ is included in the finite set $\cell(x, \omega_j)$, where $\omega_j$ is the coefficient map of the tropical polynomial $f_j$ (see \Cref{sec:caem} and \Cref{rmk:remk}\ref{rmk:remk-b}).
 Moreover, by assumption $j$ and $x$ satisfy \eqref{eqn:ineq}, 
hence, in case \eqref{eqn:ineq1}, one must have  $\cell(x, \omega_j)\subset \calA^-_j$, because the maximum in the expression
\[
\max_{\alpha \in \calA_j} f_{j, \alpha} +\<x, \alpha> 
\]
cannot be attained by a monomial of $f^+_j$, and likewise in case \eqref{eqn:ineq2}, one must have 
 $\cell(x, \omega_j)\subset \calA^+_j$.
 Therefore, this means by definition of $r_j$ that the cell $C_j$, as well as the facet $F_j$ have dimension at most $r_j - 1$, and so does the facet $r_j F_j = F_j + \cdots + F_j$. Thus, applying \Cref{cor:shapley-folkman} (of the Shapley-Folkman lemma), there exists $(q'_j,h_j(q'_j)) \in F_j$ and $a_j$ an extremal point of $C_j$ such that $r_j (q_j,h_j(q_j)) = (r_j - 1)(q'_j,h_j(q'_j)) + (a_j,h_j(a_j))$. Moreover, since $a_j$ is an extremal point of $C_j$, we know that in case \eqref{eqn:ineq1}, one has $a_j \in \calA^-_j$, while in case \eqref{eqn:ineq2}, one has $a_j \in \calA^+_j$.

We then define the \textit{row content} of $p$ in this context to be equal to a couple $(j,a_j)$ satisfying the above properties. Note that the element  $a_j \in \calA_j$ satisfying the above condition need not be unique, but we will see in the proof that this does not matter. 
Since $a_j$ is an extremal point of $\cell(x,h_j)$, we have that $(a_j, h_j(a_j))$ is a vertex of $\hypo(h_j)$, which implies that $h_j(a_j)=f_{j,a_j}$, by \Cref{obs:ess}.

Now for any subset $\calE'$ of $\Z^n$ containing $\calE$, we can construct the matrices $\mac_{\mathcal{EE}'}^\pm = (m_{pp'}^\pm)_{(p,p') \in \calE \times \calE'}$ similarly to \Cref{sec:caem}, by setting the row $p$ of $\mac_{\mathcal{EE}'}^\pm$ to be the row $(j,p-a_j)$ of $\mac_{\calE'}^\pm$, if $(j,a_j)$ is the row content of $p$.

\subsubsection{The proof of \Cref{thm:positivstellensatz}}

In order to prove \Cref{thm:positivstellensatz}, we will make use of the following lemma.

\begin{lem}\label{lem:ndiag2}
Consider $h_1, \ldots, h_k$ and $\widetilde{h}$ as defined in \Cref{sec:caem}, and let $p \in \calE$ and let $(j,a_j)$ be its row content. Then for all $p' \in \calE$ and $a_j' \in \Z^n$ such that $p' = p - a_j + a'_j$, we have
\begin{equation}
\widetilde{h}(p' - \delta) \overset{\substack{\displaystyle (\dagger) \\ \mathstrut}}{\geq} \widetilde{h}(p - \delta) - h_j(a_j) + h_j(a'_j) \overset{\substack{\displaystyle (\ddagger) \\ \mathstrut}}{\geq} \widetilde{h}(p - \delta) - f^-_{j,a_j} + f^+_{j,a'_j} \enspace ,
\label{eqn:ndiag2}
\end{equation}
and moreover at least one of the two inequalities is strict.
\end{lem}

\begin{proof}
By setting
\[
\widehat{h}_j = \underbrace{h_1 \supco \cdots \supco h_1}_{r_1 \textrm{ terms}} \supco \cdots \supco \underbrace{h_j \supco \cdots \supco h_j}_{r_j - 1 \textrm{ terms}} \supco \cdots \supco \underbrace{h_k \supco \cdots \supco h_k}_{r_k \textrm{ terms}} \enspace ,
\]
so that $\widetilde{h} = h_j \supco \widehat{h}_j$, the exact same reasoning as in the proof of \Cref{lem:ddiag} shows that
\[
\widetilde{h}(p' - \delta) \geq h_j(a'_j) + \widehat{h}_j(p - \delta - a_j) \enspace ,
\]
and that if $p'=p$, which entails that $a'_j=a_j$, then this inequality is an equality. This implies inequality $(\dagger)$.

Inequality $(\ddagger)$ simply follows from the fact that $f^-_{j,a_j} = h_j(a_j)$ and $f^+_{j,a'_j} \leq h_j(a'_j)$.

We now show that either $(\dagger)$ or $(\ddagger)$ is strict. If $p' = p$, then inequality $(\ddagger)$ reduces to $f^-_{j,a_j} \geq f^+_{j,a_j}$ which is known to be strict. Now assume that $p' \neq p$, and suppose that the equality is achieved in $(\dagger)$. Then this implies that
\begin{equation}\label{eqn:diag2}
(p' - \delta, \widetilde{h}(p' - \delta)) = (a'_j, h_j(a'_j)) + (p - \delta - a_j, \widehat{h}_j(p - \delta - a_j)) \enspace .
\end{equation}
Now consider $x, x' \in \R^n$ such that $F = \face(x,\widetilde{h})$ and $F' = \face(x',\widetilde{h})$ are the facet in the interior of which $(p - \delta, \widetilde{h}(p - \delta))$ and $(p' - \delta, \widetilde{h}(p' - \delta))$ respectively lie. Then from \Cref{cor:faces}, we have
\[
(a'_j, h_j(a'_j)) \in \face(x', h_j) \quad \textrm{and} \quad (p - \delta - a_j, \widehat{h}_j(p - \delta - a_j)) \in \face(x', \widehat{h}_j) \enspace .
\]
However, we also know from equality (\ref{eqn:diag2}) and \Cref{cor:faces} that
\[
(p - \delta - a_j, \widehat{h}_j(p - \delta - a_j)) \in \face(x, \widehat{h}_j) \enspace .
\]
Moreover, by the Shapley-Folkman lemma (\Cref{lem:shapley-folkman}), we have that
\[
r_j \face(x,h_j) = \mathcal{S} + (r_j - 1) \face(x,h_j)
\]
with $\mathcal{S} := \{ (\alpha, h_j(\alpha)) \in \Z^n \times \R : (\alpha, h_j(\alpha)) \textrm{ is a vertex of } \face(x,h_j) \}$ which is in particular a finite set, and thus
\[
\face(x, \widetilde{h}) = \mathcal{S} + \face(x, \widehat{h}_j) \enspace .
\]
Since $(p - \delta, h(p - \delta))$ is in the relative interior of $\face(x, \widetilde{h})$, and since $\delta$ was taken generic, then it follows that $(p - \delta - a_j, \widehat{h}_j(p - \delta - a_j))$ is in the relative interior of $\face(x, \widehat{h}_j)$, and that $\face(x, \widehat{h}_j)$ is a facet of $\hypo(\widehat{h}_j)$. Therefore, since $(p- \delta -a_j, \widehat{h}_j(p - \delta - a_j))$ is in both $\face(x, \widehat{h}_j)$ and $\face(x', \widehat{h}_j)$ and since it is in particular in the relative interior of the first face, we deduce that
\[
\face(x, \widehat{h}_j) \subseteq \face(x', \widehat{h}_j) \ \enspace .
\]
Since by \Cref{rmk:remk} (a), $\face(x', \widehat{h}_j)$ is a proper face of $\hypo(\widehat{h}_j)$, then it implies that it is also a facet of $\hypo(\widehat{h}_j)$, and thus that
\[
\face(x, \widehat{h}_j) = \face(x', \widehat{h}_j) \enspace .
\]
Hence, it follows from \Cref{rmk:remk} (b) that $x = x'$, and therefore $(a'_j, h_j(a'_j)) \in \face(x, h_j)$.
Therefore, we have
\[
h_j(a'_j) + \<x, a'_j> = \max_{q_j \in Q_j} h_j(q_j) + \<x, q_j> = \max_{\alpha \in \calA_j} f_{j,\alpha} + \<x, \alpha> = f(x) \enspace ,
\]
and since $f^-(x) > f^+(x)$, this implies that $a'_j \in A^-_j$, and thus
\[
f^-_{j,a'_j} + \<x, a'_j> = f^-(x) > f^+(x) \geq f^+_{j,a'_j} + \<x, a'_j> \enspace ,
\]
hence $f^+_{j,a'_j} < f^-_{j,a'_j} = h_j(a'_j)$, thus inequality $(\ddagger)$ is strict.
\end{proof}

\begin{proof}[Proof of \Cref{thm:positivstellensatz}.]
We keep the notation of \Cref{sec-row-ineq}. We first look at the case of systems in the form $f^+(x) \geq f^-(x)$, as the result for the general case will follow from this case.
For any subset $\calE'$ from $\Z^n$ containing $\calE$, we can construct the submatrices $\mac_{\mathcal{EE}'}^+ = (m_{pp'}^+)_{(p,p') \in \calE \times \calE'}$ and $\mac_{\mathcal{EE}'}^- = (m_{pp'}^-)_{(p,p') \in \calE \times \calE'}$ of respectively $\mac_{\calE'}^+$ and $\mac_{\calE'}^-$ similarly to \Cref{sec:caem}, by setting $m_{pp'}^+$ to be the coefficient in $p'$ of $X^{p-a_j}f_j^+$ where $(j,a_j)$ is the row content of $p$, \textit{i.e.}
\[
m_{pp'}^\pm = \left\{\begin{array}{ll}
f_{j,a'_j}^\pm & \textrm{if} \quad a'_j \in \calA_j^\pm\\
\zero & \textrm{otherwise,}
\end{array} \right.
\]
where $a'_j \in \Z^n$ is such that $p'=p-a_j+a'_j$, as well as the matrices $\widetilde{\mac}_{\mathcal{EE}}^{\pm} = (\widetilde{m}_{pp'}^{\pm})_{(p,p') \in \calE \times \calE}$ defined by
\[
\widetilde{m}_{pp'}^{\pm} = m_{pp'}^{\pm} - \widetilde{h}(p'-\delta)
\]
for all $p,p' \in \calE$. We now show that the matrix $\widetilde{\mac}_{\mathcal{EE}}^+$ is diagonally dominated by $\widetilde{\mac}_{\mathcal{EE}}^-$, which will lead to the desired result. Let $p,p' \in \calE$. Then:
\begin{itemize}[label=$\diamond$]
\item if $p' \neq p$, then $a'_j \neq a_j$ and thus the inequality $\widetilde{m}^-_{pp} > \widetilde{m}^+_{pp'}$ is equivalent to the inequality 
\begin{equation}\label{eq:ddiag}
\widetilde{h}(p'-\delta) > \widetilde{h}(p-\delta) - f^-_{j,a_j} + f^+_{j,a'_j} \enspace ,
\end{equation}
which follows directly from \Cref{lem:ndiag2};
\item if $p'=p$, then $a'_j = a_j$ and the inequality $\widetilde{m}^-_{pp} > \widetilde{m}^+_{pp'}$ is equivalent to the inequality
\[
f^-_{j,a_j} > f^+_{j,a_j},
\]
which is satisfied since $f^-_{j,a_j} + \<x, a_j> = f^-(x) > f^+(x) \geq f^+_{j,a_j} + \<x, a_j>$.
\end{itemize}
Thus, the matrix $\widetilde{\mac}_{\mathcal{EE}}^+$ is diagonally dominated by $\widetilde{\mac}_{\mathcal{EE}}^-$, and therefore thanks to \Cref{lem:ddom}, this shows that the only solution to the inequation
\[
\widetilde{\mac}_{\mathcal{EE}}^+ \tdot \widetilde{z} \geq \widetilde{\mac}_{\mathcal{EE}}^- \tdot \widetilde{z}
\]
of unknown $\widetilde{z} \in \T^{\calE}$ is $\widetilde{z} = \zero$.

Finally, using \Cref{lem-chang-diag} and \Cref{lem-sousmatrice},
 we find that this implies that there cannot exist a  vector $y \in \R^{\calE'}$ such that $\mac_{\calE'}^+ \tdot y \geq \mac_{\calE'}^- \tdot y$.\\

Now for the general case, assume that there is no solution in $\R^n$ to the system $f^+(x) \rhd f^-(x)$.
Then, for all $x \in \R^n$, there exists $1 \leq j \leq k$ such that
\[
f^+_j(x) \nrhd_j f^-_j(x) \enspace .
\]
We reiterate the construction of the matrices $\mac^\pm_{\calE \calE}$ as
above, as well as the renormalized matrices $\widetilde{\mac}^\pm_{\calE \calE}$, with the only difference that for $p \in \calE$ with row content $(j, a_j)$, 
one might have either $a_j \in \calA_j^+$ or $a_j \in \calA_j^-$. In fact, applying \eqref{eq:ddiag} gives us this time that
\[
\widetilde{m}^-_{pp} > \max_{p' \in \calE} \widetilde{m}^+_{pp'} \iff a_j \in \calA_j^- \quad \textrm{and} \quad \widetilde{m}^+_{pp} > \max_{p' \in \calE}  \widetilde{m}^-_{pp'} \iff a_j \in \calA_j^+ \enspace .
\]
In particular, this shows that the matrices $\widetilde{\mac}^+_{\calE \calE}$ and $\widetilde{\mac}^-_{\calE \calE}$ satisfy the condition \eqref{eq:perddom} of \Cref{cor:ddom}. It thus follows from \Cref{lem:ndiag2} that the matrices $\widetilde{\mac}^+_{\calE \calE}$ and $\widetilde{\mac}^-_{\calE \calE}$ satisfy the condition \eqref{eq:perddom} of \Cref{cor:ddom}, hence the only solution to the system $\widetilde{\mac}_{\mathcal{EE}}^+ \tdot \widetilde{z} \rhd \widetilde{\mac}_{\mathcal{EE}}^- \tdot \widetilde{z}$ of unknown $\widetilde{z} \in \T^{\calE}$ is $\widetilde{z} = \zero$, and using \Cref{lem-chang-diag} and \Cref{lem-sousmatrice}, we obtain there is no solution $y \in \R^{\calE'}$ to the equation $\mac_{\calE'}^+ \tdot y \rhd \mac_{\calE'}^- \tdot y$.
\end{proof}

\begin{rmk}
Note that in the previous construction, the matrices $\mac_{\calE \calE}^\pm$ need not necessarily be submatrices of the matrices $\mac_\calE^\pm$, as the construction of the row content in this case could allow for one row of the Macaulay matrix to appear multiple times in $\mac_{\calE \calE}^\pm$. This however does not have any impact on the outcome of the proof.
\end{rmk}

\subsection{The case of hybrid systems}

In fact, \Cref{thm:positivstellensatz} can be generalized even further so as to also include relations of the form $f_i(x) \bal \zero$. More precisely, let $f_1^\pm, \ldots, f_k^\pm$ be a collection of pairs of tropical polynomials and let $f_{k+1}, \ldots, f_\ell$ be a second collection of tropical polynomials. Keeping the notation of \Cref{sec:pos}, the system we consider is the following
\begin{equation}\label{eq:hybrid}\tag{$\mathscr{S}_\textrm{pol}$}
\left\{\begin{array}{ll}
f^+_i(x) \rhd_i f^-_i(x) & \textrm{for } 1 \leq i \leq k\\
f_i(x) \bal \zero & \textrm{for } k+1 \leq i \leq \ell
\end{array}\right. \enspace ,
\end{equation}
with $\rhd_i \in \{ \geq, =, > \}$ for all $1 \leq i \leq k$. Recall also that
\[
r_i = \left\{\begin{array}{ll}
\dim(\aff(\calA^-_i)) + 1 & \textrm{if } \rhd_i \in \{ \geq, > \}\\
\max\left(\dim(\aff(\calA^-_i)), \dim(\aff(\calA^+_i))\right) + 1 & \textrm{if } \rhd_i \in \{ = \} \enspace .
\end{array}\right.
\]
We define Canny-Emiris subsets $\calE$ of $\Z^n$ associated to system \eqref{eq:hybrid} to be sets of the form
\[
\calE :=  \left( \widetilde{Q} + \delta \right) \cap \Z^n \quad \textrm{with} \quad \widetilde{Q} = r_1Q_1 + \cdots + r_kQ_k + Q_{k+1} + \cdots + Q_\ell \enspace ,
\]
where $\delta$ is a generic vector in $V + \Z^n$, with $V$ the direction of the affine hull of $\widetilde{Q}$. Then we have the following result.

\begin{thm}[Sparse tropical Positivstellensatz for hybrid systems]
\label{thm:positivstellensatz-hybrid}
The system \eqref{eq:hybrid} has a solution $x \in \R^n$ if and only if the linear tropical system
\begin{equation}\label{eq:hybrid-linearized}\tag{$\mathscr{S}_\textrm{lin}$}
\left\{\begin{array}{l}
\mac_{\calE'}^+ \tdot y \rhd \mac_{\calE'}^- \tdot y\\
\mac_{\calE'} \tdot y \bal \zero
\end{array}\right. \enspace 
\end{equation}
has a solution $y \in \R^{\calE'}$, where $\mac^+, \mac^-$ is the pair of Macaulay matrices associated to system $f_i^+(x) \rhd_i f_i^-(x)$ for $1 \leq i \leq k$, $\mac$ is the Macaulay matrix associated to system $f_i(x) \bal \zero$ for $k + 1 \leq i \leq \ell$, and $\calE'$ is any subset of $\Z^n$ containing a nonempty Canny-Emiris subset $\calE$ of $\Z^n$ associated to the system \eqref{eq:hybrid}.
\end{thm}

\begin{proof}
This proof is mainly combination of the proofs of \Cref{thm:nullstellensatz,thm:positivstellensatz}, so we will just give the outline of the proof and skip the details to avoid repetition. Moreover, as previously, the existence of a root $x \in \R^n$ to the polynomial easily implies the existence of an element $y \in \R^{\calE'}$ solution of the system \eqref{eq:hybrid-linearized}, given by the Veronese embedding \eqref{eq:veronese}.

Assume that system \eqref{eq:hybrid} does not have a root in $\R^n$. Then this means that for all $x \in \R^n$, either $f_j(x) \nbal \zero$ for some $k+1 \leq j \leq \ell$, or $f_j^+(x) \nrhd_j f_j^-(x)$ for some $1 \leq i \leq k$ (or possibly both cases can happen at the same time). This allows us to construct a notion of row content for this case in the following way.

First, we define the maps $h_1, \ldots, h_\ell$ as in the proofs of \Cref{thm:nullstellensatz,thm:positivstellensatz}, and set
\[
\widetilde{h} := h_1^{\supco r_1} \supco \cdots \supco h_k^{\supco r_k} \supco h_{k+1} \supco \cdots \supco h_\ell \enspace .
\]
Then for all $p \in (\widetilde{Q} + \delta) \cap \Z^n$, the point $(p-\delta, \widetilde{h}(p-\delta))$ lies in the relative interior of a facet $\widetilde{F}$ of $\hypo(\widetilde{h})$, and this facet can be written as
\[
\widetilde{F} = r_1F_1 + \cdots + r_kF_k + F_{k+1} + \cdots + F_\ell \enspace ,
\]
where for all $1 \leq j \leq \ell$, $F_j$ is a facet of $\hypo(h_j)$, and there exists an element $x$ of $V$ such that $\widetilde{F} = \face(\widetilde{h},x)$ and for all $1 \leq j \leq \ell$, $F_j = \face(h_j,x)$. Then we have the following two possible cases
\begin{enumerate}[label=\it (\roman*)]
\item if there exists $k+1 \leq j \leq \ell$ such that $f_j(x) \nbal \zero$, then $F_j$ is a singleton $\{ a_j \}$, and we set the row content of $p$ to be the pair $(j, a_j)$, where such an index $j$ is taken to be maximal;
\item if $f_j(x) \bal \zero$ for all $k+1 \leq j \leq \ell$, then since $x$ is not a root of $\eqref{eq:hybrid}$, this implies that there exists $1 \leq j \leq k$ such that $f_j^+(x) \nrhd_j f_j^-(x)$, and this time we can construct the row content of $x$ using the Shapley-Folkman lemma as in the proof of \Cref{thm:positivstellensatz}.
\end{enumerate}
Now for $p \in \calE$ with row content $(j, a_j)$, if $1 \leq j \leq k$, then we apply \Cref{lem:ndiag2} to obtain that for all $p' \in \calE$ and $a'_j \in \Z^n$ such that $p' = p - a_j + a'_j$,
\[
\widetilde{h}(p' - \delta) > \widetilde{h}(p - \delta) - f_{j,a_j}^- + f_{j,a'_j}^+ \enspace ,
\]
and if $k+1 \leq j \leq \ell$, then we apply \Cref{lem:ddiag} to obtain that for all $p' \in \calE$ and $a'_j \in \Z^n$ such that $p' = p - a_j + a'_j$,
\[
\widetilde{h}(p' - \delta) > \widetilde{h}(p - \delta) - f_{j,a_j} + f_{j,a'_j} \enspace ,
\]
with equality if and only if $p' = p$. We denote $\calE_1$ the set of $p \in \calE$ satisfying the first condition, and $\calE_2$ the set of $p \in \calE$ satisfying the second condition, such that $\calE = \calE_1 \sqcup \calE_2$.

Thus, we can construct as previously matrices $\mac^+_{\calE_1 \calE}$, $\mac^-_{\calE_1 \calE}$ and $\mac_{\calE_2 \calE}$ such that $\mac^-_{\calE_1 \calE}$ and $\mac^+_{\calE_1 \calE}$ satisfy the condition \eqref{eq:perddom} of \Cref{cor:ddom}, and $\mac_{\calE_2 \calE}$ is diagonally dominant in the tropical sense.

From this point, it can easily be shown using a combination of \Cref{lem:ddns} and \ref{lem:ddom} that the only solution $z \in \T^{\calE}$ to the system
\[
\left\{\begin{array}{l}
\mac_{\calE_1 \calE}^+ \tdot z \rhd \mac_{\calE_1 \calE}^- \tdot z\\
\mac_{\calE_2 \calE} \tdot z \bal \zero \enspace ,
\end{array}\right.
\]
is ${z} = \zero$,
and thus it follows that there is no solution $y \in \R^{\calE'}$ to system \eqref{eq:hybrid-linearized}.
\end{proof}

\section{Concluding remarks}
We have developed a new approach to the tropical Nullstellensatz, initially established by Grigoriev and Podolskii~\cite{GP18}. Our proof relies  on a fundamental tool of sparse elimination theory, the Canny-Emiris set~\cite{CE93,Emi05,Stu94}. This leads to improved bounds on the truncation degree. We also obtained tropical Positivstellens\"atze, allowing one to deal both with strict and weak inequalities. For the simplest systems of tropical relations --- requiring maxima to be achieved twice --- this approach leads to a tight truncation degree, as discussed in~\Cref{rmk:tight}. For more general systems, our proof method is based on a new definition of the row content, exploiting a discrete geometry result (the Shapley-Folkman lemma), and it leads to a `dilation' of the Canny-Emiris set by a factor up to $n+1$. Whether this dilation could be avoided is left as an open problem.

We provided in~\cite{bereau2023}, following~\cite{ABG23a}, a first implementation of the present method. So far, we only considered the feasibility problem. We note however that by a dichotomy approach, one can find an explicit (exact) solution of a feasible system with rational coefficients, using the present tools. We also observe that in some special cases, unfeasibility certificates may be obtained by considering smaller linearizations. This suggests the possibility to develop accelerated algorithms, for instance, constructing dynamically linearizations with increased rows and columns sets. Such refinements will be developed elsewhere.

\bibliographystyle{alpha}
\bibliography{tropical-nullstellensatz-refs}

\end{document}